\documentclass[11pt, reqno]{amsart}
\usepackage{amsfonts,latexsym,enumerate}
\usepackage{amsmath}
\usepackage{mathtools}
\usepackage{amscd}
\usepackage{float,amsmath,amssymb,mathrsfs,bm,multirow,graphics}
\usepackage[dvips]{graphicx}
\usepackage[percent]{overpic}
\usepackage[foot]{amsaddr}
\usepackage[numbers,sort&compress]{natbib}
\usepackage{geometry}
\usepackage{hyperref}
\usepackage[]{enumerate}		
\usepackage[percent]{overpic}
\usepackage{subcaption}
\usepackage{appendix}

\newcommand{\R}{{\mathbb R}}
\newcommand{\N}{{\mathbb N}}
\newcommand{\C}{{\mathbb C}}
\newcommand{\Z}{{\mathbb Z}}
\newcommand{\proofend}{\hfill$\square$\medskip\smallskip}
\newcommand{\bigO}{\mathcal{O}}
\newcommand{\dx}{\mathrm{d}}
\newcommand{\e}{\mathrm{e}}
\newcommand{\eps}{\varepsilon}
\newcommand{\overbar}[1]{\mkern 1.5mu\overline{\mkern-1.5mu#1\mkern-1.5mu}\mkern 1.5mu}
\newcommand{\op}{\mathrm{od}}
\newcommand{\res}{\text{\upshape Res}}
\newcommand{\re}{\text{\upshape Re\,}}
\newcommand{\im}{\text{\upshape Im\,}}

\newtheorem{theorem}{Theorem}[section]
\newtheorem*{theorem*}{Theorem}
\newtheorem*{lemma*}{Lemma}
\newtheorem{proposition}[theorem]{Proposition}
\newtheorem{lemma}[theorem]{Lemma}
\newtheorem{corollary}[theorem]{Corollary}
\theoremstyle{definition}

\newtheorem*{definition*}{Definition}
\newtheorem{remark}[theorem]{Remark}
\newtheorem*{remark*}{Remark}

\numberwithin{equation}{section}

\title[DNLS with step-like oscillatory initial data]
{The defocusing nonlinear Schr\"odinger equation with step-like oscillatory initial data}

\author[S.\@ Fromm]{Samuel Fromm$^1$}
\author[J.\@ Lenells]{Jonatan Lenells$^1$}  
\author[R.\@ Quirchmayr]{Ronald Quirchmayr$^2$}

\email{\href{mailto: samfro@kth.se}{samfro@kth.se}}

\address{$^1$Department of Mathematics, KTH Royal Institute of Technology,
	100~44~Stockholm, Sweden}
\email{\href{mailto: jlenells@kth.se}{jlenells@kth.se}}

\address{$^2$Faculty of Mathematics, University of Vienna,  
	1090 Vienna, Austria}
\email{\href{mailto: ronald.quirchmayr@univie.ac.at}{ronald.quirchmayr@univie.ac.at}}

\begin{document}
\UseRawInputEncoding

\begin{abstract}
\noindent
We study the Cauchy problem for the defocusing nonlinear Schr\"odinger (NLS) equation under the assumption that the solution vanishes as $x \to + \infty$ and approaches an oscillatory plane wave as $x \to -\infty$. We first develop an inverse scattering transform formalism for solutions satisfying such step-like boundary conditions. Using this formalism, we prove that there exists a global solution of the corresponding Cauchy problem and establish a representation for this solution in terms of the solution of a Riemann--Hilbert problem. 
By performing a steepest descent analysis of this Riemann--Hilbert problem, we identify three asymptotic sectors in the half-plane $t \geq 0$ of the $xt$-plane and derive asymptotic formulas for the solution in each of these sectors. Finally, by restricting the constructed solutions to the half-line $x \geq 0$, we find a class of solutions with asymptotically time-periodic boundary values previously sought for in the context of the NLS half-line problem.
\end{abstract}

\maketitle
\thispagestyle{empty}




\section{Introduction}
\noindent
We study the Cauchy problem for the defocusing nonlinear Schr\"odinger (NLS) equation 
\begin{subequations}\label{dNLS_IVP}
\begin{alignat}{2}
& i u_t+u_{xx}-2|u|^2u=0,&\qquad&x\in\R,\quad t> 0, \label{dNLS} \\
&u(x,0)=u_0(x),&&x\in\R, \label{IC}
\end{alignat}
\end{subequations}
under the assumption that the solution $u(x,t)$ vanishes as $x \to + \infty$ and approaches a plane wave as $x \to -\infty$, i.e.,
\begin{equation} \label{boundaryconditions}
u(x,t)\sim 
\begin{cases}
\alpha \e^{2 i\beta x+ i \omega t}, & x\to -\infty, \\
0, & x\to +\infty,
\end{cases}
\end{equation}
where $\alpha>0$ and $\beta\in\R$ are two parameters. The constant $\omega \coloneqq -4\beta^2-2\alpha^2<0$ in \eqref{boundaryconditions} is determined by the requirement that $\alpha \e^{2 i\beta x+ i \omega t}$ should be a solution of \eqref{dNLS}, which is necessary for consistency.
The phase invariance of \eqref{dNLS} implies that there is no loss of generality in assuming that $\alpha$ is positive. 
In order for the boundary conditions \eqref{boundaryconditions} to be consistent with \eqref{dNLS_IVP}, the initial data must satisfy
\begin{equation} \label{u0limits}
u_0(x)\sim 
\begin{cases}
\alpha \e^{2 i \beta x},&x\to -\infty,\\
0,&x\to +\infty.
\end{cases}
\end{equation}
The exact rates of convergence in \eqref{boundaryconditions} and \eqref{u0limits} will be specified in the main theorems. 

The NLS equation with non-vanishing boundary conditions at infinity has received plenty of attention in the literature, see e.g. \cite{BM2019, B2018, BLM2021, BLS2021, BP1982, BM2017, BKS2011, BV2007, FT1986, GK2012, DPMV2013, IU1986, IU1991, J2015, KI1978, ZS1973}. 
In this paper, we use Riemann--Hilbert (RH) methods to study the Cauchy problem \eqref{dNLS_IVP} with the step-like boundary conditions \eqref{boundaryconditions}. In particular, we prove that there exists a global solution of the problem \eqref{dNLS_IVP}--\eqref{boundaryconditions} and establish a representation for this solution in terms of the solution of a RH problem. The formulation of the RH problem involves a reflection coefficient which is defined in terms of the initial data. We also obtain detailed asymptotic formulas for the long-time behavior of the solution by performing a Deift--Zhou steepest descent analysis \cite{DZ1993} of the RH problem. It turns out that there are three main asymptotic sectors in the half-plane $t \geq 0$ of the $xt$-plane separated by the two lines $x/t = 4\beta - 2\alpha$ and $x/t = 4\beta + 4\alpha$; we refer to these three sectors as the left, middle, and right sectors, see Figure \ref{sectorsfig}. In the left sector (i.e., the sector adjacent to the negative $x$-axis), the leading term in the asymptotics is given by the plane wave $\alpha \e^{2 i\beta x + i\omega t}$ multiplied by a slowly varying factor which tends to $1$ as $x/t \to -\infty$, in consistency with the boundary conditions \eqref{boundaryconditions}. In the right sector (i.e., the sector adjacent to the positive $x$-axis), the leading term is of order $t^{-1/2}$ and has a coefficient which vanishes as $x/t \to +\infty$, again in consistency with \eqref{boundaryconditions}. 
In the middle sector, which lies between the left and right sectors, the asymptotics has a different character. We will compute both the leading and the subleading terms in this sector. The leading term is generally of order $\bigO(1)$ but it vanishes as $x/t$ approaches $4\beta + 4\alpha$, thus matching the leading asymptotics in the right sector. As $x/t$ approaches $4\beta - 2\alpha$, it tends, as expected, to the leading term in the left sector. The subleading term in the middle sector is of order $\bigO(t^{-1})$. 

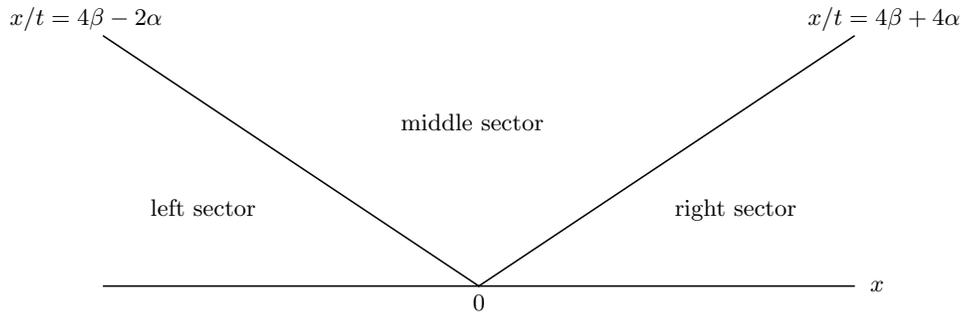
\begin{figure}[h!] \centering
	\vspace{1em}
	\begin{overpic}[width=.63\textwidth]{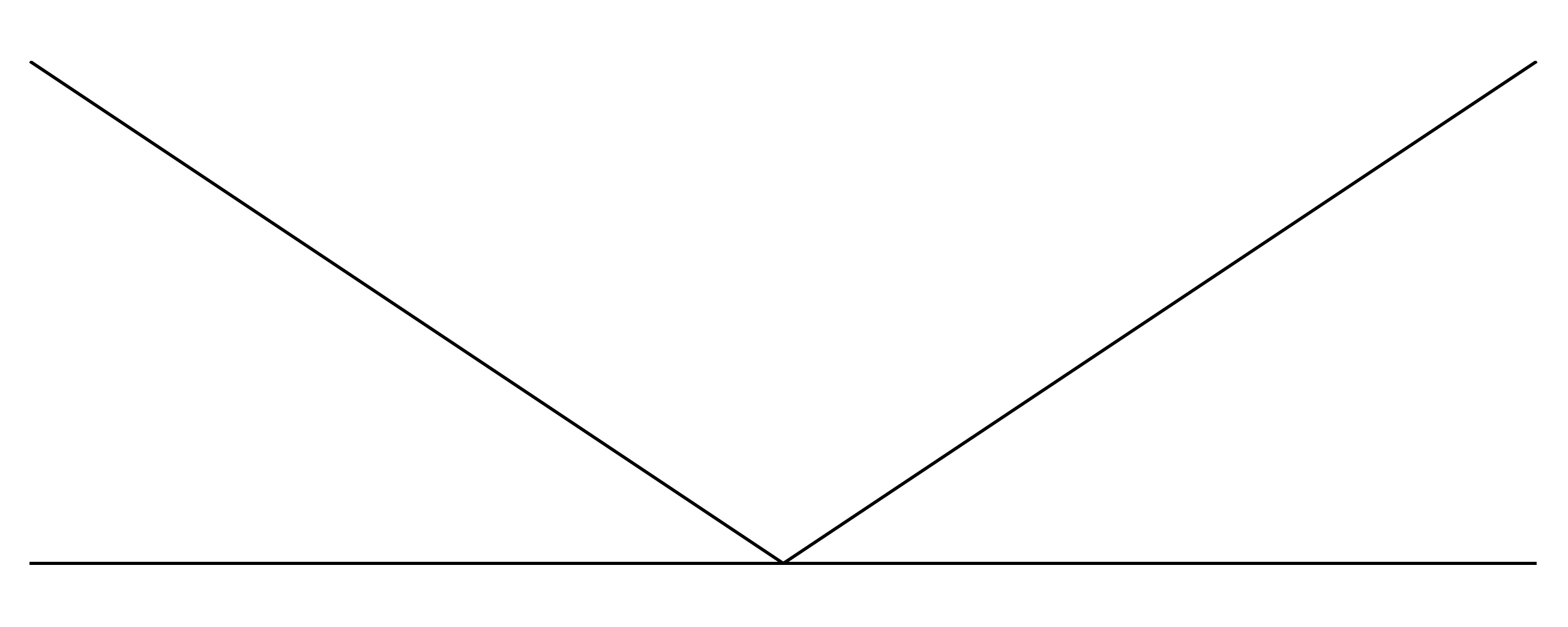}
		\put(49.2,1){\footnotesize{$0$}}     
		\put(100,3.4){\footnotesize{$x$}}  
		\put(8,13){\footnotesize{left sector}} 
		\put(40,24){\footnotesize{middle sector}} 
		\put(75,13){\footnotesize{right sector}}
		\put(-10,37.5){\footnotesize{$x/t =4\beta - 2\alpha$}}
		\put(92,37.5){\footnotesize{$x/t =4\beta+4\alpha$}}
	\end{overpic}
	\vspace{0em}
	\caption{The left, middle, and right asymptotic sectors in the $xt$-plane.} 
	\label{sectorsfig}
\end{figure}

The main results of the paper are presented in the form of six theorems, which are all stated in Section \ref{mainresultssec}.

\begin{enumerate}[-]
\item Theorem \ref{thm_direct} treats the direct problem, i.e., the construction of the scattering data from the given initial data. 

\item Theorem \ref{thm_inverse} proves global existence for the Cauchy problem \eqref{dNLS_IVP}--\eqref{boundaryconditions} and establishes a representation for the solution in terms of the solution of a RH problem. This is achieved by combining the result of Theorem \ref{thm_direct} with a solution of the inverse problem. 

\item Theorem \ref{thm_sec_M}, Theorem \ref{thm_sec_L}, and Theorem \ref{thm_sec_R} present asymptotic formulas for the solution of \eqref{dNLS_IVP}--\eqref{u0limits} in the middle, left, and right asymptotic sectors, respectively.

\item Theorem \ref{thm_ibvp} provides a class of solutions of the NLS on the half-line $x \geq 0$ with asymptotically time-periodic boundary values that was sought for in \cite{Len15}.
\end{enumerate}

\subsection{Relation to earlier work}
First results on the implementation of an Inverse Scattering Transform (IST) to the defocusing NLS with nonzero boundary conditions were obtained in \cite{ZS1973} where solutions satisfying $|u(x,t)| \to c$ as $x \to \pm \infty$ for some constant $c > 0$ were considered. The asymmetric case where $u(x,t)$ approaches plane waves of different amplitudes as $x \to + \infty$ and $x \to -\infty$ was considered in \cite{BP1982}. 
We refer to \cite{BM2019, BLM2021, BLS2021, FT1986, GK2012, DPMV2013, IU1991, J2015, KI1978} and the references therein for further works on the IST for NLS. 
The implementation of the IST presented here can be compared to the implementation for the focusing NLS in \cite{BLS2021}.

Studies of the long-time asymptotics for NLS with nonzero boundary conditions include \cite{B2018, BLS2021, BM2017, BKS2011, BV2007, IU1986, J2015, BTVZ2007}. Most of these studies concern the focusing NLS. Of particular interest to us is the work \cite{J2015}, where Jenkins derived detailed formulas for the long-time asymptotics for the defocusing NLS equation with step-like initial data of the form
\begin{align}\label{u0Jenkins}
u_0(x) = \begin{cases} 1, & x < 0, \\
A \e^{-2i\mu x}, & x > 0,
\end{cases}
\end{align}
for any choice of the constants $A > 0$ and $\mu \in \R$. 
If $u(x,t)$ solves \eqref{dNLS}, then the function $\tilde{u}(x,t)$ defined by
\begin{equation*}
	\tilde{u}(x,t)\coloneqq a u(a(x+4bt),a^2t) \e^{-2i b(x+2bt)},
\end{equation*}
also solves \eqref{dNLS} for any $a>0$ and $b \in \R$. Moreover, if $u_0$ satisfies \eqref{u0Jenkins}, then $\tilde{u}_0(x) = \tilde{u}(x,0)$ satisfies
\begin{align}\label{utilde0}
\tilde{u}_0(x) = \begin{cases}
A_1 \e^{-2i B_1 x},  &x < 0,\\
A_2 \e^{-2iB_2 x},  & x > 0,
\end{cases}
\end{align}
where $A_1 = a$, $B_1 = b$, $A_2 = aA$, and $B_2 = \mu a + b$.
Hence, whenever the initial data satisfy \eqref{utilde0} with $A_1 > 0, A_2 > 0$, and $B_1, B_2 \in \R$, the asymptotics can be deduced from \cite{J2015}. However, the nonzero boundary conditions \eqref{boundaryconditions} considered in this paper (for which $A_2 = 0$) fall outside the scope of \cite{J2015}. It is therefore not surprising that the qualitative structure of the asymptotic picture we find is different from the one obtained in \cite{J2015}.

A main difference between \cite{J2015} and the present paper is that whereas \cite{J2015} assumes that the initial data is {\it equal} to $1$ and $A \e^{-2i\mu x}$ for $x < 0$ and $x > 0$, respectively, we only impose the boundary conditions \eqref{boundaryconditions} {\it asymptotically} as $x \to \pm \infty$. In fact, one of the main motivations behind this work was to understand how asymptotically imposed nonzero boundary conditions can be handled for an integrable PDE. The fact that the boundary conditions are only asymptotically imposed has rather far-reaching implications for the proof, because it means that instead of having an exact expression for the reflection coefficient $r(k)$ at our disposal, we need to deduce the pertinent properties of $r(k)$ from its definition in terms of solutions of Volterra equations and construct analytic approximations (or dbar extensions) for the asymptotic analysis. The construction of analytic approximations is challenging especially near the branch cut associated with the nonzero boundary conditions. We overcome this challenge by generalizing a construction in \cite{Fromm19}. 

We will show that there are no solitonic structures present in the long-time asymptotics for \eqref{dNLS} in the case of the boundary conditions \eqref{boundaryconditions}. In particular, the discrete spectrum is always empty. In the case of nonzero backgrounds at both $+\infty$ and $-\infty$ of the same amplitudes, the discrete spectrum is generally non-empty, corresponding to the fact that the defocusing NLS supports dark solitons \cite{ZS1973} (see also \cite{DPMV2013}). On the other hand, the discrete spectrum is always empty in the case of zero boundary conditions at both $\pm \infty$, corresponding to the fact that the equation admits no bright solitons, see \cite{A2011}. It is interesting that vanishing boundary conditions at $+\infty$ alone imply the absence of asymptotic solitonic structures.

The long-time asymptotics for the {\it focusing} NLS in the case when the left background is zero, i.e.,
\begin{align}
u_0(x) = \begin{cases} 0, & x < 0, \\
A \e^{-2iB x}, & x > 0,
\end{cases}
\end{align}
where $A > 0$ and $B \in \R$ are constants, was analyzed in \cite{BKS2011}. It was found that there are three asymptotic sectors in this case: a slow decay sector adjacent to the negative $x$-axis, a modulated plane wave sector adjacent to the positive $x$-axis, and a modulated elliptic wave (genus $1$) sector in between. This means that the asymptotic picture observed in \cite{BKS2011} bears similarities to the one we find here for the defocusing NLS with the boundary conditions \eqref{boundaryconditions}: in both cases there are three asymptotic sectors, and in both cases there is a slow decay sector and a modulated plane wave sector adjacent to the $x$-axis, as mandated by the boundary conditions. However, in the focusing case studied in \cite{BKS2011} the middle sector is a genus $1$ sector (i.e., the leading asymptotic term is expressed in terms of theta functions attached to a genus $1$ Riemann surface), whereas in the defocusing case considered here the middle sector is a genus $0$ sector (i.e., the underlying Riemann surface is of genus $0$).
As in \cite{J2015}, it is assumed in \cite{BKS2011} that the initial data is a pure step-like profile (meaning that $u_0(x)$ is exactly equal to $0$ for $x < 0$ and to $A e^{-2iB x}$ for $x > 0$). 
We expect that the techniques developed in this paper will be useful to handle the case of boundary conditions that are only asymptotically imposed also for the focusing NLS as well as other integrable PDEs.

\subsection{Application to the half-line problem}
One motivation for studying step-like boundary conditions of the form \eqref{boundaryconditions} is that the restriction to the quarter-plane $\{x \geq 0, \, t \geq 0\}$ of any solution of the Cauchy problem \eqref{dNLS_IVP}--\eqref{boundaryconditions} provides a solution of the half-line problem for \eqref{dNLS} with decay as $x \to \infty$. This initial-boundary value problem for NLS has been studied extensively (see e.g. \cite{BF2008, BSZ2018, BKS2009, F1997, F2005, FHM2017, FIS2005, Fromm19, HMY2019, H2005, Len15, LFunifiedII, LFtperiodicI, LFtperiodicII, LQ21}). Of particular interest is the case of (asymptotically) time-periodic boundary conditions \cite{BF2008}. The main difficulty when studying such boundary conditions is the construction of the Dirichlet-to-Neumann map \cite{LFtperiodicI}. If both the Dirichlet and the Neumann boundary values are known, then an effective solution representation can be obtained \cite{F1997, FIS2005}. However, for a well-posed problem only one of these boundary values can be independently prescribed, and the construction of the unknown boundary value (i.e., the construction of the generalized Dirichlet-to-Neumann map) is in general a highly nonlinear problem \cite{F2005}. In the case of asymptotically time-periodic boundary conditions, progress can be made by considering the Dirichlet-to-Neumann map in the large $t$ limit (which is enough for the calculation of long-time asymptotics for the solution for any $x \geq 0$) \cite{BKS2009, LFtperiodicI}. For example, assuming that the Dirichlet and Neumann values asymptote to periodic functions, the large $t$-behavior of the Neumann value can be effectively computed from the Dirichlet data to all orders in perturbation theory \cite{LFtperiodicII}. 
It is particularly interesting to consider the situation when the boundary values asymptote to single exponentials. It was shown for the focusing NLS in \cite{BIK2009,BKS2009} that there exists a solution $u$ satisfying $u(0,t) \sim \alpha \e^{i\omega t}$ and $u_x(0,t) \sim c\e^{i\omega t}$ as $t \to \infty$ if and only if the triplet of constants $(\alpha,\omega,c)$ belongs to one of two explicitly given families, thereby characterizing the Dirichlet-to-Neumann map in this regime. A similar result was obtained for the defocusing NLS in \cite{Len15}. However, in \cite{Len15} {\it five} families were found instead of {\it two}, and the fact that $(\alpha,\omega,c)$ belongs to one of these families was (due to the lack of a vanishing lemma for the defocusing NLS on the half-line) only shown to be a {\it necessary} condition for the existence of a solution $u$ with $u(0,t) \sim \alpha \e^{i\omega t}$ and $u_x(0,t) \sim c\e^{i\omega t}$. Thus, to characterize the Dirichlet-to-Neumann map in this regime in the defocusing case, one needs to determine whether there really are solutions of the initial-boundary value problem on the half-line corresponding to each of the five families found in \cite{Len15}. One of these five families is (see \cite[Eq. (2.3c)]{Len15})
\begin{align}\label{familyc}
\Big\{\big(\alpha,\omega,c =  i \alpha \sqrt{-2\alpha^2-\omega}\big) \colon \alpha>0, \omega < -3\alpha^2 \Big\}.
\end{align}
We show in Theorem \ref{thm_ibvp} that for each choice of $(\alpha,\omega,c)$ in the family \eqref{familyc}, 
there is indeed a solution $u$ of the defocusing NLS such that $u(0,t) \sim \alpha \e^{i\omega t}$ and $u_x(0,t) \sim c\e^{i\omega t}$ as $t \to \infty$. This answers affirmatively the question left open in \cite{Len15} for this family. Moreover, Theorem \ref{thm_ibvp} provides a natural interpretation of the family \eqref{familyc}: boundary values corresponding to this family arise by restricting solutions of the Cauchy problem \eqref{dNLS_IVP}--\eqref{boundaryconditions} to the half-line $x \geq 0$ whenever the $t$-axis belongs to the left asymptotic sector, i.e., whenever $\alpha > 0$ and $\beta \in \R$ are such that $4\beta - 2 \alpha > 0$.

\subsection{Organization of the paper}
The six main theorems are stated in Section \ref{mainresultssec}. Their proofs can be found in Sections \ref{directsec}-\ref{ibvpsec}; the proof of the first theorem is given in Section \ref{directsec}, the proof of the second theorem in Section \ref{inversesec}, and so on.
The appendix gives an asymptotic expansion for a Cauchy type integral that appears in the steepest descent analysis.

\subsection{Notation}
We will make use of the following notation.
The three Pauli matrices will be denoted by $\sigma_1$, $\sigma_2$, and $\sigma_3$. 
For $A\in \C^{2\times 2}$, we let $\e^{\hat\sigma_3}A \coloneqq \e^{\sigma_3}A \e^{-\sigma_3}$. 
We denote by $[A]_1$ the first column of $A$ and by $[A]_2$ its second column.
We let $\C_+ \coloneqq \{\im k > 0\}$ and $\C_- \coloneqq \{\im k < 0\}$ denote the open upper and lower half-planes. Similarly, we let $\R_+ \coloneqq (0,+\infty)$ and $\R_- \coloneqq (-\infty, 0)$.

\section{Main results}\label{mainresultssec}
\noindent
Before stating the main results, we need to introduce some ingredients relevant for the direct and inverse scattering problems.

\subsection{Eigenfunctions and reflection coefficient}\label{RHproblem}
Equation \eqref{dNLS} admits the Lax pair
\begin{align}\label{lax}
\begin{cases}
\phi_x +  i k\sigma_3 \phi = \mathsf{U}\phi,
\\ 
\phi_t + 2 i k^2 \sigma_3 \phi = \mathsf{V} \phi,	
\end{cases}
\end{align}
where $k \in \C$  is the spectral parameter and 
\begin{align}\label{UVdef}
\mathsf{U} \coloneqq \begin{pmatrix} 0 & u \\
\bar{u} & 0 \end{pmatrix}, \qquad
\mathsf{V} \coloneqq \begin{pmatrix} - i |u|^2 & 2ku +  i u_x \\
2k \bar{u} -  i \bar{u}_x &  i |u|^2 \end{pmatrix}.
\end{align}

Suppose $\alpha>0$ and $\beta\in\R$. Let $\omega \coloneqq -4\beta^2-2\alpha^2<0$ and consider the following plane wave solution of \eqref{dNLS}:
\begin{equation}\label{backgroundsolution}
u^b\colon \R\times [0,\infty)\to \C, \qquad u^b(x,t) \coloneqq \alpha \e^{2 i\beta x+ i \omega t}.
\end{equation}
The corresponding ``background'' Lax pair is given by 
\begin{equation}\label{blax}
\begin{cases}
\phi^b_x +  i k\sigma_3 \phi^b = \mathsf{U}^b \phi^b,
\\ 
\phi^b_t + 2 i k^2 \sigma_3 \phi^b = \mathsf{V}^b \phi^b,	
\end{cases}
\end{equation}
with
\begin{align*}
& \mathsf{U}^b(x,t) = \begin{pmatrix} 0 & \alpha \e^{2 i\beta x+ i \omega t} \\
\alpha \e^{-2 i\beta x- i \omega t} & 0 \end{pmatrix}, 
	\\
& \mathsf{V}^b(x,t) = \begin{pmatrix} - i \alpha^2 & 2\alpha(k-\beta) \e^{2 i\beta x+  i\omega t} \\
2\alpha(k-\beta) \e^{-2 i\beta x-  i\omega t} &  i \alpha^2 \end{pmatrix}.
\end{align*}

\subsubsection{Background eigenfunction}\label{sec_background_ef}
Let $E_1\coloneqq -\beta-\alpha$ and $E_2\coloneqq -\beta + \alpha$. We define the functions $X,\Omega \colon \C\setminus (E_1,E_2)\to \C$ by
\begin{equation*}
	X(k)=\sqrt{(k-E_1)(k-E_2)}, \qquad \Omega(k)=2\left(k-\beta\right)X(k),
\end{equation*}
where the branch of the square root is chosen such that
\begin{equation*}
	X(k)=\sqrt{\left(k+\beta\right)^2-\alpha^2}=k+\beta+ \bigO(k^{-1})\qquad \text{as} \quad k\to\infty;
\end{equation*}
it follows that
$\Omega(k)=2 k^2 + \frac{\omega}2+ \bigO(k^{-1})$ as $k\to\infty$.
We further consider the function
\begin{equation}\label{Deltadef}
\Delta \colon \C\setminus [E_1,E_2]\to \C, \qquad	\Delta(k) = \bigg(\frac{k-E_2}{k-E_1}\bigg)^{1/4},
\end{equation}
where the branch of the fourth root is such that $\Delta(k)=1+ \bigO(k^{-1})$ as $k\to\infty$. 
 
The background Lax pair \eqref{blax} admits the explicit background eigenfunction
\begin{align}\label{backgroundefunction}
	\phi_1^b \colon \R\times [0,\infty)\times \big(\C\setminus [E_1,E_2]\big) \to \C^{2\times 2}, \quad 
\phi^b_1(x,t,k) = \e^{ i\left({\beta} x +\frac{\omega}{2} t\right) \sigma_3} \, s^b(k) \, \e^{- i( X(k) x+\Omega(k)t)\sigma_3},
\end{align}
where $s^b\colon \C\setminus [E_1,E_2]\to \C^{2\times 2}$ is given by
\begin{align}\label{def_s^b}
s^b \coloneqq \frac{1}{2}\begin{pmatrix} \Delta+ \Delta^{-1} &  i(\Delta - \Delta^{-1}) \\
- i(\Delta - \Delta^{-1}) &  \Delta+ \Delta^{-1} \end{pmatrix},
\end{align}
and satisfies $\det s^b(k) =1$ for $k\in\C\setminus[E_1,E_2]$.

\subsubsection{Reflection coefficient}
Let $u^b_0(x) = \alpha \e^{2i\beta x}$ and let $u_0\colon\R \to \C$ be such that 
\begin{equation*}
	(u_0 - u^b_0)|_{\R_-} \in  L^1(\R_-) \quad \text{and} \quad  u_0|_{\R_+} \in  L^1(\R_+).
\end{equation*}
Let $\mathsf{U}_0(x) = \mathsf{U}(x,0)$ and $\mathsf{U}_0^b(x) = \mathsf{U}^b(x,0)$. 
Define the eigenfunctions $\{\phi_j(x,k)\}^2_{j=1}$ as the unique solutions of the Volterra integral equations
\begin{subequations}\label{phi_Volterra}
\begin{align}
\phi_1(x,k)&=\phi_1^b(x,0,k)+\int_{-\infty}^x\phi_1^b(x,0,k)(\phi_1^b)^{-1}(x',0,k)\left[(\mathsf{U}_0-\mathsf{U}_0^b)(x')\right]\phi_1(x',k) \, \dx x',
\label{phi1_Volterra}\\
\phi_2(x,k)&=\e^{- i k x\sigma_3} - \int^\infty_x \e^{ i k(x'-x)\sigma_3} \mathsf{U}_0(x') \phi_2(x',k) \, \dx x'. \label{phi2_Volterra}
\end{align}
\end{subequations}
For each $x \in \R$, the map $\phi_1(x,\cdot) \colon (\C_+, \C_-) \to \C^{2\times2}$ is analytic and has a continuous extension to $(\overbar{\C_+},\overbar{\C_-}) \setminus \{E_1, E_2\}$, which is also denoted by $\phi_1$. This means that the first column of $\phi_1$ is continuous on $\overbar{\C_+}\setminus \{E_1, E_2\}$ while the second column is continuous on $\overbar{\C_-}\setminus \{E_1, E_2\}$.
For each $x \in \R$, the map $\phi_2(x,\cdot)\colon (\overbar{\C_-},\overbar{\C_+}) \to \C^{2\times2}$ is continuous and its restriction to $(\C_-,\C_+)$ is analytic. 
Both $\phi_1$ and $\phi_2$ solve the $x$-part of the Lax pair \eqref{lax} for $k\in\R\setminus\{E_1,E_2\}$ and $(x,t)\in\R\times[0,\infty)$. Hence there exists a unique matrix-valued function $s(k)$ such that
\begin{equation}\label{def_s}
	s(k) =\phi^{-1}_2(x,k)\,\phi_1(x,k), \qquad x \in \R, ~ k \in \R\setminus\{E_1,E_2\}.
\end{equation}
The symmetries of the Lax pair imply that $s$ has the form
\begin{equation} \label{def_s_entries}
	s(k)=\begin{pmatrix} a(k) & -b(k)\\ -\overbar{b(k)} & \overbar{a(k)}\end{pmatrix}, \qquad k \in \R\setminus\{E_1,E_2\}.
\end{equation}
We define the \emph{reflection coefficient} $r\colon \R\setminus\{E_1,E_2\}\to\C$ by
\begin{equation}\label{def_r}
	r(k)\coloneqq 
	\begin{cases} \frac{\overbar{b(k)}}{a(k)}, &k\in\R\setminus[E_1,E_2],  \vspace{.5em} \\
		-\frac{\overbar{a(k)}}{a(k)}, & k \in (E_1, E_2).
	\end{cases}
\end{equation}

\subsection{The direct problem} \label{sec_main_thm_direct}
Our first main result is concerned with the direct problem, i.e., the construction of the reflection coefficient $r(k)$ and the associated eigenfunctions from the initial data $u_0(x)$.

\begin{theorem}[Solution of the direct problem]\label{thm_direct}
Let $\alpha, \beta\in\R$ with $\alpha>0$ and let $N_1 \geq 2$ be an integer. Suppose $u_0\colon\R \to \C$ satisfies
\begin{equation}\label{u0assumptions}
x^n(u_0- u^b_0) \big|_{\R_-} \in L^1(\R_-),\quad  x^n u_0\big|_{\R_+} \in L^1 (\R_+),	\qquad n = 0, \dots,  N_1.
\end{equation}

\begin{enumerate}[\upshape (a)]
\item \label{thm_direct_a}
The reflection coefficient $r$ extends continuously to an element in $\mathcal C^{N_1}(\R\setminus \{E_1,E_2\})\cap \mathcal C(\R)$, again denoted by $r$, and satisfies the following properties.
\begin{enumerate}[\upshape (i)]
	\item\label{prop_r_i}
	$|r(k)| = 1$ for $k \in [E_1,E_2]$, and $|r(k)|<1$ for $k \in \R\setminus[E_1,E_2]$.
	\item \label{prop_r_iii}
	Near the branch points $E_j$, $j=1,2$, it holds that
	\begin{equation}\label{ratbranchpoints}
		r(k) =
		\begin{cases}
			\sum^{N_1-1}_{l=0} q_{2,l} (k-E_2)^{l/2} + o\big((k-E_2)^{\frac{N_1-1}{2}}\big)   &\text{as} \quad k\searrow E_2, \\
			\sum^{N_1-1}_{l=0}  i^l q_{2,l} (E_2-k)^{l/2} + o\big((E_2 -k)^{\frac{N_1-1}{2}}\big)   &\text{as}\quad k\nearrow E_2, \\
			\sum^{N_1-1}_{l=0}  i^l q_{1,l} (k-E_1)^{l/2} + o\big((k-E_1)^{\frac{N_1-1}{2}}\big)  &\text{as}\quad k\searrow E_1, \\
			\sum^{N_1-1}_{l=0} (-1)^l q_{1,l} (E_1-k)^{l/2} + o\big((E_1-k)^{\frac{N_1-1}{2}}\big)   &\text{as}\quad k\nearrow E_1, \\
		\end{cases}
	\end{equation}
	for some coefficients $q_{j,l}\in\C$ such that
	\begin{align} \label{coeffqjl}
	|q_{j,0}| =1, \quad q_{j,1}\neq 0, \quad \sum_{l=0}^n  i^{n-l} (- i)^l q_{j,n-l} \overbar{q_{j,l}} =0,   \qquad j=1,2, ~ n = 1, \dots, N_1-1.
\end{align}
\end{enumerate} 

\item \label{thm_direct_b}
If $u_0$ in addition to \eqref{u0assumptions} also satisfies \begin{align}\label{u0assumptions2}
u_0\in\mathcal C^{N_2}(\R) \quad \text{and} \quad \partial^n_x u_0 \in L^\infty(\R) \quad \text{for} \quad n = 0, \dots, N_2, 
\end{align}
for some integer $0\leq N_2\leq 4$, then 
\begin{equation*}
	 \partial^n_k r(k) = \bigO(k^{-N_2-1}) \qquad\text{as} \quad k\to\pm\infty, \quad n = 0, \dots,  N_1.
\end{equation*}

\item \label{thm_direct_c}
If $u_0$ satisfies \eqref{u0assumptions} and \eqref{u0assumptions2} for some integers $N_1 \geq 2$ and $N_2 \geq 0$, then the function $m_0(x,\cdot)\colon \C\setminus \R\to GL(2,\C)$, $x\in\R$, given by
\begin{align}
	\begin{aligned} \label{def_m_0}
	m_0(x,k) \coloneqq \begin{cases}
		\begin{pmatrix} \frac{[\phi_1(x,k)]_1}{a(k)} & [\phi_2(x,k)]_2 \end{pmatrix} \e^{ i k x\sigma_3}, & k \in \C_+, 
		\vspace{.1cm}	 \\ 
		\begin{pmatrix} [\phi_2(x,k)]_1 & \frac{[\phi_1(x,k)]_2}{\overbar{a(\bar{k})}} \end{pmatrix} \e^{ i k x\sigma_3}, & k \in \C_-,
	\end{cases}
\end{aligned}	
\end{align}
satisfies the following RH problem:
\begin{itemize}
\item $m_0(x,k)$ is analytic for $k \in \C\setminus \R$;

\item $m_0(x,\cdot)$ has continuous boundary values on $\R$ from the upper and lower half-planes denoted by $m_{0+}$ and $m_{0-}$, respectively, which obey the jump relation  
\begin{align}\label{RHm0}
m_{0+}(x,k) = m_{0-}(x,k) v_0(x,k) \quad \text{for} \; k \in \R,
\end{align}
where the jump matrix $v_0$ is given by
\begin{align} \label{def_v_0}
	v_0(x,k) &= \begin{pmatrix}1 - |r(k)|^2 & \overbar{r(k)} \e^{-2 i k x} \\
		-r(k) \e^{2 i k x} & 1 \end{pmatrix};
\end{align}

\item $m_0(x,k)=I+ \bigO(k^{-1})$ as $k\to \infty$.

\end{itemize}

\end{enumerate}
\end{theorem}
\begin{proof}
See Section  \ref{directsec}.
\end{proof}

The behavior \eqref{ratbranchpoints} of $r(k)$ at the two branch points $E_1$ and $E_2$ can be understood as follows. Suppose there is a $C > 0$ such that $u_0(x) = 0$  for $x > C$ and $u_0(x) = u_0^b(x)$  for $x < -C$. Then the reflection coefficient is defined and analytic for all $\C \setminus [E_1, E_2]$, and the function $r(k)$ defined for $k \in \R \setminus \{E_1, E_2\}$ in \eqref{def_r} denotes the boundary values of this analytic function from the upper half-plane. In this case, the relations between the coefficients in \eqref{ratbranchpoints} relevant for $k \searrow E_j$ and $k \nearrow E_j$ follow from the existence of the analytic continuation. In fact, viewing the reflection coefficient as a function on the two-sheeted Riemann surface defined by $X$ with a branch cut along $[E_1,E_2]$, we see that it is natural that the expansion at $E_j$ is in powers of the local analytic coordinate $\sqrt{k-E_j}$. 

For general initial data, $r(k)$ does not have an analytic continuation into the upper half-plane, but the structure of $r$ near the branch points remains unchanged. This will be crucial later in the steepest descent analysis, where we will use \eqref{ratbranchpoints} to construct an analytic approximation of $r$.

\begin{remark*}
The restriction $N_2 \leq 4$ in Theorem \ref{thm_direct}\eqref{thm_direct_b} is made for convenience and is not essential.
\end{remark*}

\subsection{Solution of the Cauchy problem}
Our second main result establishes existence of a global solution to the Cauchy problem \eqref{dNLS_IVP} with step-like initial data, and provides an expression for this solution in terms of the solution of a RH problem. 

\begin{definition*}
Suppose  $u_0\colon\R \to \C$ is step-like in the sense that \eqref{u0assumptions} holds for $N_1 = 2$. We say that $u:\R \times [0, \infty) \to \C$ is a {\it global solution of the Cauchy problem \eqref{dNLS_IVP} with initial data $u_0$} if the following hold:
\begin{enumerate}[\upshape (i)]
	\item $u(x,\cdot)\in\mathcal C^1\big((0,\infty)\big) \cap \mathcal C\big([0,\infty)\big)$ for each $x\in \R$
	and $u(\cdot,t)\in\mathcal C^2(\R)$ for each $t\in [0,\infty)$;
	
	\item $u$ satisfies the defocusing NLS equation \eqref{dNLS} for all $(x,t)\in\R \times(0,\infty)$;
	
	\item $u(x,0) = u_0(x)$ for all $x\in\R$;
	
	\item For each $t \geq 0$, $u$ satisfies the step-like boundary conditions
	\begin{align}\label{uboundaryconditions}
u(x, t) = \begin{cases}
\alpha \e^{2 i\beta x+ i \omega t} + o(1), & x\to -\infty, \\
o(1), & x\to +\infty.
\end{cases}
\end{align}
\end{enumerate}
\end{definition*}

\begin{theorem}[Solution of the Cauchy problem with step-like initial data] \label{thm_inverse}
Let $\alpha >0$ and $\beta\in\R$. Suppose $u_0\colon \R \to \C$ satisfies \eqref{u0assumptions} for $N_1 = 2$ and \eqref{u0assumptions2} for $N_2 = 3$. Then the Cauchy problem \eqref{dNLS_IVP} with initial data $u_0$ has a global solution and this solution can be constructed as follows.

Let $r\colon \R\setminus\{E_1,E_2\}\to\C$ be the reflection coefficient corresponding to $u_0$ according to \eqref{def_r}, and define  $v(x,t,\cdot)\colon \R \to SL(2,\C)$ by
\begin{align} \label{def_v}
	v(x,t,k) &\coloneqq \begin{pmatrix}1 - |r(k)|^2 & \overbar{r(k)} \e^{-2 i (kx + 2k^2 t)} \\
		-r(k) \e^{2 i (kx + 2k^2 t)} & 1 \end{pmatrix}.
\end{align}
For each $(x,t)\in \R\times [0,\infty)$, the RH problem
\begin{itemize}
\item $m(x,t,k)$ is analytic for $k \in \C\setminus \R$,

\item $m(x,t,\cdot)$ has continuous boundary values on $\R$ from the upper and lower half-planes, denoted by $m_+$ and $m_-$, respectively, which satisfy the jump relation  
\begin{align}\label{RHm}
m_+(x,t,k) = m_-(x,t,k) v(x,t,k) \quad \text{for} \; k \in \R,
\end{align}

\item $m(x,t,k)=I+ \bigO(k^{-1})$ as $k\to \infty$,

\end{itemize}
has a unique solution $m(x,t,\cdot)$, and 
\begin{equation} \label{def_u}
u(x,t) \coloneqq 2 i	\lim_{k\to\infty}  k m_{12}(x,t,k), \qquad (x,t)\in \R\times [0,\infty),
\end{equation}
where the limit is taken along any ray $\{k \colon \arg k = \varphi\}$ with $\varphi \in (0,\pi) \cup (\pi, 2\pi)$,
is a global solution of the Cauchy problem \eqref{dNLS_IVP} with initial value $u_0$. 
\end{theorem}
\begin{proof}
See Section  \ref{inversesec}; the asymptotic property in \eqref{uboundaryconditions} is treated in Sections \ref{leftsec}\&\ref{rightsec}. 
\end{proof}

\subsection{Long-time asymptotics} \label{sec_asymptotics}
Our next three theorems deal with the long-time behavior of the solution $u(x,t)$ found in Theorem \ref{thm_inverse}.

\subsubsection{Asymptotic sectors} \label{sec_asymptotic_sectors}
Let $\mathcal H$ denote the closed half-plane
\begin{equation*}
	\mathcal H \coloneqq \{(x,t)\colon x\in\R, t\geq 0\}.
\end{equation*}
Let $\alpha>0$, $\beta\in\R$, and $\xi\coloneqq x/t$. 
We consider the asymptotic sectors 
\begin{align*}
	\mathcal L &\coloneqq \{(x,t)\in\mathcal H\colon \xi\in I_{\mathcal L}\}, 
	&& \,\, I_{\mathcal L}\coloneqq (-\infty, 4\beta - 2\alpha-\delta],   \\
	\mathcal M &\coloneqq \{(x,t)\in\mathcal H\colon \xi\in I_{\mathcal M}\}, 
	&& I_{\mathcal M}\coloneqq [4\beta - 2\alpha+\delta, 4\beta +4\alpha-\delta], \\
	\mathcal R &\coloneqq   \{(x,t)\in\mathcal H\colon \xi\in I_{\mathcal R}\}, 
	&&\,\, I_{\mathcal R}\coloneqq [4\beta +4\alpha+\delta,+\infty),
\end{align*}
where $\delta>0$ is a small positive constant, see Figures \ref{fig:sectors}--\ref{fig:sectors3}.

\begin{figure}[h!] \centering
	\vspace{.2em}
	\begin{overpic}[width=.63\textwidth]{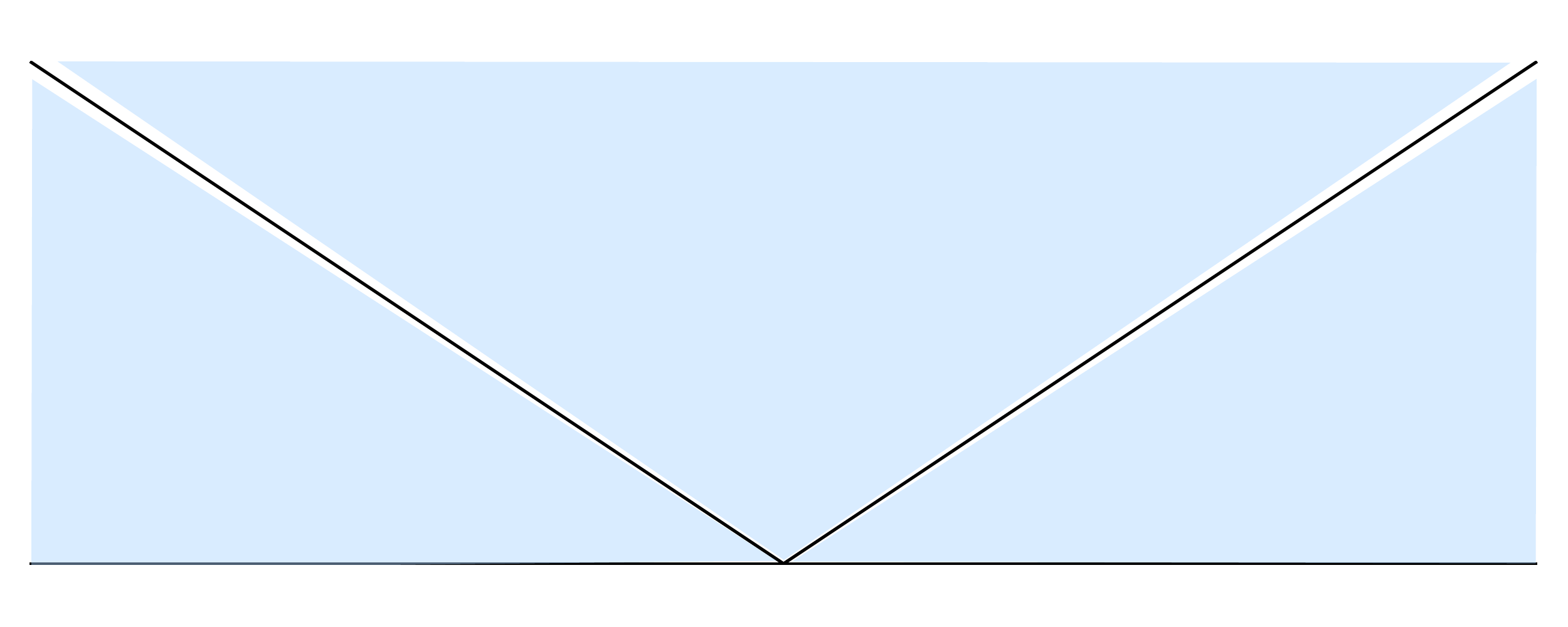}
		\put(49.2,1){\footnotesize{$0$}}     
		\put(100,3.4){\footnotesize{$x$}}  
		\put(14,13){\footnotesize{$\mathcal L$}} 
		\put(48,26){\footnotesize{$\mathcal M$}} 
		\put(80,13){\footnotesize{$\mathcal R$}}
		\put(-10,37.5){\footnotesize{$\xi=4\beta - 2\alpha$}}
		\put(92,37.5){\footnotesize{$\xi=4\beta+4\alpha$}}
	\end{overpic}
	\caption{Illustration of the asymptotic sectors $\mathcal L$, $\mathcal M$, and $\mathcal R$ in the case when $4\beta - 2\alpha < 0 < 4\beta+4\alpha$.} 
	\label{fig:sectors}
\end{figure}

\subsubsection{Asymptotics in $\mathcal{M}$} \label{sec_asymp_M}
Our first asymptotic theorem provides the asymptotics in the middle sector.
Recall that $E_1 =  -\beta-\alpha$.
For $\xi\in I_{\mathcal M}$, let
\begin{align*} 
	k_0(\xi) &\coloneqq \frac{\beta+\alpha - \xi}{3}, \\
	\mathcal{X}(\xi,k) &\coloneqq \sqrt{(k - E_1)(k-k_0)},  \qquad k\in \C \setminus (E_1,k_0),
\end{align*} 
where the branch is such that $\mathcal{X} \sim k$ as $k\to\infty$.
Furthermore, consider the $\xi$-dependent quantities $C_{k_0}, g_\infty, \mathcal{D}^{-2}_\infty$ given by
\begin{flalign}
	&\quad\;\;
	\begin{aligned} \label{C_k0}
		C_{k_0} \coloneqq \frac{\sqrt{k_0 - E_1}}{2 \pi} \Bigg[ & i \int^{E_1}_{-\infty} \frac{\log(1-|r(s)|^2)}{\mathcal{X}(s)(s-k_0)} \, \dx s \\
		&\quad+ \int^{k_0}_{E_1} \frac{\frac{\log(r(s))}{\sqrt{s - E_1}} - \frac{\log(r(k_0))}{\sqrt{k_0 - E_1}}}{s-k_0} \frac{\dx s}{\sqrt{k_0-s}}
		+ \frac{2 \log(r(k_0))}{k_0-E_1} \Bigg],
	\end{aligned}\\
	&\;\; g_\infty(\xi) \coloneqq \frac{1}{12} \Big(\xi ^2 - 4 E_1 \xi -8 E_1^2 \Big), \label{def_g_infty_sec_M} \\
	&\mathcal{D}^{-2}_\infty (\xi) \coloneqq \exp
	\Bigg[\frac{1}{\pi  i} \Bigg( \int^{E_1}_{-\infty} \frac{\log(1-|r(s)|^2)}{\mathcal{X}(\xi,s)} \, \dx s + 
	\int^{k_0}_{E_1} \frac{\log r(s)}{\mathcal{X}_+(\xi,s)} \, \dx s  \Bigg)\Bigg],
	\label{Dinfty^-2}
\end{flalign}
where the branch of the function $\log{r(s)}$ is such that it is continuous for $s \in [E_1, k_0]$ and the principal branch is used for $\log{r(E_1)}$.

\begin{theorem}[Asymptotics in the middle sector] \label{thm_sec_M}
Suppose $u_0$ satisfies the assumptions of Theorem \ref{thm_direct} with $N_1 = 8$ and $N_2=3$. 
Let $u$ be the solution of the Cauchy problem \eqref{dNLS_IVP} with initial data $u_0$ constructed in Theorem \ref{thm_inverse}. As $t\to \infty$, the restriction $u|_{\mathcal M}$ of $u$ to $\mathcal M$ satisfies
	\begin{align} \label{lim_km_12}
			u|_{\mathcal M}(x,t)
			= - \mathcal{D}^{-2}_\infty(\xi) \,\e^{2 i t g_\infty (\xi)}   \bigg[\frac{4(\alpha + \beta)-\xi}{6} +   \frac{12 C_{k_0}^2  - \frac{24}{\sqrt{k_0 - E_1}} C_{k_0} +\frac{7}{k_0 - E_1} }{144 i t}  \bigg]
			+  \bigO (t^{-2})
	\end{align}
	 uniformly for $\xi\in I_{\mathcal M}$.
\end{theorem}
\begin{proof}
See Section  \ref{middlesec}.
\end{proof}

\subsubsection{Asymptotics in $\mathcal{L}$} \label{sec_asymp_L}
Let $\xi\in I_{\mathcal L}$. Recall that $X(k)=\sqrt{(k-E_1)(k-E_2)}$, where $E_1 = -\beta-\alpha$ and $E_2 = -\beta + \alpha$. In order to explicitly state the leading and subleading terms of the asymptotics of $u$ in $\mathcal{L}$ in terms of $\alpha$, $\beta$, $\xi$, and the reflection coefficient $r$, we introduce the following quantities:
\begin{align}
	\begin{aligned}\label{leftsectorquantities}
	k_0(\xi) &\coloneqq  -\frac{4\beta + \xi}{8} + \sqrt{\frac{\alpha^2}{2} + \Big(\frac{4\beta - \xi}{8}\Big)^2}, \qquad \nu \coloneqq -\frac1{2\pi}\log(1-|r(k_0)|^2), 
		\\
	 g(k) & \coloneqq (2k -2\beta+\xi)X(k), \qquad \psi(k_0) \coloneqq \frac{2\sqrt{2}\big({\frac{\alpha^2}{2} + \big(\frac{4\beta - \xi}{8}\big)^2}\big)^{1/4}}{\sqrt{X(k_0)}},
		\\
	D(\xi, k) & \coloneqq \e^{\frac{X(k)}{2\pi  i}\big\{\big(\int_{-\infty}^{E_1} + \int_{E_2}^{k_0}\big) \frac{\log(1 - |r(s)|^2)}{X(s)(s - k)}  \, \dx s
		+ \int_{E_1}^{E_2} \frac{ \log r(s)}{X_+(s)(s - k)}  \, \dx s \big\}}, \qquad
	D_\infty(\xi) \coloneqq \lim_{k\to\infty}D(\xi, k), 
		\\
	D_b(k_0)&\coloneqq \lim_{\R\ni k\searrow k_0}\left[(k-k_0)^{-i\nu} D(\xi, k)\right],  \qquad  \beta_{k_0}^{\mathsf X} \coloneqq \sqrt{\nu} \e^{ i\left(\frac{3\pi}{4} - \arg (-r(k_0)) + \arg \Gamma( i\nu )\right)}, 
	\end{aligned}
\end{align}
where $\Gamma$ denotes the Gamma function and the branch of the function $\log{r(s)}$ is such that it is continuous for $s \in [E_1, E_2]$ and the principal branch is used for $\log{r(E_1)}$.
Furthermore, recalling the definition \eqref{Deltadef} of $\Delta(k)$ and that $\omega =  -4\beta^2-2\alpha^2$, we define $u_a \colon \mathcal L \to \C$ by
\begin{align} \label{def_u_a}
	\begin{aligned}
	u_a(x,t)
	\coloneqq -D^{-2}_\infty(\xi) \,\e^{2 i\beta x + i\omega t}
	&\Bigg(\frac{ i t^{- i\nu}{\beta^{\mathsf X}_{k_0}}(\Delta(k_0)^2+1)^2}{2\e^{2 i tg(k_0)}\Delta(k_0)^2\psi(k_0)^{1+2 i \nu}D_b(k_0)^{-2}}
	\\  
	&\quad +\frac{ i t^{ i\nu}\overbar{\beta^{\mathsf X}_{k_0}}(\Delta(k_0)^2-1)^2}{2\e^{-2 i tg(k_0)}\Delta(k_0)^2\psi(k_0)^{1-2 i\nu}D_b(k_0)^2}\Bigg).		
	\end{aligned}
\end{align}

\begin{theorem}[Asymptotics in the left sector]\label{thm_sec_L}
	Suppose $u_0$ satisfies the assumptions of Theorem \ref{thm_direct} with $N_1 = 8$ and $N_2=4$. 
Let $u$ be the solution of the Cauchy problem \eqref{dNLS_IVP} with initial data $u_0$ constructed in Theorem \ref{thm_inverse}. As $t\to \infty$, the restriction $u|_{\mathcal L}$ of $u$ to $\mathcal L$ satisfies
	\begin{align}\label{u_asymp_L}
			u|_{\mathcal L}(x,t) &= -D^{-2}_\infty(\xi)\, \alpha \e^{2 i\beta x + i\omega t}+\frac{u_a(x,t)}{\sqrt{t}}+\bigO\!\left(\frac{\log t}{t}\right)
	\end{align}
	uniformly for $\xi\in I_{\mathcal L}$, where 
	\begin{equation} \label{lim_D^{-2}_infty}
			\lim_{\xi\to-\infty} D^{-2}_\infty(\xi) = -1.
	\end{equation}
	Moreover, the error term in \eqref{u_asymp_L} can be replaced by $\bigO\big(|x|^{-1}\max(1,\log t)\big)$ as $x \to -\infty$, i.e., 	
\begin{align}\label{u_asymp_Lx}
		u|_{\mathcal L}(x,t)
		=  -D^{-2}_\infty(\xi)\, \alpha \e^{2 i\beta x + i\omega t} + \frac{u_a(x,t)}{\sqrt{t}} + \bigO\bigg(\frac{\max(1,\log t)}{|x|}\bigg), \qquad x \to -\infty,
	\end{align}
uniformly for $\xi \in I_{\mathcal L}$. In particular, for each fixed $t \geq 0$, $u(x,t) = \alpha \e^{2 i\beta x+ i \omega t} + o(1)$ as $x \to -\infty$.
Finally, the derivative $u_x$ satisfies, uniformly for $\xi\in I_{\mathcal L}$,
	\begin{align}\label{ux_asymp_L}
			u_x|_{\mathcal L} &= -D^{-2}_\infty(\xi)\,2i\beta \alpha \e^{2 i\beta x + i\omega t}+\bigO(t^{-1/2}), \qquad t \to \infty.
	\end{align}
\end{theorem}
\begin{proof}
See Section  \ref{leftsec}.
\end{proof}

\begin{figure}[h!] \centering
	\begin{overpic}[width=.63\textwidth]{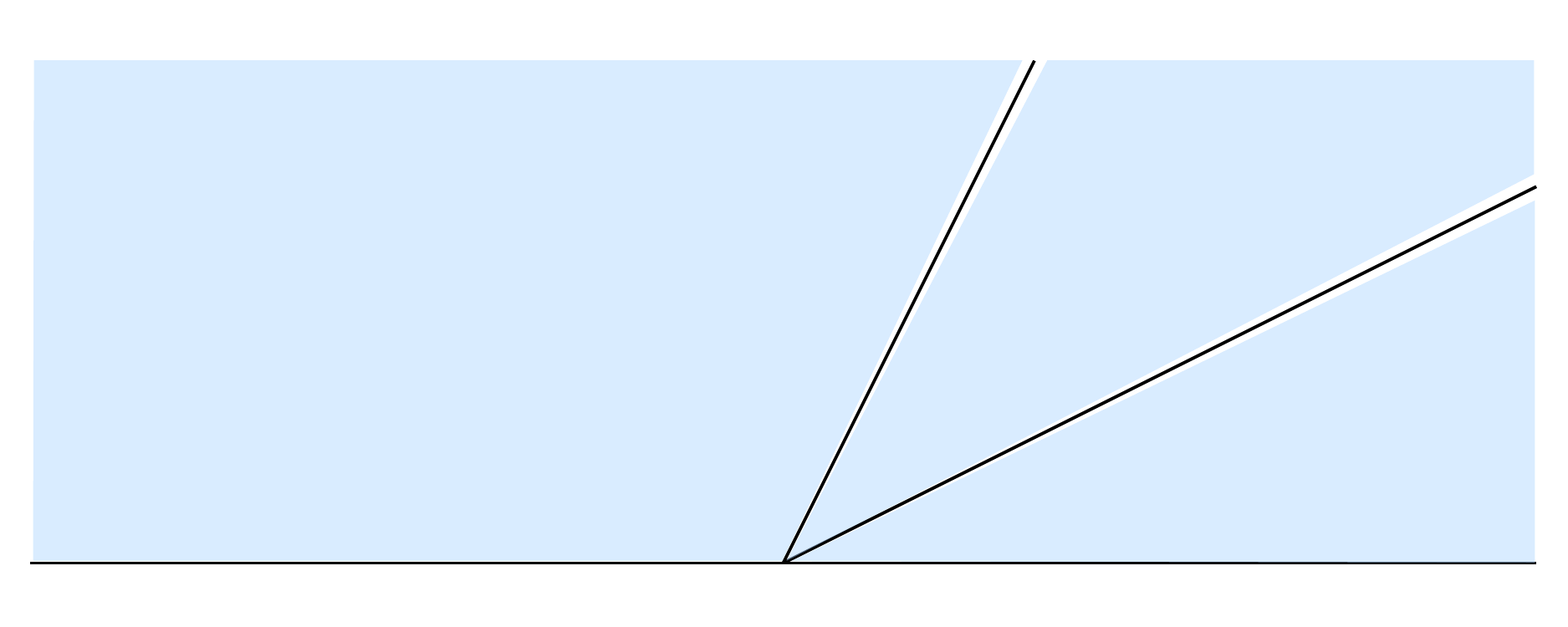}
		\put(49.2,1){\footnotesize{$0$}}     
		\put(100,3.4){\footnotesize{$x$}}  
		\put(25,18){\footnotesize{$\mathcal L$}} 
		\put(72,26){\footnotesize{$\mathcal M$}} 
		\put(80,10){\footnotesize{$\mathcal R$}}
		\put(66,37.5){\footnotesize{$\xi=4\beta - 2\alpha$}}
		\put(99,27.5){\footnotesize{$\xi=4\beta+4\alpha$}}
	\end{overpic}
	\caption{Illustration of the asymptotic sectors $\mathcal L$, $\mathcal M$, and $\mathcal R$ in the case when $0 < 4\beta - 2\alpha < 4\beta+4\alpha$.} 
	\label{fig:sectors2}
\end{figure}
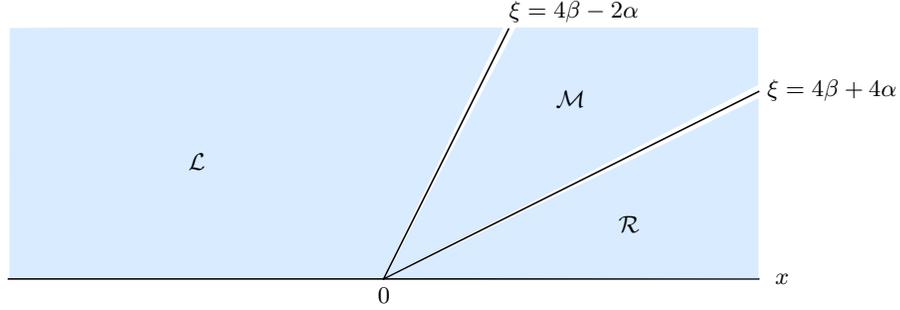

\begin{remark*}
The asymptotic formulas \eqref{u_asymp_L} and \eqref{u_asymp_Lx} are equivalent for any fixed $\xi < 0$. The formula \eqref{u_asymp_Lx} is useful as $\xi \to - \infty$ (i.e., in the region close to the negative $x$-axis), while \eqref{u_asymp_L} has a better error term than \eqref{u_asymp_Lx} in the part of $\mathcal{L}$ near and to the right of the $t$-axis which is present if $4\beta - 2\alpha > 0$, see Figure \ref{fig:sectors2}.
\end{remark*}

\begin{remark*}[Matching]
The asymptotic formulas for $u|_{\mathcal{M}}$ and $u|_{\mathcal{L}}$ match formally to leading order on the line $\xi = 4\beta - 2\alpha$ where the middle and left sectors meet. Indeed, by \eqref{lim_km_12}, the leading term of $u|_{\mathcal{M}}$ evaluated at $\xi = 4\beta - 2\alpha$ is
\begin{equation*}
	- \mathcal D^{-2}_\infty(\xi) \,\e^{2 i t g_\infty (\xi)}  \frac{4(\alpha + \beta)-\xi}{6}\Big|_{\xi = 4\beta - 2\alpha}
	= -D^{-2}_\infty(\xi) \, \e^{2 i\beta x + i\omega t} \alpha \big|_{\xi = 4\beta - 2\alpha}
\end{equation*} 
which coincides with the leading term in \eqref{u_asymp_L} evaluated at $\xi = 4\beta - 2\alpha$.
\end{remark*}

\begin{remark*}
The choice of branch of $\log{r(s)}$ in \eqref{C_k0}, \eqref{Dinfty^-2}, and \eqref{leftsectorquantities} does in fact not matter as long as $\log{r(s)}$ remains continuous. This is a consequence of the identities 
$$\int^{k_0}_{E_1} \frac{\frac{1}{\sqrt{s - E_1}} - \frac{1}{\sqrt{k_0 - E_1}}}{s-k_0} \frac{\dx s}{\sqrt{k_0-s}} + \frac{2}{k_0-E_1} = 0, \quad \int^{k_0}_{E_1} \frac{\dx s}{\mathcal{X}_+(\xi,s)} = -\pi i,$$
and
$$ \quad \int^{E_2}_{E_1} \frac{\dx s}{X_+(\xi,s)} = -\pi i, \quad \frac{X(\xi, k_0)}{2\pi i} \int_{E_1}^{E_2} \frac{\dx s}{X_+(\xi, s)(s-k_0)} = \frac{1}{2}.$$
\end{remark*}

\subsubsection{Asymptotics in $\mathcal{R}$} \label{sec_asymp_R}

For $\xi\in I_{\mathcal R}$ let
\begin{align*}
	k_0(\xi) &\coloneqq -\frac{\xi}{4}\\
	\varphi(\xi,t)&\coloneqq \frac{3\pi+\xi^2 t}{4} 
	- \arg[-r(k_0)] + \arg\Gamma\bigg( \frac{\log\big[1-|r(k_0)|^2\big]}{2\pi i} \bigg) \\
	&\qquad +\frac{1}{2\pi}\log\big[1-|r(k_0)|^2\big] \log[8t]
	+\frac{1}{\pi} \int^{k_0}_{-\infty} \log[k_0-s]  \, \dx  \log(1-|r(s)|^2),
\end{align*}
where $\Gamma$ denotes the Gamma function.

\begin{theorem}[Asymptotics in the right sector] \label{thm_sec_R}
Suppose $u_0$ satisfies the assumptions of Theorem \ref{thm_direct} with $N_1 = 8$ and $N_2=4$. 
Let $u$ be the solution of the Cauchy problem \eqref{dNLS_IVP} with initial data $u_0$ constructed in Theorem \ref{thm_inverse}. As $t\to \infty$, the restriction $u|_{\mathcal R}$ of $u$ to $\mathcal R$ satisfies
	\begin{align} \label{u_asymp_R}
			u|_{\mathcal R}(x,t)
			= -\frac{2 i}{\sqrt{8t}}\sqrt{-\frac{1}{2\pi}\log(1-|r(k_0)|^2)} \, \e^{ i\varphi(\xi,t)}
			+ \bigO\bigg(\frac{\log t}{t}\bigg)
	\end{align}
uniformly for $\xi \in I_{\mathcal R}$. Moreover, the error term in \eqref{u_asymp_R} can be replaced by $\bigO\big(x^{-1}\max(1,\log t)\big)$ as $x \to +\infty$, i.e., 
	\begin{align}\label{u_asymp_Rx}
		u|_{\mathcal R}(x,t)
		= -\frac{2 i}{\sqrt{8t}}\sqrt{-\frac{1}{2\pi}\log(1-|r(k_0)|^2)} \, \e^{ i\varphi(\xi,t)}
		+ \bigO\bigg(\frac{\max(1,\log t)}{x}\bigg), \qquad x \to +\infty,
	\end{align}
uniformly for $\xi \in I_{\mathcal R}$. In particular, for each fixed $t \geq 0$, $u(x,t) = o(1)$ as $x \to +\infty$.
\end{theorem}
\begin{proof}
See Section  \ref{rightsec}.
\end{proof}

\begin{figure}[h!] \centering
	\begin{overpic}[width=.63\textwidth]{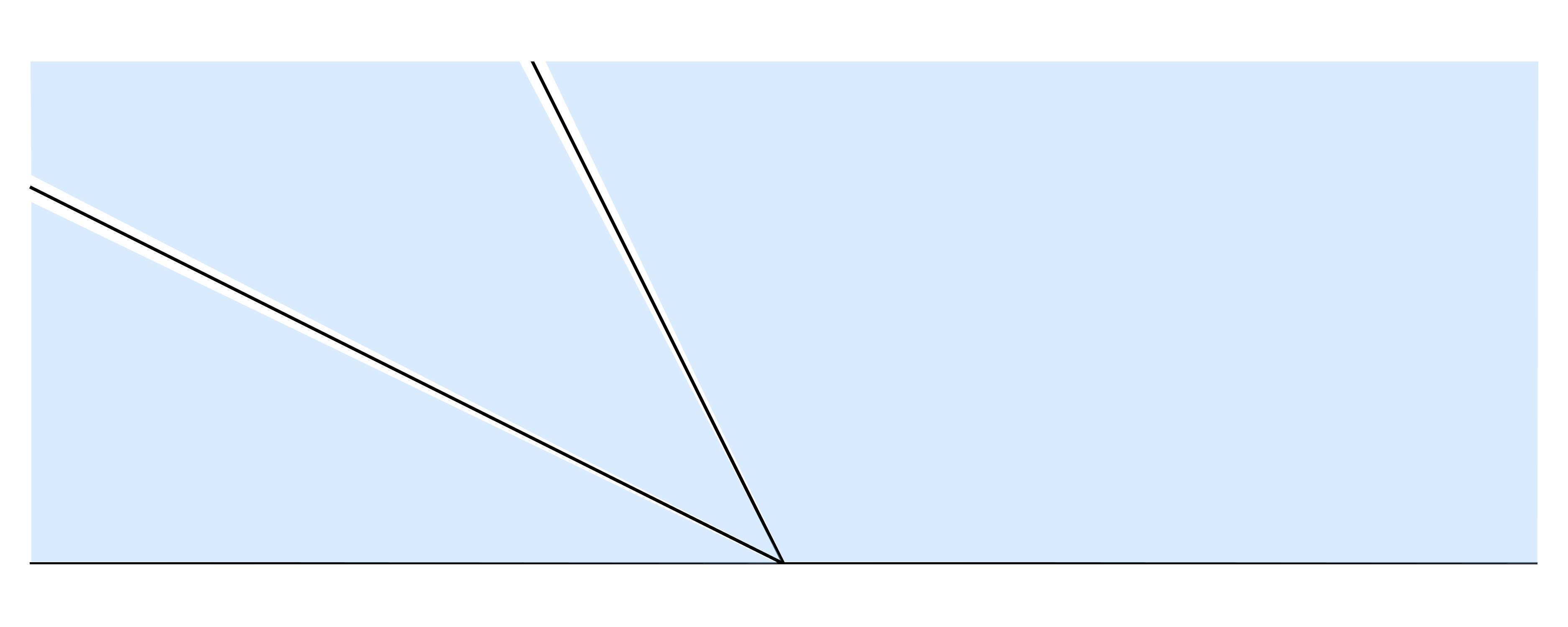}
		\put(49.2,1){\footnotesize{$0$}}     
		\put(100,3.4){\footnotesize{$x$}}  
		\put(14,11){\footnotesize{$\mathcal L$}} 
		\put(25,24){\footnotesize{$\mathcal M$}} 
		\put(70,20){\footnotesize{$\mathcal R$}}
		\put(-16,27.5){\footnotesize{$4\beta - 2\alpha=\xi$}}
		\put(32,37.5){\footnotesize{$\xi=4\beta+4\alpha$}}
	\end{overpic}
	\caption{Illustration of the asymptotic sectors $\mathcal L$, $\mathcal M$, and $\mathcal R$ in the case when $4\beta - 2\alpha < 4\beta+4\alpha < 0$.} 
	\label{fig:sectors3}
\end{figure}
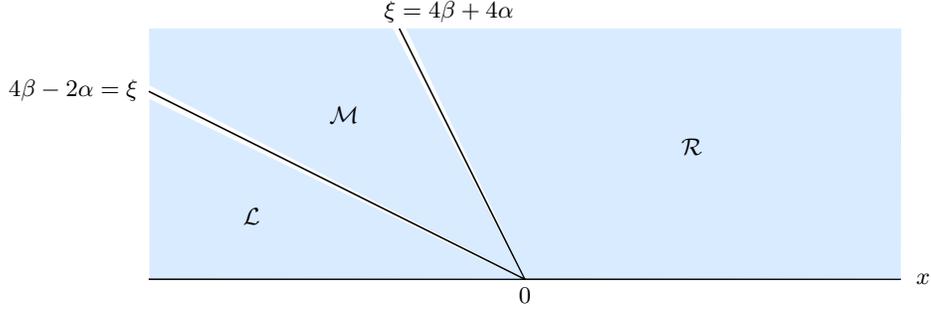

\begin{remark*}
The asymptotic formulas \eqref{u_asymp_R} and \eqref{u_asymp_Rx} are equivalent for any fixed $\xi > 0$. The formula \eqref{u_asymp_Rx} is useful as $\xi \to + \infty$ (i.e., in the region close to the positive $x$-axis), while \eqref{u_asymp_R} has a better error term than \eqref{u_asymp_Rx} in the part of $\mathcal{R}$ near and to the left of the $t$-axis which is present if $4\beta+4\alpha < 0$, see Figure \ref{fig:sectors3}. 
\end{remark*}

\begin{remark*}[Matching]
The leading term in the asymptotics \eqref{lim_km_12} for $u_{\mathcal{M}}$ clearly vanishes on the line $\xi = 4\beta + 4\alpha$ where the middle and right sectors meet. This is consistent with the slow decay of $u|_{\mathcal{R}}$ found in \eqref{u_asymp_R}.
\end{remark*}

\begin{remark*}[Transition regions]
The asymptotic formulas derived in Theorems \ref{thm_sec_M}--\ref{thm_sec_R} are valid uniformly for $x/t$ in any compact subsets of $(-\infty, 4\beta - 2\alpha)$, $(4\beta - 2\alpha, 4\beta + 4\alpha)$, and $(4\beta + 4\alpha, +\infty)$, respectively, and they match formally to leading order where the different sectors meet. However, a separate analysis---which we will not pursue here---is required to obtain uniform asymptotics in the two narrow transition regions containing the lines $x/t = 4\beta - 2\alpha$ and $x/t = 4\beta + 4\alpha$. 
\end{remark*}

\subsection{Application to the half-line problem} \label{sec_ibvp}
If $4\beta-2\alpha>0$, the asymptotic formula \eqref{u_asymp_L} evaluated at $\xi = 0$ yields
\begin{align*}
u(0,t) &= -D^{-2}_\infty(0)\, \alpha \e^{i\omega t}+ \mathcal{O}(t^{-1/2}), \qquad t \to \infty,
\end{align*}
showing that the restriction of $u$ to the $t$-axis asymptotes to a single exponential for large $t$. 
This suggests the following theorem.

\begin{theorem}[Solutions of the half-line problem with asymptotically $t$-periodic boundary values] \label{thm_ibvp}
Let $(\alpha,\omega,c) \in (0,\infty) \times \R \times \C$ belong to the family \eqref{familyc}. Then there is a solution $u$ of the defocusing NLS equation on the half-line $x \geq 0$ such that 
\begin{align}\label{uhalflinea}
u(x,t) \to 0 \quad \text{as} \quad  x \to +\infty,\quad t\geq 0,
\end{align}
and
\begin{align}\label{uhalflineb}
u(0,t) = \alpha \e^{i\omega t} + \bigO(t^{-1/2}) \quad \text{and} \quad u_x(0,t) = c \e^{i\omega t} + \bigO(t^{-1/2}) \quad \text{as} \quad t \to \infty.
\end{align}
In fact, let $\beta \coloneqq c/(2i\alpha) > 0$ and suppose $u_0$ satisfies the assumptions of Theorem \ref{thm_direct} corresponding to $\alpha$ and $\beta$ with $N_1=8$ and $N_2=4$. Let $u$ be the solution of the Cauchy problem \eqref{dNLS_IVP} with initial data $u_0$ constructed in Theorem \ref{thm_inverse}. Then $\tilde{u}(x,t) \coloneqq -D^{2}_\infty(0)u(x,t)$ satisfies \eqref{uhalflinea} and \eqref{uhalflineb}, so (since $|D^{2}_\infty(0)| = 1$) the restriction of $\tilde{u}$ to the quarter-plane $\{x \geq 0, t \geq 0\}$ is a solution of the half-line problem with the desired properties.
\end{theorem}
\begin{proof}
	See Section \ref{ibvpsec}.
\end{proof}

\section{Proof of Theorem \ref{thm_direct}}\label{directsec}
\noindent
Let $\alpha > 0$ and $\beta\in\R$. Suppose $u_0$ satisfies \eqref{u0assumptions} for some $N_1 \geq 2$ and let $\phi_1^b$ be the solution of the background Lax pair defined in \eqref{backgroundefunction}. 

Define $\mu_1(\cdot,k)\colon \R\to\C^{2\times 2}$, $k\in (\overbar{\C_+},\overbar{\C_-})\setminus [E_1,E_2]$, as the unique solution of the Volterra equation
\begin{align} 
\begin{aligned}\label{defmu10}
	\mu_1(x,k) =  \e^{ i \beta x \hat \sigma_3} s^b(k)+\int_{-\infty}^x\phi_1^b(x,0,k)&(\phi_1^b)^{-1}(x',0,k)\left[(\mathsf{U}_0 -\mathsf{U}_0^b)(x')\right]
	\\
	&\times
	\mu_1(x',k) \e^{ i (X(k)-\beta) (x-x')\sigma_3}\, \dx x',
\end{aligned}
\end{align}
where $k\in (\overbar{\C_+},\overbar{\C_-})\setminus [E_1,E_2]$ indicates that $[\mu_1]_1$ and $[\mu_1]_2$ are defined for $k\in \overbar{\C_+} \setminus [E_1,E_2]$ and $k\in \overbar{\C_-} \setminus [E_1,E_2]$, respectively.
For $k\in(\overbar{\C_-},\overbar{\C_+})$, we define the function $\mu_2 (\cdot,k)\colon \R\to\C$ as the unique solution of
\begin{align} \label{def_mu_2}
	\mu_2(x,k) =  I- \int_x^\infty \e^{ i k(x'-x) \hat{\sigma}_3} (\mathsf{U}_0 \mu_2)(x',k) \,\dx x'.
\end{align}
Some basic properties of $\mu_1$ and $\mu_2$ are collected in the following proposition whose proof is standard.

\begin{proposition}\label{prop_mu}
The functions $\mu_1$ and $\mu_2$ have the following properties:
\begin{enumerate}[\upshape (i)]
		\item \label{prop_mu_i}
		For each $x \in \R$, the map $\mu_1(x,\cdot) \colon (\C_+, \C_-) \to \C^{2\times2}$ is analytic and has a continuous extension to $(\overbar{\C_+},\overbar{\C_-}) \setminus \{E_1, E_2\}$, also denoted by $\mu_1$. Moreover, the restriction of $\mu_1$ to $\R\setminus\{E_1,E_2\}$ lies in $\mathcal C^{N_1}(\R\setminus\{E_1,E_2\},\C^{2\times2})$.
		
		\item \label{prop_mu_ii}
		For each $x \in \R$, the map $\mu_2(x,\cdot)\colon (\overbar{\C_-},\overbar{\C_+}) \to \C^{2\times2}$ is continuous and its restriction to $(\C_-,\C_+)$ is analytic. Moreover, its
		the restriction to $\R$ lies in $\mathcal C^{N_1}(\R,\C^{2\times2})$.
		
		\item \label{prop_mu_iii}
		It holds that
		\begin{align*}
			\lim_{x\to-\infty} [\mu_1(x,k)-\e^{i\beta x \hat\sigma_3} s^b(k)]&=0,\quad k\in(\overbar{\C_+},\overbar{\C_-})\setminus\{E_1,E_2\},
			\\
			\lim_{x\to+\infty}\mu_2(x,0,k)&=I,\quad k\in(\overbar{\C_-},\overbar{\C_+}).
		\end{align*}
		
		\item \label{prop_mu_iv}
		For each $x \in\R$, $\mu_1$ and $\mu_2$ obey the following symmetries:
		\begin{align}
			&\sigma_1\overbar{\mu_1(x,\bar{k})}\sigma_1=\mu_1 (x,k), \qquad  k\in (\overbar{\C_+},\overbar{\C_-})\setminus \{E_1, E_2\}, \label{symmetry_mu_1} \\  
			&\sigma_1\overbar{\mu_2(x,\bar{k})}\sigma_1=\mu_2 (x,k), \qquad  k\in (\overbar{\C_-},\overbar{\C_+}) \label{symmetry_mu_2}.
		\end{align}
		
		\item \label{prop_mu_v}
		For each $x \in\R$,
		\begin{align}
			&\det \mu_1(x,k)=1, \qquad k\in  \R\setminus{[E_1,E_2]},    \\
			&\det \mu_2(x,k)=1, \qquad k\in  \R. 
		\end{align}
	\end{enumerate}
\end{proposition}

Note that our notation is such that if $k \in (E_1, E_2)$ lies on the branch cut, then $[\mu_1(x,k)]_1$ denotes the boundary value of $[\mu_1(x,\kappa)]_1$ as $\kappa$ approaches $k$ from the upper half-plane, whereas $[\mu_1(x,k)]_2$ denotes the boundary value of $[\mu_1(x,\kappa)]_2$ as $\kappa$ approaches $k$ from the lower half-plane; an analogous remark applies to $[\phi_1]_1$ and $[\phi_1]_2$. 

The functions $\mu_1$ and $\mu_2$ are related to the eigenfunctions $\phi_1$ and $\phi_2$ defined in \eqref{phi_Volterra} by
\begin{align}
\begin{aligned}\label{defphi120}
 & \phi_1(x,k) = \mu_1(x,k) \e^{- i (X(k)-\beta)x \sigma_3},  
 &&k\in (\overbar{\C_+},\overbar{\C_-})\setminus \{E_1,E_2\},
  	\\ 
 &  \phi_2(x,k) = \mu_2(x,k) \e^{- i kx\sigma_3},  
 &&k\in (\overbar{\C_-},\overbar{\C_+}).
\end{aligned}
\end{align}
The functions $\phi_1$ and $\phi_2$ satisfy the $x$-part of the Lax pair \eqref{lax} at $t = 0$, $\mu_2$ satisfies
\begin{equation} \label{Lax_mu2xpart}
  \mu_{2x} +  i k [\sigma_3,\mu_2] = \mathsf{U}_0 \mu_2
\end{equation}
and $\mu_1$ satisfies
\begin{equation}
  \mu_{1x} +  i k \sigma_3 \mu_1 -  i(X(k)-\beta) \mu_1\sigma_3 = \mathsf{U}_0 \mu_1.
\end{equation}

\begin{lemma} \label{lem_symm_phi_1_(E_1,E_2)}
	It holds that
	\begin{equation}\label{phi_identity_(E_1,E_2)}
		[\phi_{1}(x,k)]_{1} =  [\phi_{1}(x,k)]_{2}, \qquad  (x,k)\in \R \times(E_1,E_2).
	\end{equation}
\end{lemma}
\begin{proof}
	The function $\phi_{1}^b(x,0,\cdot)(\phi_{1}^b)^{-1}(x',0,\cdot)\colon \C\to\C$ is entire for each $(x,x')\in \R\times\R$, because as a function of $x$, for fixed $x'$, it satisfies the same linear ODE as $\phi_1^b$ but is normalized to equal $I$ at the finite point $x = x'$. In particular, $\phi_{1}^b(x,0,\cdot)(\phi_{1}^b)^{-1}(x',0,\cdot)$ has no jump across $(E_1,E_2)$. 
	Moreover, using that $s_+^b = s_-^b \sigma_1\sigma_3$, $X_+ = -X_-$, and $\Omega_+=-\Omega_-$ on $(E_1,E_2)$, we infer from \eqref{backgroundefunction} that 
	\begin{equation*}
		[\phi^b_{1+}(x,0,k)]_1 = [\phi^b_{1-}(x,0,k)]_2, \qquad (x,k)\in \R \times(E_1,E_2).
	\end{equation*}
	It follows from \eqref{phi1_Volterra} that the two Volterra equations for the columns $[\phi_{1}(\cdot,k)]_1$ and $[\phi_{1}(\cdot,k)]_2$ are identical for $k \in (E_1,E_2)$.
	This proves the lemma. 
\end{proof}

The next proposition establishes several properties of the spectral functions $a(k)$ and $b(k)$. In particular, it shows that the discrete spectrum is always empty in the case of the boundary conditions \eqref{boundaryconditions} in the sense that $a(k)$ never has zeros.

\begin{proposition}\label{cor_prop_mu}
The spectral functions $a,b$ defined by \eqref{def_s_entries} satisfy
\begin{equation*}
	\begin{cases}
	a(k)  =   \phi_{1,11} \phi_{2,22}  -  \phi_{1,21}  \phi_{2,12}, \\
	\overbar{a(k)}  =  \phi_{1,22}  \phi_{2,11}  -  \phi_{1,12} \phi_{2,21},   \\
	b(k)  =  \phi_{1,22} \phi_{2,12} - \phi_{1,12} \phi_{2,22},   \\
	\overbar{b(k)}  =  \phi_{1,11} \phi_{2,21} -  \phi_{1,21} \phi_{2,11},	
	\end{cases}
	\qquad k\in\R\setminus\{E_1,E_2\},\; x \in\R.
\end{equation*}
In particular, $a,b \in \mathcal{C}^{N_1}(\R\setminus\{E_1,E_2\})$ and $a(k)$ admits an extension to a function in $\mathcal C^{N_1}(\overbar{\C_+}\setminus\{E_1,E_2\})$, again denoted by $a$, which is analytic in $\C_+$, nonzero in $\overbar{\C_+}\setminus\{E_1,E_2\}$, and given by
\begin{equation*}
	a(k) = \det \big([\phi_1(x,k)]_1,[\phi_2(x,k)]_2\big),\qquad k\in\overbar{\C_+}\setminus\{E_1,E_2\},\; x \in\R.
\end{equation*} 
Furthermore, $\det s=|a|^2-|b|^2=1$ for all $k\in\R\setminus[E_1,E_2]$; in particular, $|a|\geq 1$ and $|b|/|a|<1$ for all $k\in\R\setminus[E_1,E_2]$.
\end{proposition}
\begin{proof}
We prove that $a(k)\neq0$ for all $k\in\overbar{\C_+}\setminus\{E_1,E_2\}$; all remaining assertions follow immediately from Proposition \ref{prop_mu} in combination with \eqref{def_s}, \eqref{def_s_entries} and \eqref{defphi120}. 	
Since $|a|\geq 1$ for all $k\in\R\setminus[E_1,E_2]$ it remains to show that $a(k)\neq0$ in $k\in\C_+\cup(E_1,E_2)$.
We treat the cases $k\in\C_+$ and $k\in (E_1,E_2)$ separately; an \emph{argumentum ad absurdum} will be employed for both situations.

First we assume that there exists $\kappa \in \C_+$ such that $a(\kappa) = 0$. Let $L^2(\R, \C^2)$ denote the Hilbert space of vector valued functions $f = (f_1, f_2)$ equipped with the inner product
\begin{equation*}
	\langle f, g\rangle = \int_\R (\bar{f}_1 g_1 + \bar{f}_2 g_2) \,\dx x.
\end{equation*}
Then the operator $L \coloneqq i\begin{pmatrix} \partial_x & -u_0 \\ \bar{u}_0 & -\partial_x \end{pmatrix}$ satisfies
\begin{equation*}
	\langle Lf, g \rangle = \langle f, Lg\rangle \quad \text{whenever} \quad f,g \in H^1(\R, \C^2) \subseteq L^2(\R, \C^2),
\end{equation*}
where we have used the fact that $\lim_{x \to \pm \infty} \phi(x) = 0$ for all functions $\phi$ in the Sobolev space $H^1(\R, \C^2)$ to integrate by parts. 

Define $h\colon \R \to \C^2$ by
\begin{equation*}
	h(x) = [\phi_2(x,\kappa)]_2 = [\mu_2(x,\kappa)]_2 \e^{i\kappa x}.
\end{equation*}
In the following we show that $Lh = \kappa h$ in $L^2(\R, \C^2)$.
Since $h$ solves the $x$-part of the Lax pair \eqref{lax} at $t=0$ with $k=\kappa$, it suffices to show that $h\in H^1(\R, \C^2)$.

Proposition \ref{prop_mu}\eqref{prop_mu_iii} implies that $h$ has exponential decay as $x \to + \infty$.
Furthermore, by employing that
\begin{equation*}
		a(k) = \det \big([\phi_1(x,k)]_1,[\phi_2(x,k)]_2\big),\qquad k\in\overbar{\C_+}\setminus\{E_1,E_2\},\; x \in\R,
\end{equation*} 
the assumption $a(\kappa) = 0$ implies that there is a constant $c \in \C$ such that
\begin{equation*}
	h(x) = c [\phi_1(x,\kappa)]_1 = c[\mu_1(x,\kappa)]_1 \e^{-i(X(\kappa) - \beta)x}, \qquad x \in \R.
\end{equation*}
Hence, using $\im X(\kappa)>0$ and employing Proposition \ref{prop_mu}\eqref{prop_mu_iii} yields exponential decay of $h$ as $x \to -\infty$.
	
This shows that $h \in H^1(\R, \C^2)$ and we conclude that $\kappa$ is an eigenvalue of $L$ with eigenfunction $h$ in $L^2(\R, \C^2)$. Since $L$ is self-adjoint, this yields that $\kappa$ is real:
\begin{equation*}
	\bar{\kappa} \langle h, h \rangle = \langle Lh, h \rangle =\langle h, Lh \rangle = \kappa \langle h, h \rangle,
\end{equation*}
which contradicts the initial assumption that $\kappa\in\C_+$.

It only remains to prove that $a(k) \neq 0$ for $k \in (E_1, E_2)$. 
Employing \eqref{phi_identity_(E_1,E_2)} yields that
\begin{align}\label{aequalsminusb}
	a(k)  =   \phi_{1,11} \phi_{2,22}  -  \phi_{1,21}  \phi_{2,12}
	= \phi_{1,12} \phi_{2,22}  -  \phi_{1,22}  \phi_{2,12}
	= -b(k), \qquad k \in (E_1, E_2).
\end{align}
Let us suppose that $a(\kappa) = 0$ for some $\kappa \in (E_1, E_2)$. Then $b(\kappa) = 0$ by \eqref{aequalsminusb}. 
Since $\phi_2(x,\kappa)$ has unit determinant for all $x \in \R$, \eqref{def_s_entries} implies that $\phi_1(x,\kappa)$ vanishes for all $x \in \R$. 
But then
\begin{equation*}
	\mu_1(x, \kappa) = \phi_1(x,\kappa) \e^{i(X_+(\kappa) - \beta)x\sigma_3}
\end{equation*}
also vanishes for all $x \in \R$, which contradicts Proposition \ref{prop_mu}\eqref{prop_mu_iii},
because $s^b(k) \neq 0$ for all $k \in (E_1, E_2)$. 
Consequently, $a(k) \neq 0$ for $k \in (E_1, E_2)$. 
\end{proof}

\subsection{Asymptotics of $r$ as $k\to E_j$}
We begin with introducing a useful representation for $\mu_1$.
Let $E\colon \R\times \R\times \C\setminus (E_1,E_2)\to \C^{2\times2}$ be defined by
\begin{equation} \label{def_E}
E(x,x',k) \coloneqq  \e^{ i (X(k)-\beta) (x-x')} \phi_1^b(x,0,k)(\phi_1^b)^{-1}(x',0,k),	
\end{equation}
which is analytic in $\C\setminus [E_1,E_2]$ with continuous boundary values on $(E_1,E_2)$. 
For $x\in\R$, we  set
\begin{equation*}
	H(x) \coloneqq \begin{pmatrix}
		1 & 0 \\ 0 & \e^{-2 i\beta x} \end{pmatrix}.
\end{equation*} 
The Volterra equation for $[\mu_1]_1$ reads, for $(x,k)\in \R \times \big(\overbar{\C_+}\setminus\{E_1,E_2\}\big)$,
\begin{equation*}
[\mu_1(x,k)]_1 = H(x) [s^b(k)]_1 + \int_{-\infty}^x E(x,x',k) (\mathsf{U}_0 - \mathsf{U}_0^b)(x') [\mu_1(x',k)]_1 \,\dx x',	
\end{equation*}
where it is understood that $s^b = s_+^b$ and $E = E_+$ if $k\in(E_1,E_2)$. Thus, defining $\Lambda \colon \R\times ( \overbar{\C_+}\setminus\{E_1,E_2\}) \to\C^{2\times2}$ by  
\begin{align}\label{LambdaVolterra}
	\Lambda(x,k) = H(x) + \int_{-\infty}^x E(x,x',k) (\mathsf{U}_0 - \mathsf{U}_0^b)(x') \Lambda(x',k) \,\dx x',
\end{align}
where $E = E_+$ if $k\in(E_1,E_2)$, we see that
\begin{equation} \label{Lambda}
[\mu_1]_1 = \Lambda [s^b]_1, \qquad x \in \R, ~ k \in  \overbar{\C_+}\setminus\{E_1,E_2\},
\end{equation}
where $s^b = s^b_+$ if $k \in (E_1, E_2)$.

\begin{lemma} \label{lem_Lambda}
For every $x\in\R$, $\Lambda(x,\cdot)$ has a continuous extension $\Lambda(x,\cdot)\colon \overbar{\C_+}\to \C^{2\times2}$, which 
is analytic in $\C_+$, and its restriction to $\R$ lies in $\mathcal C^{N_1}(\R\setminus\{E_1,E_2\},\C^{2\times2})\cap \mathcal C(\R,\C^{2\times2})$. Additionally, 
\begin{equation} \label{Lambda_at_E_j}
	\Lambda(x,k) = \Lambda(x,E_j) + \sum^{N_1-1}_{l=1} \Lambda^{(l/2)}_j(x) (k-E_j)^{l/2} 
	+ o\big((k-E_j)^{\frac{N_1 -1}{2}}\big)
\end{equation}
as $\overbar{\C_+} \ni k \to E_j$ for coefficients $\Lambda^{(l/2)}_j(x)\in\C$, $j=1,2$, $l=1,\dots,N_1-1$. 
\end{lemma}
\begin{proof}
It follows from Proposition \ref{prop_mu}\eqref{prop_mu_i} in combination with \eqref{Lambda} and the definition of $s^b$ that $\Lambda(x,\cdot)$ is analytic in $\C_+$ and that the restriction $\Lambda(x,\cdot)|_{\R\setminus\{E_1,E_2\}}$ lies in $\mathcal C^{N_1}(\R\setminus\{E_1,E_2\},\C^{2\times2})$ for each $x\in\R$.

Next we show that $\Lambda(x,\cdot)$ has a continuous extension to $E_j$ for each $x\in\R$, $j=1,2$, and that \eqref{Lambda_at_E_j} holds.
Let $F\colon \R\times \R\times (\C\setminus (E_1,E_2))\to \C^{2\times2}$ be defined by
\begin{equation*}
	F(x,x',k)
	\coloneqq  \begin{pmatrix}
		\frac{(E_1+X(k)-k)^2  \e^{2  i X(k) (x-x')} -(k+X(k)-E_1)^2}{4  (E_1-k)X(k)}
		& \frac{ i (\e^{2  i X(k) (x-x')}-1) ((E_1-k)^2-X(k)^2)}{4  (E_1-k)X(k)} \\
		\frac{ i (\e^{2  i X(k) (x-x')}-1) ((E_1-k)^2-X(k)^2)}{4 (E_1-k)X(k)} 
		& \frac{(E_1+X(k)-k)^2 - (k+X(k)-E_1)^2 \e^{2  i X(k)(x-x')} }{4  (E_1-k) X(k)} 
	\end{pmatrix}.
\end{equation*} 
A direct calculation shows that
\begin{align*}
	\begin{aligned}
		E(x,x',k)  & =  \e^{ i X(k) (x-x')}\e^{- i\beta (x-x')} \e^{ i \beta x \sigma_3} s^b(k) \e^{- i X(k) (x-x')\sigma_3} s^b(k)^{-1} \e^{- i \beta x' \sigma_3}	\\
		&= H(x) s^b(k) 
		\begin{pmatrix}1 & 0 \\ 0 &  \e^{2 i X(k) (x-x')} \end{pmatrix}
		s^b(k)^{-1} H^{-1}(x')	\\
		&= H(x) F(x,x',k) H^{-1}(x').
	\end{aligned}
\end{align*}
Fix $j\in\{1,2\}$, let $0<\epsilon<(E_2-E_1)/2$ and consider the region
\begin{equation*}
	U_\epsilon \coloneqq B_{\sqrt\epsilon}(0) \cap Q_1 \subseteq \C,
\end{equation*}
where $Q_1 \coloneqq \{z\in\C\colon \re z > 0, \im z > 0\}$ denotes the first quadrant in the complex pane. Consider the change of variables 
\begin{equation*}
	\lambda\colon \overbar{B_\epsilon(E_j)\cap\C_+} \to \overbar{U_\epsilon}, \qquad k\mapsto\lambda(k)\coloneqq \sqrt{k - E_j}.
\end{equation*}
For $x,x'\in\R$, $E(x, x', \cdot)$ is a smooth function of $\lambda \in \overbar{U_\epsilon}$. 
Moreover, 
\begin{equation} \label{E_estimate}
	|E(x,x',k)| \leq C(1 + |x-x'|), \qquad -\infty < x' \leq x, \; k \in \lambda^{-1}( \overbar{U_\epsilon}).
\end{equation}
To see this, we define 
\begin{equation*}
	f\colon \R\times\lambda^{-1}( \overbar{U_\epsilon})\to\C^{2\times2}, \quad f(x,k)\coloneqq F(x,0,k),	
\end{equation*}
so that $F(x,x',k)=f(x-x',k)$.
The estimate in \eqref{E_estimate} will follow if we can find a constant $C>0$ such that
\begin{equation}\label{f_estimate}
	|f(x,k)|\leq C(1+x), \qquad x\geq0, \; k\in \lambda^{-1}( \overbar{U_\epsilon}).
\end{equation}
Since 
\begin{align}\label{f_x}
	f_x(x,k) = -\frac{\e^{2 i X(k)x}}{2 i}
	\begin{pmatrix}
		2( X(k)- k) +E_1 + E_2   &  i(E_1-E_2)  \\
		 i(E_1-E_2) &   2 (X(k)+k) - E_1 - E_2 \\
	\end{pmatrix},
\end{align}
we see that $|f_x(x,k)|$ is uniformly bounded for all $x\geq 0$ and $k \in\lambda^{-1}( \overbar{U_\epsilon})$.
Since $f(x,k) = I + \int_0^x f_x(x',k) \,\dx x'$, this shows \eqref{f_estimate} and hence also \eqref{E_estimate}. 

With the help of \eqref{E_estimate} and the decay assumption \eqref{u0assumptions} on $u_0-u^b_0$ we can define 
\begin{align}\label{Lambdal}
	\Lambda_l(x,k)= \int\limits_{-\infty < x_1 \leq \cdots \leq x_l \leq x_{l+1} = x}\prod_{j=1}^l E(x_{j+1},x_{j},k) (\mathsf{U}_0-\mathsf{U}^b_0)(x_j) H(x_1) \,\dx x_1\cdots \dx x_l.
\end{align}
for integers $l\geq1$,
and estimate  that
\begin{align*}
	|\Lambda_l(x,k)| 
	&\leq \int\limits_{-\infty < x_1 \leq \cdots \leq x_l \leq x_{l+1} = x}\prod_{j=1}^l |E(x_{j+1},x_{j},k)| |(\mathsf{U}_0-\mathsf{U}_0^b)(x_j)|\,\dx x_1\cdots \dx x_l \\
	&\leq C \int\limits_{-\infty < x_1 \leq \cdots \leq x_l \leq x_{l+1} = x}\prod_{j=1}^l (1+ |x -x_{j}|) |(u_0-u^b_0)(x_j)|\,\dx x_1\cdots \dx x_l \\
	&\leq \frac{C \left(\int^x_{-\infty}(1+|x-x'|)|(u_0-u^b_0)(x')| \,\dx x' \right)^l}{l!} 
\end{align*}
for all $x\in \R$ and $k \in\lambda^{-1}( \overbar{U_\epsilon})$.
Consequently, by the assumption \eqref{u0assumptions} on $u_0$, the Neumann series 
\begin{equation*}
	\Lambda(x,k) = \sum_{l=0}^\infty \Lambda_l(x,k)
\end{equation*}
with $\Lambda_0(x)\coloneqq H(x)$ converges absolutely and uniformly for all $k \in\lambda^{-1}( \overbar{U_\epsilon})$:
\begin{align*}
	|\Lambda(x,k)| & \leq \sum_{l=0}^\infty |\Lambda_l(x,k)|
	\leq C \exp\!\bigg(\int^x_{-\infty}(1+|x-x'|)|(u_0-u^b_0)(x')| \,\dx x'\bigg).
\end{align*}
It follows that $\Lambda(x,\cdot)$ possesses a continuous extension to $\overbar{\C_+}$ for each $x\in\R$.

By taking the $\lambda$-derivative of \eqref{f_x}, we observe that there exists a $C>0$ such that 
\begin{equation*}
	|f_{x\lambda}(x,k)| \leq C(1+x), \qquad x \geq 0, \;k\in \lambda^{-1}( \overbar{U_\epsilon}).
\end{equation*}
Using that $f_\lambda(x,k) = \int_0^x f_{x\lambda}(x',k) \,\dx x'$, we obtain that
\begin{equation*}
	|f_\lambda(x,k)| \leq C(1+x)^2, \qquad x \geq 0, \;  k\in \lambda^{-1}( \overbar{U_\epsilon}).	
\end{equation*}
and hence, for each fixed $x\in\R$ and all $-\infty < x' \leq x$,
\begin{align*}
	|\partial_\lambda E(x, x', k)
	& \leq C (1 + |x - x'|)^2, \qquad k\in \lambda^{-1}( \overbar{U_\epsilon}).
\end{align*}
This estimate together with the assumption that  $u_0$ satisfies \eqref{u0assumptions} for some $N_1 \geq 2$ enables us to interchange the order of integration and differentiation when differentiating $\Lambda_l$ in \eqref{Lambdal} with respect to $\lambda$. 
We obtain that each $\Lambda_l$, $l\geq1$, is a $\mathcal{C}^1$-function of $\lambda\in\overbar{U_\epsilon}$. 
By induction, we infer that each $\Lambda_l$ is $N_1-1$ times continuously differentiable on $\lambda\in\overbar{U_\epsilon}$. 
Indeed, since $x^n (u_0-u^b_0)(x)$ is in $L^1(-\infty,0)$ for each $n = 0, \dots, N_1$, and
\begin{align*}
	|\partial^n_\lambda E(x, x', k) | \leq C (1 + |x - x'|)^{n+1}, \qquad k\in \lambda^{-1}( \overbar{U_\epsilon}), 
\end{align*}
we find for all $1\leq n \leq N_1-1$ suitable integrable majorants and obtain series expansions for the $n$-th derivative $\partial^n_\lambda \Lambda_l(x,\cdot)\colon\overbar{U_\epsilon}\cap\R\to\C$, $l\geq1$, $x\in\R$.
These estimates also show that the differentiated series $\sum_{l=0}^\infty \partial^n_\lambda \Lambda_l(x,\cdot)$ converges absolutely and uniformly in $\overbar{U_\epsilon}$ for each $1\leq n \leq N_1-1$, showing that $\Lambda(x,\cdot)$  is $N_1-1$ times continuously differentiable on $\overbar{U_\epsilon}$. 
Taylor expanding at $\lambda = 0$ up to order $N_1-1$ and applying the transformation $\lambda \mapsto k$, we obtain \eqref{Lambda_at_E_j} with the stated error term.	
\end{proof}

\begin{corollary} \label{cor_lem_Lambda}
At $E_j$, $j=1,2$, the eigenfunction $\mu_1(x,\cdot)$, $x\in\R$, and spectral function $a$ satisfy
\begin{align}
	\mu_1(x,k) &= \hat{\mu}_{1,j}(x) (k-E_j)^{-1/4}+ \bigO((k-E_j)^{1/4}), \qquad (\overbar{\C_+},\overbar{\C_-})\setminus \{E_1,E_2\}\ni k \to E_j, \label{mu_1_at_E_j} \\	
	a(k) &= \hat{a}_j (k-E_j)^{-1/4}  + \bigO((k-E_j)^{1/4}), \qquad\qquad \overbar{\C_+}\setminus \{E_1,E_2\}\ni k \to E_j, \label{a_at_E_j}
\end{align}
for coefficients $\hat{\mu}_{1,j}(x)\in \C^{2\times2}$ and $\hat{a}_j \in\C\setminus \{0\}$.
\end{corollary}
\begin{proof}
It holds that $s^b(k) = \hat{s}_j^b(k-E_j)^{-1/4} + \bigO((k-E_j)^{1/4})$ as $k \to E_j$ for some constant matrices $\hat{s}_j^b$, $j = 1,2$.
Employing \eqref{Lambda} and \eqref{Lambda_at_E_j} yields that for $x\in\R$,
\begin{equation*}
	[\mu_1(x,k)]_1 = [\hat{\mu}_{1,j}(x)]_1 (k-E_j)^{-1/4} + \bigO\big((k-E_j)^{1/4}\big), \qquad \overbar{\C_+}\setminus\{E_1,E_2\} \ni k \to E_j, ~ j = 1,2,
\end{equation*}
with $[\hat{\mu}_{1,j}(x)]_1\in\C^2$. The symmetry \eqref{symmetry_mu_1} implies that $[\mu_1(x,k)]_2$ has a similar expansion as $k$ approaches $E_j$ from $\overbar{\C_-}\setminus\{E_1,E_2\}$, $j = 1,2$. This shows \eqref{mu_1_at_E_j}. 

Since $s(k) = \mu^{-1}_2(0,k)\mu_1(0,k)$ for $k\in\R\setminus\{E_1,E_2\}$ and $\mu^{-1}_2(0,k)\in \mathcal C^{N_1}(\R,\C^{2\times2})$ by Proposition \ref{prop_mu}, we infer that there exist coefficients $\hat{a}_1, \hat{a}_2 \in \C$ such that 
\begin{equation*}
	a(k) = \hat{a}_j (k-E_j)^{-1/4} + \bigO\big((k-E_j)^{1/4}\big) \qquad \overbar{\C_+}\setminus\{E_1,E_2\} \ni k \to E_j, ~ j = 1,2.
\end{equation*}
Since $|a(k)| \geq 1$ for $k\in\R\setminus[E_1,E_2]$, it follows that $\hat{a}_j \neq 0$, $j = 1,2$, which confirms \eqref{a_at_E_j}. 
\end{proof}

\subsection{Proof of Theorem~\ref{thm_direct}\eqref{thm_direct_a}}
Proposition \ref{cor_prop_mu} implies that $r\in\mathcal C^{N_1}(\R\setminus \{E_1,E_2\})$. 

Let us consider the behavior of $r$ at the branch points. The definition \eqref{def_r} of $r$ in conjunction with \eqref{defphi120} and Proposition \ref{cor_prop_mu} implies that
\begin{align*}
	&r(k) = 
	\frac{\mu_{1,11}(0,k)  \mu_{2,21}(0,k)  - \mu_{1,21}(0,k)  \mu_{2,11}(0,k)}
	{\mu_{1,11}(0,k)  \mu_{2,22}(0,k)   -  \mu_{1,21}(0,k)   \mu_{2,12}(0,k)},
	\qquad \phantom{-}  k\in \R\setminus\{E_1,E_2\}, 
\end{align*}
where we have used \eqref{phi_identity_(E_1,E_2)} for $k \in (E_1, E_2)$. 
Using \eqref{Lambda} and employing the definition \eqref{def_s^b} of $s^b$, we deduce that, for all $k\in\R\setminus\{E_1,E_2\}$,
\begin{equation}\label{r_identity}
	r(k) = 
	\frac{(\Lambda_{11} (\Delta_+^2+1) - i (\Delta_+^2-1) \Lambda_{12} )\mu_{2,21} 
		- (\Lambda_{21} (\Delta_+^2+1) - i (\Delta_+^2-1) \Lambda_{22} ) \mu_{2,11}}
	{(\Lambda_{11} (\Delta_+^2+1) - i  (\Delta_+^2-1)\Lambda_{12}) \mu_{2,22} 
		- (\Lambda_{21} (\Delta_+^2+1) - i (\Delta_+^2-1) \Lambda_{22})\mu_{2,12}}\bigg|_{x=0}.
\end{equation}
Recall from Proposition \ref{prop_mu}\eqref{prop_mu_ii} that $\mu_2 \in\mathcal C^{N_1}(\R,\C^{2\times2})$. Therefore, it follows from \eqref{Deltadef}, \eqref{Lambda_at_E_j}, and \eqref{r_identity} that
$r$ has expansions of the form
\begin{equation*}
r(k) =
\begin{cases}
			\sum^{N_1-1}_{l=-1} q_{2,l} (k-E_2)^{l/2} + o\big((k-E_2)^{\frac{N_1-1}{2}}\big)   &\text{as} \quad k\searrow E_2, \\
			\sum^{N_1-1}_{l=-1}  p_{2,l} (E_2-k)^{l/2} + o\big((E_2 -k)^{\frac{N_1-1}{2}}\big)   &\text{as}\quad k\nearrow E_2, \\
			\sum^{N_1-1}_{l=-1} i^l q_{1,l} (k-E_1)^{l/2} + o\big((k-E_1)^{\frac{N_1-1}{2}}\big)  &\text{as}\quad k\searrow E_1, \\
			\sum^{N_1-1}_{l=-1} i^l p_{1,l} (E_1-k)^{l/2} + o\big((E_1-k)^{\frac{N_1-1}{2}}\big)   &\text{as}\quad k\nearrow E_1, \\
\end{cases}
\end{equation*}
for some coefficients $q_{j,l}, p_{j,l} \in \C$. Since $\Lambda$ and $\Delta_+$ have analytic continuations to $\C_+$ and $\mu_2$ is $C^{N_1}$ at $k \in E_j$, we conclude that the coefficients obey the relation $p_{j,l} = i^l q_{j,l}$. From the definition of $r$ it is clear that $|r|=1$ on the interval $(E_1,E_2)$. Consequently, the coefficients of the singular terms $|k-E_j|^{-1/2}$, $j = 1,2$, must vanish. 
In particular, $r \in \mathcal C(\R)$. 
Using again that $|r|=1$ on $(E_1,E_2)$, we find that $q_{1,0} = r(E_1)$ and $q_{2,0} = r(E_2)$ have absolute value $1$. More generally, considering the coefficient of $|k-E_j|^{n/2}$ in the expansion of $|r|^2 = 1$ as $k \in (E_1, E_2)$ approaches $E_j$, we find that $\sum_{l=0}^n  i^{n-l} (- i)^l q_{j,n-l} \overbar{q_{j,l}} =0$ for $j=1,2$ and $n = 0,1, \dots, N_1-1$. 

It remains to prove that $q_{j,1} \neq 0$, $j=1,2$. In the following we demonstrate that $q_{2,1} \neq 0$; the case $j=1$ is very similar. The unit determinant relation 
$|a|^2 - |b|^2 = 1$ for $k  \in \R \setminus [E_1, E_2]$ yields that
\begin{equation*}
	|a|^2 = \frac{1}{1- |r|^2}, \qquad k  \in \R \setminus [E_1, E_2].
\end{equation*}
Thus, \eqref{a_at_E_j} implies that there exists a small $\epsilon > 0$ and a constant $c_1 > 0$ such that
\begin{equation*}
\frac{1}{1- |r|^2} \leq \frac{c_1}{\sqrt{k-E_2}}, \qquad k \in (E_2, E_2 + \epsilon),
\end{equation*}
which implies that
\begin{equation}\label{pr_q_{2,1}_1}
	|r|^2 \leq 1- \frac{1}{c_1} \sqrt{k-E_2}, \qquad k \in (E_2, E_2 + \epsilon).
\end{equation}
Since $N_1 \geq 2$, we have $r(k) = q_{2,0} + q_{2,1} (k-E_2)^{1/2} + o((k-E_2)^{1/2})$ as $k\searrow E_2$ with $|q_{2,0}|=1$, which is consistent with \eqref{pr_q_{2,1}_1} only if $q_{2,1} \neq 0$. 
\proofend

\subsection{Proof of Theorem~\ref{thm_direct}\eqref{thm_direct_b}}
Suppose $u_0$ satisfies \eqref{u0assumptions} and \eqref{u0assumptions2} for some integers $N_1 \geq 2$ and $0 \leq N_2 \leq 4$.
Define
\begin{align}\label{mu_1p}
	\mu_{1,N_2}(x,k) &\coloneqq I + \sum^{N_2}_{n=1}\frac{\mu^{(n)}_{1}(x)}{k^n}, \qquad \mu_{2,N_2}(x,k) \coloneqq I + \sum^{N_2}_{n=1}\frac{\mu^{(n)}_{2}(x)}{k^n},
\end{align} 
where the first few  coefficients $\mu^{(n)}_j$ are given by  
\begin{align*}
	\begin{aligned}
	\mu^{(1)}_{1} &= \frac{1}{2} 
	\begin{pmatrix}
		 i\Gamma_1 & - i u_0\\  i \bar{u}_0 &- i\Gamma_1	
	\end{pmatrix}, 
	\quad 	\mu^{(1)}_{2} = \frac{1}{2}
	\begin{pmatrix}
		 i\Lambda_1&- i u_0 \\ i\bar{u}_0 &- i\Lambda_1
	\end{pmatrix},  
		\\
	\mu^{(2)}_{1} &= \frac{1}{4} 
	\begin{pmatrix}
		u_0^b \bar{u}_0 - \frac{\alpha^2}{2} + \Gamma_2& u_{0x} - u_0\Gamma_1 \\ 
		\bar{u}_{0x} - \bar{u}_0 \Gamma_1 & \overbar{u_0^b} u_0 - \frac{\alpha^2}{2} + \overbar{\Gamma_2}
	\end{pmatrix},
 	\quad	\mu^{(2)}_{2} = \frac{1}{4}
 	\begin{pmatrix}
 		\Lambda_2&u_{0x} - u_0\Lambda_1\\ \bar{u}_{0x} - \bar{u}_0 \Lambda_1& \overbar{\Lambda_2}
 	\end{pmatrix}, 
		\\ 
(\mu^{(3)}_{1})_{11} &= \frac{1}{8} \Big(\frac{\alpha^2}{2}(4\beta- i\Gamma_1) -  i u_0^b(\bar{u}_{0x} - 2 i\beta\bar{u}_0-\bar{u}_0 \Gamma_1) +  i\Gamma_3\Big),
	\quad (\mu^{(3)}_{1})_{22} = \overbar{(\mu^{(3)}_{1})_{11}}, 
	\\ 
	(\mu^{(3)}_{1})_{12} &= -\frac{ i}{8} \Big(u_{0x} \Gamma_1 - u_{0x x} + \frac{u_0}{2} (2u_0 (\bar{u}_0 +\overbar{u_0^b})-\alpha^2+2\overbar{\Gamma_2})\Big), 
	\quad	(\mu^{(3)}_{1})_{21} = \overbar{(\mu^{(3)}_{1})_{12}}, 
	\\ 
		\mu^{(3)}_{2} &=  \frac{1}{8}
	\begin{pmatrix}
		 i\Lambda_3 &  i\big(u_{0x x} - u_{0x} \Lambda_1 - u_0(|u_0|^2 + \overbar{\Lambda_2})\big) \\ - i(\bar{u}_{0xx} - \bar{u}_{0x} \Lambda_1 - \bar{u}_0(|u_0|^2 + \Lambda_2)) & - i\overbar{\Lambda_3}		
	\end{pmatrix},
	\end{aligned}
\end{align*}
with
\begin{align*}
	\Gamma_1(x) &= \int^x_{-\infty} (|u_0|^2-\alpha^2) \,\dx x', \quad \Lambda_1(x) = -\int^\infty_x |u_0|^2 \,\dx x',
		\\
\Gamma_2(x) & = \int^x_{-\infty} \Big(  2 i\beta(\alpha^2 - u_0^b\bar{u}_0)  + (\alpha^2 - |u_0|^2)\Gamma_1  +(u_0 - u_0^b)\bar{u}_{0x}  \Big) \,\dx x', 
		\\
\Lambda_2(x) & = -\int^\infty_x u_0 (\bar{u}_{0x} -\bar{u}_0 \Lambda_1) \,\dx x',
	\\
\Gamma_3(x) &= \int^x_{-\infty} \Big(|u_0|^4-\alpha^4 + (u_0^b \bar{u}_0 - \alpha^2)(4\beta^2-2 i\beta\Gamma_1) 
	\\  
& \quad +(u_0-u_0^b) (\bar{u}_{0x} \Gamma_1 -\bar{u}_{0x x})   + (|u_0|^2-\alpha^2) \Gamma_2  \Big) \,\dx x',
	\\
\Lambda_3(x) &= -\int^\infty_x u_0 \Big(  \bar{u}_{0x} \Lambda_1 + \bar{u}_0(|u_0|^2 + \Lambda_2) -\bar{u}_{0x x}\Big) \,\dx x'.
\end{align*}

\begin{proposition} \label{prop_mu_j_large_k}
There exist bounded positive functions $f_\pm \in \mathcal C(\R)$ whose restrictions $f_\pm|_{\R_\pm}$ lie in $L^1(\R_\pm)$ and a constant $K>0$ such that, for $j = 0, \dots, N_1$,
\begin{subequations}\label{mu1mu2atinfinity}
\begin{align} 
	&\big|\partial^j_k\big(\mu_1(x,k) - \mu_{1,N_2}(x,k)\big)\big|\leq \frac{f_-(x)}{|k|^{N_2+1}}, 
	\qquad x\in\R,\; k\in(\overbar{\C_+},\overbar{\C_-}),\; |k|\geq K, 
		\\
	&\big|\partial^j_k\big(\mu_2(x,k) - \mu_{2,N_2}(x,k)\big)\big|\leq \frac{f_+(x)}{|k|^{N_2+1}}, 
	\qquad x\in\R, \; k\in(\overbar{\C_-},\overbar{\C_+}),\; |k|\geq K. 
\end{align}
\end{subequations}
\end{proposition}
\begin{proof}
We briefly discuss the formal derivation of the coefficients $\mu^{(n)}_j$ and omit the corresponding tedious yet well-known derivation of the estimates in \eqref{mu1mu2atinfinity}). 
A rigorous Neumann series approach giving analytical meaning to these formal computations can be found e.g.~in \cite{Lenells16} for the large $k$ behavior of  mKdV eigenfunctions (of both Lax pair equations) to all orders; we refer to~\cite{LQ20} for the leading order asymptotics of $t$-periodic NLS eigenfunctions.	

To determine the coefficients $\mu^{(n)}_2$ we make the formal ansatz
\begin{equation}\label{mu2_k_infty}
	\mu_2(x,k) \sim I + \sum^\infty_{n=1}\frac{\mu^{(n)}_2(x)}{k^n} \qquad x,k\in\R,\;k\to\infty.
\end{equation}
Recalling that $\mu_2$ satisfies \eqref{Lax_mu2xpart}, this ansatz yields the following recursion for the coefficients $\mu^{(n)}_2$:
\begin{align*}
	\mu^{(0)}_2 &\coloneqq I \\
	\big(\mu^{(n+1)}_2\big)^\op  &\;= \frac{ i}{2} \sigma_3 \Big( \partial_x \big(\mu^{(n)}_2\big)^\op -\mathsf{U}_0 \big(\mu^{(n)}_2\big)^\dx \Big), \\
	\partial_x \big(\mu^{(n+1)}_2\big)^\dx &\;= \mathsf{U}_0 \big(\mu^{(n+1)}_2\big)^\op,
\end{align*}
where $A^\dx$ and $A^\op$ denote the diagonal and off-diagonal part of $A\in\C^{2\times2}$. Solving this system recursively via integration from $x$ to $+\infty$ determines the coefficients $\mu^{(n)}_2$; according to the assumptions on $u_0$ they are well-defined up to order $n=N_2+1$.

To determine the coefficients $\mu^{(n)}_1$, some preliminary considerations are necessary. 
We first define $\mu_1^b$ by 
\begin{align*}
\mu_1^b(x,k) = \e^{ i \beta x \hat \sigma_3} s^b(k),
\end{align*}
so that
\begin{align*}
\phi_1^b(x,k) = \mu_1^b(x,k) \e^{- i(X(k)-\beta)x \sigma_3}.
\end{align*}
Setting $\nu \coloneqq (\mu_1^b)^{-1}\mu_1$ and recalling the Volterra integral equation \eqref{defmu10} for $\mu_1$, we obtain that 
\begin{align*}
	\nu =  I
	+\int_{-\infty}^x
	\e^{- i(X(k)-\beta)(x-x') \sigma_3}(\mu_1^b)^{-1}(x',k)(\mathsf{U}-\mathsf{U}^b)(x')
	\mu_1^b(x',k) \nu(x',k) \e^{ i (X(k)-\beta) (x-x')\sigma_3} \,\dx x'.
\end{align*}
The corresponding differential equation for $\nu$ reads
\begin{align}\label{nu_eq}
	\nu_x +  i(X(k)-\beta) [\sigma_3, \nu]  = (\mu_1^b)^{-1} (\mathsf{U}-\mathsf{U}^b) \mu_1^b \nu.
\end{align}
By plugging the ansatz
\begin{equation*}
	\nu(x,k) \sim I + \sum^\infty_{n=1}\frac{\nu^{(n)}(x)}{k^n}, \qquad x \in\R,\; k\to\infty
\end{equation*}
into \eqref{nu_eq} and expanding the functions $X$, $\mu^b_1$, and $(\mu^b_1)^{-1}$ as $k \to \infty$, the coefficients $\nu^{(n)}$ can be recursively determined via integration from $-\infty$ to $x$.
Calculating the coefficients $\nu^{(n)}$, $1\leq n\leq N_2$, in the described manner and setting  
\begin{equation*}
	\nu_{N_2}(x,k) = I + \sum^{N_2}_{n=1}\frac{\nu^{(n)}(x)}{k^n}
\end{equation*}
confirms that 
\begin{equation*}
	\mu_{1,N_2}(x,k)=\mu^b_1(x,k) \nu_{N_2}(x,k)+\bigO(k^{-N_2-1}), \qquad x \in\R,\;k\to\infty,
\end{equation*}
with $\mu_{1,N_2}$ given by \eqref{mu_1p}. 
\end{proof}

\begin{remark}
Proposition \ref{prop_mu_j_large_k} can be extended to the case of arbitrary finite $N_2 \geq 0$. 
\end{remark}

Theorem \ref{thm_direct}\eqref{thm_direct_b} is a direct consequence of the following corollary.
\begin{corollary} \label{cor_s_large_k}
	The entries $a$  and $b$ of the scattering matrix $s$ satisfy
	\begin{equation} \label{cor_s_large_k_0}
		\begin{cases}
		\partial^j_k a(k) = \partial^j_k  \Big(1 + \sum^{N_2}_{n=1} a_n k^{-n} \Big) +  \bigO(k^{-N_2-1}), &\overbar{\C_+}\ni k\to\infty, \\
		\partial^j_k b(k) = \bigO(k^{-N_2-1}), &\R\ni k\to\infty,
	\end{cases}
		\qquad 0\leq j\leq N_1, 
	\end{equation}
	where $\{a_n\}_1^{N_2}$ are complex coefficients. The first three of the coefficients $a_n$ are given by
	\begin{align}
		\begin{aligned}\label{cor_s_large_k_1}
		a_1&=\frac{ i(\Gamma_1-\Lambda_1)}{2}, \\
		a_2&=\frac{\alpha\big(\frac{\alpha}{2}-u\big)-(\Gamma^2_1+\overbar{\Gamma_2}-\Gamma_1\Lambda_1-\overbar{\Lambda_2})}{4}, \\
		a_3&=\frac{ i}{8}\bigg( \alpha \Big(\frac{\alpha}{2}(\Gamma_1-\Lambda_1)+2 i\alpha\beta-u(2 i\beta+\Gamma_1-\Lambda_1)-u_x\Big)
			\\ 
		&\qquad\quad-\Gamma^3_1-2\Gamma_1\overbar{\Gamma_2}+\overbar{\Gamma_3}+\Gamma^2_1\Lambda_1+\overbar{\Gamma_2}\Lambda_1+\Gamma_1\overbar{\Lambda_2}-\overbar{\Lambda_3}\bigg).		
		\end{aligned}
	\end{align}
\end{corollary}
\begin{proof}
Due to the decay assumption \eqref{u0assumptions} on $u_0$, we infer that $s\in\mathcal C^{N_1}(\R\setminus\{E_1,E_2\},\C^{2\times 2})$.
According to \eqref{def_s}, \eqref{defphi120}, and Proposition \ref{prop_mu_j_large_k} it holds that 
\begin{align*}
s(k) 
= \mu_1(0,k)^{-1}\mu_2(0,k)
= \big((\mu_{1,N_2})^{-1} \mu_{2,N_2}\big)(0,k) + \bigO(k^{-N_2-1}),\qquad k\to\infty,
\end{align*}
where the $(2,2)$-entry is valid as $k\to\infty$ with $k \in \overbar{\C_+}$ and the $(1,2)$-entry is valid as $k \to\infty$ with $k \in \R$.
Substituting in the explicit expressions for $\mu_{1,N_2}$ and $\mu_{2,N_2}$ defined in \eqref{mu_1p}, we find 
\eqref{cor_s_large_k_0} for $j=0$ with $a_n$, $n=1,2,3$, given by \eqref{cor_s_large_k_1}. 
As a consequence of Proposition \ref{prop_mu_j_large_k}, the expansions in \eqref{cor_s_large_k_0} can be differentiated up to $N_1$ times termwise without worsening the error terms.
\end{proof}

\subsection{Proof of Theorem \ref{thm_direct}\eqref{thm_direct_c}}
Suppose $u_0$ satisfies \eqref{u0assumptions} and \eqref{u0assumptions2} for some integers $N_1 \geq 2$ and $N_2 \geq 0$.
We will show that $m_0(x,\cdot)$, $x\in\R$, defined by \eqref{def_m_0}, satisfies the stated RH problem. 

It follows from the Propositions \ref{prop_mu} and \ref{cor_prop_mu} that $m_0(x,k)$ is analytic for $k \in \C\setminus \R$. 
Moreover, the Propositions \ref{prop_mu_j_large_k} and \ref{cor_s_large_k} imply that $m_0(x,k)=I+ \bigO(k^{-1})$ as $k\to \infty$. On the other hand, again by the Propositions \ref{prop_mu} and \ref{cor_prop_mu}, $m_0(x,\cdot)$ has continuous boundary values along $\R\setminus\{E_1,E_2\}$ from $\C_\pm$ for each $x\in\R$. 
From Corollary \ref{cor_lem_Lambda}, using in particular the fact that $\hat{a}_j \neq 0$, it follows that $m_0(x,k)= \hat{m}_{0,j}(x) + \bigO\big((k-E_j)^{1/2}\big)$ as $k\to E_j$, $j = 1,2$, for some matrices $\hat{m}_{0,j}(x)$. In particular, the boundary values of $m_0$ are continuous also at $E_1$ and $E_2$.

In the following, we show that $m_0$ satisfies the jump relation in \eqref{RHm0}---first for $k\in \R\setminus [E_1,E_2]$, and then for $k\in (E_1,E_2)$. By the continuity of $m_{0\pm}$ just established and the continuity of $r$ established in Theorem \ref{thm_direct}\eqref{thm_direct_a}, it follows that this jump relation is in fact satisfied for all $k\in\R$.

\subsubsection{Jump across $\R\setminus[E_1,E_2]$} \label{subsec_jump_1}

For $k\in\R\setminus[E_1,E_2]$, the columns $\frac{[\phi_1(\cdot,k)]_1}{a(k)}$ and  $[\phi_2(\cdot,k)]_2$ are two linearly independent solutions of the $x$-part in \eqref{lax}, since $a(k)\neq 0$ by Proposition \ref{cor_prop_mu}; the same holds true for $[\phi_2(\cdot,k)]_1$ and $\frac{[\phi_1(\cdot,k)]_2}{\overbar{a(k)}}$. 
Hence there exists a matrix
\begin{equation*}
	\tilde v(k) = \begin{pmatrix} \tilde v_{11}(k) &\tilde v_{12}(k) \\\tilde v_{21}(k) &\tilde v_{22}(k)\end{pmatrix} \in GL(2,\C)
\end{equation*}
such that
\begin{align*}
\begin{pmatrix} \frac{[\phi_1(x,k)]_1}{a(k)} & [\phi_2(x,k)]_2 \end{pmatrix}
=\begin{pmatrix}  [\phi_2(x,k)]_1 & \frac{[\phi_1(x,k)]_2}{\overbar{a(k)}} \end{pmatrix} \tilde v (k), \qquad  k\in \R\setminus[E_1,E_2],\; x\in\R,
\end{align*}
that is,
\begin{align}\label{v0columns}
\begin{cases}
 \frac{[\phi_1]_1}{a} =  [\phi_2]_1 \tilde v_{11} + \frac{[\phi_1]_2}{\overbar{a}} \tilde v_{21}, \\
[\phi_2]_2  =  [\phi_2]_1 \tilde v_{12} +\frac{[\phi_1]_2}{\overbar{a}} \tilde v_{22},
\end{cases}
\qquad  k\in \R\setminus[E_1,E_2],\; x\in\R.
\end{align}

Applying $\det(\cdot, [\phi_1]_2)$ to the first equation in \eqref{v0columns} gives
\begin{equation*}
	\det\!\bigg(\frac{[\phi_1(x,k)]_1}{a(k)}, [\phi_{1}(x,k)]_2\bigg)
	= \det\!\big([\phi_2(x,k)]_1, [\phi_{1}(x,k)]_2\big)  v_{11}, \qquad   k\in \R\setminus[E_1,E_2],\;x\in\R.
\end{equation*}
Since $[\phi_2]_1=\bar a [\phi_1]_1+\bar b [\phi_1]_2$ and $\det \phi_j=1$, we obtain
\begin{align*}
\frac1{a(k)}=\overbar{a(k)} \tilde v_{11}(k),  \qquad k\in \R\setminus[E_1,E_2].
\end{align*}
Using that $a \bar a-b\bar b=1$ on $\R\setminus[E_1,E_2]$, we infer that
\begin{equation*}
\tilde v_{11}(k)=\frac{1}{|a(k)|^2}=1-|r(k)|^2, \qquad k\in \R\setminus[E_1,E_2].
\end{equation*}

Similarly, applying $\det(\cdot,[\phi_{2}]_1)$ to the first equation in \eqref{v0columns} gives
\begin{equation*}
	\det\!\bigg(\frac{[\phi_1(x,k)]_1}{a(k)}, [\phi_{2}(x,k)]_1\bigg)
	= \det\!\bigg(\frac{[\phi_1(x,k)]_2}{\overbar{a(k)}}, [\phi_{2}(x,k)]_1\bigg) \tilde v_{21} 
\end{equation*}
for $k\in \R\setminus[E_1,E_2]$, and an application of $[\phi_2]_1=\bar a [\phi_1]_1+\bar b [\phi_1]_2$ yields
\begin{equation*}
	\det\!\bigg(\frac{[\phi_1(x,k)]_1}{a(k)}, \bar b(k) [\phi_1(x,k)]_2\bigg)
	= \det\!\bigg(\frac{[\phi_1(x,k)]_2}{\overbar{a(k)}}, \overbar{a(k)} [\phi_1(x,k)]_1\bigg) \tilde  v_{21},
\end{equation*}
for $k\in \R\setminus[E_1,E_2]$, which simplifies to
\begin{equation*}
\tilde	v_{21}(k)=-\frac{\overbar{b(k)}}{a(k)}=-r(k), \qquad k\in \R\setminus[E_1,E_2].
\end{equation*}
Analogously it follows from the second equation in \eqref{v0columns} that
\begin{equation*}
	\tilde v_{12}(k) = \overbar{r(k)}, \quad \tilde v_{22}(k) = 1, \qquad k\in \R\setminus[E_1,E_2].
\end{equation*}
We conclude that $\big(m_0(x,k)\big)_+=\big(m_0(x,k)\big)_- v_0(x,k)$ for all $x\in\R$ and $k\in \R\setminus[E_1,E_2]$ with $v_0(x,k)$ given by \eqref{def_v_0}.

\subsubsection{Jump across $(E_1,E_2)$}
By Proposition \ref{cor_prop_mu} and Lemma \ref{lem_symm_phi_1_(E_1,E_2)},
\begin{equation} \label{identity_a_det}
	\begin{cases} 
	a  =\det\!\big( [\phi_{1}]_1,[\phi_2]_2\big)\neq0, &k\in \overbar{\C_+}\setminus\{E_1,E_2\}, \\
	\bar a =\det\!\big([\phi_2]_1,  [\phi_{1}]_1\big)\neq0, &k\in (E_1,E_2),
	\end{cases}	
	\qquad x\in\R. 
\end{equation}
Using similar arguments as in Section \ref{subsec_jump_1}, we find a matrix 
\begin{equation*}
	\hat v(k) = \begin{pmatrix} \hat v_{11}(k) & \hat v_{12}(k) \\ \hat v_{21}(k) & \hat v_{22}(k)\end{pmatrix}, \qquad k\in (E_1,E_2),
\end{equation*}
such that
\begin{align*}
\begin{pmatrix} \frac{[\phi_{1}(x,k)]_{1}}{a(k)} & [\phi_2(x,k)]_2 \end{pmatrix}=\begin{pmatrix}  [\phi_2(x,k)]_1 & \frac{[\phi_{1}(x,t,k)]_2}{\overbar{a(k)}} \end{pmatrix} \hat v (k),  \qquad   k\in (E_1,E_2),\; x\in\R.
\end{align*}
This implies, using \eqref{phi_identity_(E_1,E_2)}, that
\begin{align}\label{v0Jcolumns2}
\begin{cases}
 \frac{[\phi_{1}]_1}{a} =  [\phi_2]_1 \hat v_{11} + \frac{[\phi_{1}]_1}{\overbar{a}} \hat v_{21}, \\
[\phi_2]_2  =  [\phi_2]_1 \hat v_{12} +\frac{[\phi_{1}]_1}{\overbar{a}} \hat v_{22},
\end{cases}
\qquad  k\in (E_1,E_2),\; x\in\R.
\end{align}
In view of \eqref{identity_a_det}, the identities in \eqref{v0Jcolumns2} imply that 
\begin{equation*}
	\hat v(k) = 
	\begin{pmatrix}
		0& -\frac{a(k)}{\overbar{a(k)}}\\
		\frac{\overbar{a(k)}}{a(k)}& 1
	\end{pmatrix}
	=
	\begin{pmatrix}
		0&\overbar{r(k)}\\
		-r(k)& 1
	\end{pmatrix},
	\qquad   k\in (E_1,E_2),
\end{equation*}
which one can see by successively applying $\det(\cdot, [\phi_{1}]_1)$ and $\det(\cdot, [\phi_2]_1)$ to the first equation, and $\det(\cdot,  [\phi_{1}]_1)$ and $\det([\phi_2]_1,\cdot)$ to the second equation in \eqref{v0Jcolumns2}.
Recalling that $|r|\equiv1$ on $(E_1,E_2)$ we conclude that $m_{0+}(x,k) = m_{0-}(x,k) v_0(x,k)$ for all $k\in (E_1,E_2)$ and $x\in\R$ with $v_0(x,k)$ given by \eqref{def_v_0}.
\proofend

\section{Proof of Theorem \ref{thm_inverse}}\label{inversesec}
\noindent
Suppose $u_0\colon \R \to \C$ satisfies \eqref{u0assumptions} for $N_1 = 2$ and \eqref{u0assumptions2} for $N_2 = 3$. Define $r(k)$ by \eqref{def_r}. Theorem \ref{thm_direct} implies that $r \in \mathcal C^{2}(\R\setminus \{E_1,E_2\})\cap \mathcal C(\R)$ and that
\begin{equation*}
	\partial^m_k r(k) = \bigO(k^{-4}) \qquad\text{as} \quad k\to\pm\infty, \quad m = 0, 1,2.
\end{equation*} 
In fact, $v$ is H\"older continuous on  $\R$ with exponent $1/2$ because, by \eqref{ratbranchpoints} and \eqref{coeffqjl}, $r'(k)$ has a square root singularity as $k \to E_j$, $j =1,2$.

Since the jump matrix $v$ has unit determinant, the RH problem \eqref{RHm} has at most one solution. Moreover, the existence of a solution of the RH problem \eqref{RHm} is a consequence of the H\"older continuity of $v$ and the following vanishing lemma, see \cite{Z1989} and \cite[Appendix A]{KMM2003}. 

\begin{proposition}[Vanishing Lemma]\label{vanishinglemma}
	Let $x\in\R$ and $t\geq 0$. Suppose that $m(x,t,k)$ is a solution of the RH problem of Theorem \ref{thm_inverse} but with the normalization condition $m(x,t,k)=I+ \bigO(k^{-1})$ replaced by the homogeneous condition $m(x,t,k) = \bigO(k^{-1})$ as $k\to \infty$. Then $m$ vanishes everywhere.
\end{proposition}
\begin{proof}
We write $m(k)=m(x,t,k)$ and define	$G(k)\coloneqq m(k)\, \overbar{m(\bar{k})}^\intercal$.
	Since $m(k) = \bigO(k^{-1})$ as $k\to \infty$, we have
	\begin{align*}
		0&=\int_\R G_+(k) \,\dx k =\int_\R m_-(k)v(k) \overbar{m_-(k)}^\intercal \,\dx k,
		\\
		0&=\int_\R G_-(k) \,\dx k =\int_\R m_-(k)\overbar{v(k)}^\intercal \overbar{m_-(k)}^\intercal \,\dx k,
	\end{align*}
and hence
	\begin{align}\label{mXm}
		\int_\R m_- Y\overbar{m_-}^\intercal \, \dx k=0,
	\end{align}
	where
	\begin{equation*}
		Y(k)= 
	\frac12v(k)+
	\frac12\overbar{v(k)}^\intercal=\begin{pmatrix}1 - |r(k)|^2 & 0 \\
		0 & 1 \end{pmatrix},\qquad k\in\R.	
	\end{equation*}
By Theorem \ref{thm_direct},  $|r|<1$ on $\R\setminus[E_1,E_2]$. It follows that $Y$ is positive semidefinite on $\R$ and positive definite on $\R\setminus[E_1,E_2]$.
	Hence we conclude that the diagonal entries of $m_-Y\overbar{m_-}^\intercal$ are nonnegative.
	The identity \eqref{mXm} thus implies that the diagonal elements of $m_-Y\overbar{m_-}^\intercal$ vanish for a.e. $k\in\R$. Since $Y$ is positive definite on  $\R\setminus[E_1,E_2]$, we conclude that $m_-=0$ on $\R\setminus[E_1,E_2]$, and hence also $m_+=m_-v=0$ on $\R\setminus[E_1,E_2]$.
	Standard arguments show that $m$ is zero everywhere.
\end{proof}

It is well-known how to prove that the limit in \eqref{def_u} exists for all $(x,t)\in\R\times[0,\infty)$, and that $u$ is a classical solution of \eqref{dNLS}; see e.g. \cite[Theorem 7]{Lenells16} for a rigorous implementation of these steps in the case of the mKdV equation.
The central part in the proof that $u$ is a classical solution is to show that the solution $m$ of the RH problem \eqref{RHm} solves the Lax pair equations
\begin{align*}
	\begin{cases}
		m_x+ i k[\sigma_3,m]=\mathsf{U}m,
		\\
		m_t+2 i k^2[\sigma_3,m]=\mathsf{V}m
	\end{cases}
\end{align*}
for all $(x,t)\in\R\times[0,\infty)$ and $k\in\C\setminus\R$. For this purpose, it is sufficient that $r\in (L^2\cap L^\infty)(\R)$ and $r(k)= \bigO(k^{-4})$ as $k\to\pm\infty$ (i.e., one order of decay less than in the mKdV setting, since NLS only involves second-order derivatives with respect to $x$). Both requirements are fulfilled in view of Theorem \ref{thm_direct}. 

The fact that $u(x,t)$ satisfies the boundary conditions \eqref{uboundaryconditions} as $x \to \pm \infty$ follows from a steepest descent analysis of the RH problem for $m$, see Theorems \ref{thm_sec_L} and \ref{thm_sec_R}. 

Let us verify that $u$ satisfies the initial condition $u(x,0) = u_0(x)$. The function $m_0(x,k)$ defined in \eqref{def_m_0} satisfies the RH problem \eqref{RHm} evaluated at $t = 0$ as a consequence of Theorem \ref{thm_direct}\eqref{thm_direct_c}. By uniqueness, we have $m(x,0,k) = m_0(x,k)$. Moreover, by \eqref{def_m_0}, \eqref{defphi120}, \eqref{mu_1p} and \eqref{mu1mu2atinfinity},
\begin{equation*}
	2 i \overset{\angle}{\lim_{k \to \infty}} k [m_0(x,k)]_{12} = u_0(x),
\end{equation*}
where $\angle$ indicates that the limit is taken nontangentially with respect to $\R$, i.e., with $\arg k \in (\epsilon, \pi-\epsilon) \cup (\pi+ \epsilon, 2\pi - \epsilon)$ for any $\epsilon > 0$. 
Hence,
\begin{align*}
	u(x,0)= 2 i \overset{\angle}{\lim_{k \to \infty}} (k\,m(x,0,k))_{12}
	= 2 i \overset{\angle}{\lim_{k \to \infty}} k [m_0(x,k)]_{12}=u_0(x).
\end{align*}
We conclude that $u$ is a global solution of the Cauchy problem \eqref{dNLS_IVP} with initial data $u_0$.
\proofend

\section{Proof of Theorem \ref{thm_sec_M}}\label{middlesec} 
\noindent
Suppose $u_0$ satisfies the assumptions of Theorem \ref{thm_direct} with $N_1 = 8$ and $N_2=3$, and let $m$ be the unique solution of the RH problem in Theorem \ref{thm_inverse} corresponding to $u_0$. We use Deift-Zhou steepest descent arguments \cite{DZ1993} to establish the large $t$ behavior of $m$ in the sector $\mathcal M$. We then use \eqref{def_u} to find the large $t$ behavior of $u(x,t)$. Suitable transformations of the RH problem are implemented in Sections \ref{sec_1st_trans}--\ref{sec_4th_trans}, parametrices are introduced in Sections \ref{sec_global_parametrix}--\ref{sec_local_parametrix}, and the asymptotics of $u$ is finally established in Section \ref{sec_solution_asymp}. Other works using Deift-Zhou steepest descent arguments to study asymptotics for the NLS equation include \cite{BB2019, BTVZ2007, BT2013, DZ2003, TVZ2004, TVZ2006, BM2017, BKS2009, BLS2021, BKS2011, BV2007, DIZ1993, Fromm19, J2015, K1996}.

\subsection{First transformation} \label{sec_1st_trans}
Recall that $E_1\coloneqq -\beta-\alpha$, $k_0(\xi) \coloneqq -\frac{E_1 + \xi}{3}$ for $\xi\in I_{\mathcal M}$, and
\begin{equation*}
	\mathcal{X}\colon I_{\mathcal M}\times \C\setminus [E_1,k_0] \to \C, \quad
	\mathcal{X}(\xi,k)\coloneqq \sqrt{(k-E_1)(k-k_0)},
\end{equation*}
where the branch is such that $\mathcal{X}(\xi,k)=k+\bigO(1)$ as $k\to\infty$.

\subsubsection{The $g$-function} \label{sec_g}
Let  
\begin{equation} \label{def_g}
	g\colon I_{\mathcal M}\times \C\setminus [E_1,k_0] \to \C, \quad
	g(\xi,k) \coloneqq 2(k-k_0) \mathcal{X}(\xi,k).
\end{equation} 
Then
\begin{equation} \label{g_ktoinfty}
	g(\xi,k) = 2k^2 +\xi k + g_\infty (\xi) + \bigO(k^{-1}), \qquad k\to \infty, 
\end{equation}
with $g_\infty (\xi)$ given in \eqref{def_g_infty_sec_M}.
For brevity, we will often write $\mathcal{X}(k)$ and $g(k)$ for $\mathcal{X}(\xi,k)$ and $g(\xi,k)$ etc. 
The $k$-derivative of $g$ given by
\begin{equation*} 
	g'(k) = 2 \mathcal{X}(k) + \frac{(k-k_0)(2k-E_1+k_0)}{\mathcal{X}(k)}
\end{equation*} 
has two zeros at the point $\mu_\pm(\xi)$ given by
\begin{equation*} 
	\mu_+ =k_0  \quad \text{and} \quad \mu_-  = \frac{8E_1 - \xi}{12} \in  (E_1,k_0).
\end{equation*}  
It is clear that $\mu_-(\xi)$ ranges within a compact subset of $(E_1,E_2)$ for $\xi\in I_{\mathcal M}$. 
The $g$-function satisfies the jump condition
\begin{equation} \label{g_jump}
	g_+ + g_- =0 \quad \text{on} \quad (E_1,k_0).
\end{equation} 
with $g_+(k)  \in - i \R_+$ and $g_-(k) \in  i\R_+$ for $k \in (E_1,k_0)$. Moreover, $g$ obeys the symmetry
\begin{equation*}
	g(k) = \overbar{g(\bar{k})}, \quad k\in\C\setminus (E_1,E_2).
\end{equation*}
The series expansion of $g$ at $k_0$ is given by
\begin{align*}
	g(\xi,k)  = &\; 2\sqrt{k_0-E_1}(k-k_0)^{3/2}\bigg(1 
	+ \frac{k-k_0}{2(k_0-E_1)}  
	+ \bigO\big((k-k_0)^{2}\big)\bigg), \quad	  \C \setminus (-\infty, k_0] \ni  k \to k_0,
\end{align*}
where the branch cuts for $(k-k_0)^{3/2}$ and $(k-k_0)^{5/2}$ run along $(-\infty, k_0]$, and the error term is uniform with respect to $\xi \in I_{\mathcal M}$.
A signature table for $\im g$ is shown in Figure~\ref{fig:g_function}. 
\begin{figure}[h!] 
	\begin{center}
		\begin{overpic}[width=.60\textwidth]{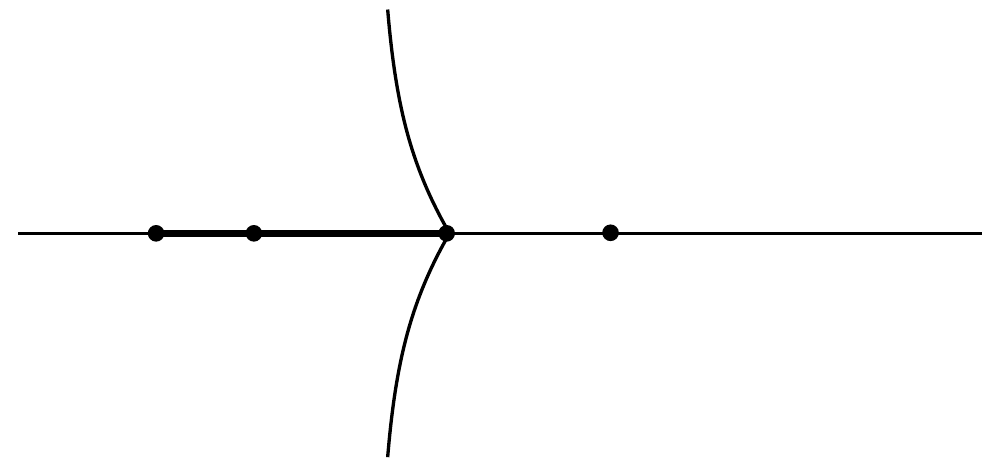}
			\put(14,19){{\footnotesize $E_1$}}
			\put(59.5,19){{\footnotesize $E_2$}}
			\put(46,25){{\footnotesize $k_0$}}
			\put(23.8,25.2){{\footnotesize $\mu_-$}}
			\put(56,35){{\footnotesize $ \im g > 0$}}
			\put(56,9){{\footnotesize $ \im g < 0$}}
			\put(18,35){{\footnotesize $ \im g < 0$}}
			\put(18,9){{\footnotesize $ \im g > 0$}}
		\end{overpic}
		\caption{The zero set, branch cut (thick line between $E_1$ and $k_0$) and signature table of $ \im g$; the vertical component of the zero set asymptotes to the line $\re k = -\xi/4$.}
		\label{fig:g_function}
	\end{center}
\end{figure}
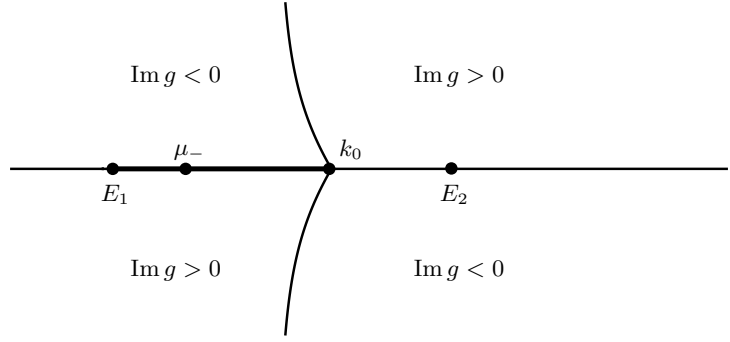

\subsubsection{Definition of $m^{(1)}$}
For $(x,t,k)\in\mathcal M\times \C\setminus\R$ we define $m^{(1)}$ by
\begin{equation*} 
	m^{(1)} (x,t,k) \coloneqq \e^{- i t g_\infty(\xi)\sigma_3} \, m(x,t,k) \, \e^{ i t[g(\xi,k)-k\xi - 2k^2]\sigma_3}.
\end{equation*} 
Then for each $(x,t)\in \mathcal M$, $m^{(1)}$ satisfies the following RH problem:
\begin{itemize}
\item $m^{(1)}(x,t,\cdot)$ is analytic in $\C\setminus \R$;

\item $m^{(1)}(x,t,\cdot)$ has continuous boundary values on $\R$ satisfying the jump relation 
\begin{align*}
m^{(1)}_+(x,t,k) = m^{(1)}_-(x,t,k) v^{(1)}(x,t,k) \quad \text{for} \; k \in \R,
\end{align*}
where
\begin{equation*} 
	v^{(1)}(x,t,k) = \e^{ i t[k\xi + 2k^2 -g_-(\xi,k)]\sigma_3} \, v(x,t,k) \, \e^{ i t[g_+(\xi,k)-k\xi - 2k^2]\sigma_3}, \qquad k \in \R;
\end{equation*} 

\item $m^{(1)}(x,t,k)=I+ \bigO(k^{-1})$ as $k\to \infty$.

\end{itemize}

In view of \eqref{g_jump} and the fact that $|r|\equiv 1$ on $[E_1,E_2]$, we find
\begin{align*}  
		v^{(1)} = \begin{cases}
		\begin{pmatrix}
			1-|r|^2 & \bar r \e^{-2 i t g} \\
			- r \e^{2 i t g} & 1
		\end{pmatrix}, & k \in \R \setminus (E_1,k_0), \\
		\begin{pmatrix}
			0 & \bar r  \\
			- r  & \e^{-2 i t g_+}
		\end{pmatrix}, & k \in [E_1,k_0].\end{cases}
\end{align*}

\subsection{Second transformation}
The second transformation serves as preparation for the subsequent one. The introduction of an auxiliary function $\mathcal{D}$ will allow us to open up the ray $(-\infty,\mu_-)$ into two rays going from $\mu_-$ to $\infty$ in the upper and lower half-planes, respectively. Furthermore, the local approximation of $\mathcal{D}$ at $k_0$, will be crucial in order to define the local parametrix at $k_0$ in Section \ref{sec_local_parametrix}.   

\subsubsection{The function $\mathcal{D}$}
Define $\mathcal{D}\colon I_{\mathcal M} \times \C\setminus (-\infty,k_0)\to \C$ by 
\begin{equation} \label{D}
	\mathcal{D}(\xi,k) \coloneqq \exp\Bigg[    
	\frac{\mathcal{X}(\xi,k)}{2 \pi  i} \Bigg( \int^{E_1}_{-\infty} \frac{\log(1-|r(s)|^2)}{\mathcal{X}(\xi,s)(s-k)} \, \dx s + 
	\int^{k_0}_{E_1} \frac{\log{r(s)}}{\mathcal{X}_+(\xi,s)(s-k)} \, \dx s  \Bigg)\Bigg],
\end{equation} 
where the branch of the function $\log{r(s)}$ is such that it is continuous for $s \in [E_1, k_0]$ and the principal branch is used for $\log{r(E_1)}$.

\begin{lemma} \label{lem_D}
	For each $\xi\in I_{\mathcal M}$ the function $\mathcal{D}(\xi,\cdot)$ is analytic in $\C\setminus (-\infty,k_0]$ with continuous boundary values $\mathcal{D}_\pm(\xi,\cdot)$ on $(-\infty,k_0)\setminus \{E_1\}$ satisfying the jump relations
	\begin{equation} \label{D_jump}
		\begin{cases}
			\mathcal{D}_+ = \mathcal{D}_- (1-|r|^2) &\text{on} \quad (-\infty,E_1), \\
			\mathcal{D}_+  \mathcal{D}_- = r &\text{on} \quad (E_1,k_0).
		\end{cases}
	\end{equation} 
	Furthermore, $\mathcal{D}(\xi,\cdot)$ satisfies the following properties for any $\xi\in I_{\mathcal M}$.
	\begin{enumerate}[\upshape (i)]
		\item \label{lem_D_i}
		$\mathcal{D}(\xi,\cdot)$ obeys the symmetry 
		\begin{equation} \label{D_symm}
			\mathcal{D}(k) \overbar{\mathcal{D}(\bar{k})} = 1 \quad \text{for} \quad k\in\C \setminus (-\infty,k_0];
		\end{equation} 
		in particular, $\mathcal{D}_+(k) \overbar{\mathcal{D}_-(k)} = 1$ for $k\in (-\infty,k_0)\setminus \{E_1\}$. 
		\item \label{lem_D_ii}
		For $k\in\C\setminus (-\infty,k_0)$,  
		\begin{equation*} 
			\mathcal{D}^{\pm 1}(\xi,k) = \mathcal{D}^{\pm 1}_\infty(\xi) + \bigO(k^{-1}) \quad \text{as} \quad k\to\infty,
		\end{equation*} 
		uniformly with respect to $\xi\in I_{\mathcal M}$, where
		\begin{equation*} 
			\mathcal{D}_\infty(\xi) = \exp
			\Bigg[-\frac{1}{2 \pi  i} \Bigg( \int^{E_1}_{-\infty} \frac{\log(1-|r(s)|^2)}{\mathcal{X}(\xi,s)} \, \dx s + 
			\int^{k_0}_{E_1} \frac{\log{r(s)}}{\mathcal{X}_+(\xi,s)} \, \dx s  \Bigg)\Bigg],
		\end{equation*} 
		which satisfies $|\mathcal{D}_\infty(\xi)|=1$ for all $\xi\in I_{\mathcal M}$.

\item \label{lem_D_iii}
		As $\overbar{\C_+}\ni k\to E_1$,  
		\begin{equation}\label{calDatE1} 
			\mathcal{D}(\xi,k) = (k-E_1)^{1/4} \e^{d_0} \, \big(1+\bigO((k-E_1)^{1/2} )\big)
		\end{equation} 
		uniformly for $\xi\in I_{\mathcal M}$,
		where 
		\begin{equation*}
			d_0 \coloneqq \frac{1}{2}\log(\overbar{q_{1,0}} q_{1,1}+ q_{1,0} \overbar{q_{1,1}}) + \frac{1}{2} \log q_{1,0} \in \C
		\end{equation*}
		and the principal branch is used for the fourth root and the logarithms.
		By \eqref{D_symm}, a similar formula holds for $\overbar{\C_-}\ni k\to E_1$. 
		\item \label{lem_D_iv}
		For every fixed $\epsilon>0$, the functions $\mathcal{D}^{\pm 1}(\xi,\cdot)\colon \C\setminus D_\epsilon(E_1)\to \C$ are uniformly bounded for $\xi\in I_{\mathcal M}$, where $D_\epsilon(E_1)$ is the open disk of radius $\epsilon$ centered at $E_1$. 
		
	\item \label{lem_D_k0}
		As $k\in\C\setminus (-\infty,k_0)$ approaches $k_0$,
		\begin{align} 
			\mathcal{D}(k) &= \sqrt{r(k_0)} \Bigg[1 - C_{k_0}  \sqrt{k - k_0} + \frac{C_{k_0}^2 +  \frac{r'(k_0)}{r(k_0)}}{2} (k-k_0) + \bigO(|k-k_0|^{3/2})\Bigg], \label{D_at_k0} \\
			\mathcal{D}^{-1}(k) &= \frac{1}{\sqrt{r(k_0)}} \Bigg[1 + C_{k_0}  \sqrt{k - k_0} + \frac{C_{k_0}^2 - \frac{r'(k_0)}{r(k_0)}}{2} (k-k_0) + \bigO(|k-k_0|^{3/2}) \Bigg], \label{invD_at_k0} 
		\end{align}
		uniformly for $\xi\in I_{\mathcal M}$, with $C_{k_0}$ given by \eqref{C_k0}. 
	\end{enumerate}
\end{lemma}
\begin{proof}
	Basic properties of the Cauchy transform imply that $\mathcal{D}(\xi,\cdot)$ is analytic in $\C\setminus (-\infty,k_0]$ and that it obeys the jump condition \eqref{D_jump}.
	The properties \eqref{lem_D_i}, \eqref{lem_D_ii}, and \eqref{lem_D_iv} follow easily from the definition of $\mathcal{D}$ and the properties of $r$, cf.~Theorem \ref{thm_direct}; 

To prove \eqref{lem_D_iii}, note first that
\begin{align}\label{expansion1}
\log(1-|r(s)|^2) = \log c_1 + \frac{1}{2}\log(E_1-s) 
+ c_2\sqrt{E_1-s}
+ \bigO(E_1-s), \qquad s \nearrow E_1,
\end{align}
where $c_1 \coloneqq \overbar{q_{1,0}} q_{1,1}+ q_{1,0} \overbar{q_{1,1}}$ and $c_2 \coloneqq -c_1^{-1}(2\re(q_{1,2}\overbar{q_{1,0}}) + |q_{1,1}|^2)$. An application of Lemma \ref{lem_sing_int_near_boundary} shows that
\begin{align}\label{cauchya}
\frac{\mathcal{X}(\xi,k)}{2 \pi  i} \int^{E_1}_{-\infty} \frac{\log{c_1}}{\mathcal{X}_+(\xi,s)(s-k)} \,\dx s = \frac{1}{2}\log{c_1} + \bigO(\sqrt{k-E_1}), \qquad \overbar{\C_+}\ni k \to E_1.
\end{align}
On the other hand,
\begin{equation*}
	\frac{1}{2\pi i} \int_{-\infty}^{E_1} \frac{\log(E_1 - s)}{\mathcal{X}(\xi,s)(s-k)} \,\dx s
	= f(k) + \bigO(1),
\end{equation*}
where the function
\begin{equation*}
	f(k) \coloneqq \frac{1}{2\pi i} \int_{-\infty}^{E_1} \frac{\log(E_1 - s)}{-\sqrt{|s-E_1||E_1-k_0|}(s-k)} \,\dx s
\end{equation*}
is analytic for $k \in \C \setminus (-\infty, E_1]$, obeys the bound $f(k) = o((k-E_1)^{-1})$ as $k \to E_1$, and satisfies 
\begin{equation*}
f_+(k) - f_-(k) =  \frac{\log|E_1 - k|}{-\sqrt{|k-E_1||E_1-k_0|}}, \qquad k \in (-\infty, E_1].	
\end{equation*}
It follows that
\begin{equation*}
	f(k) = \frac{1}{2i}\frac{\log(k-E_1)}{\sqrt{k-E_1}\sqrt{|E_1-k_0|}}, \qquad k \in \C \setminus (-\infty, E_1],
\end{equation*}
where principal branches are used. We conclude that, as $\overbar{\C_+}\ni k \to E_1$,
\begin{align}\label{cauchy1b}
\frac{\mathcal{X}(\xi,k)}{2 \pi  i} \int^{E_1}_{-\infty} \frac{\frac{1}{2}\log(E_1 - s)}{\mathcal{X}(\xi,s)(s-k)} \,\dx s = \frac{1}{4}\log(k-E_1) + \bigO(\sqrt{k-E_1}).
\end{align}
Furthermore, as $\overbar{\C_+}\ni k \to E_1$,
\begin{align}\label{cauchy1c}
\frac{\mathcal{X}(\xi,k)}{2 \pi  i} \int^{E_1}_{-\infty} \frac{c_2\sqrt{E_1-s}}{\mathcal{X}(\xi,s)(s-k)} \,\dx s 
& = - \frac{c_2}{2 \pi}  \sqrt{k-E_1} \log(k-E_1) + \bigO(\sqrt{k-E_1}).
\end{align}

Next note that
\begin{align}\label{expansion2}
\log r(s) = \log{q_{1,0}} +  \frac{iq_{1,1}}{q_{1,0}} \sqrt{s-E_1} + \bigO(s-E_2), \qquad s \searrow E_1.
\end{align}
A version of Lemma \ref{lem_sing_int_near_boundary} adapted to the left endpoint implies that
\begin{align}\label{cauchy2a}
\frac{\mathcal{X}(\xi,k)}{2 \pi  i} \int^{k_0}_{E_1} \frac{\log{q_{1,0}}}{\mathcal{X}_+(\xi,s)(s-k)} \,\dx s
= \frac{\log{q_{1,0}}}{2} + \bigO(\sqrt{k-E_1}), \qquad \overbar{\C_+}\ni k \to E_1.
\end{align}
Furthermore, as $\overbar{\C_+}\ni k \to E_1$,
\begin{align}\label{cauchy2b}
\frac{\mathcal{X}(\xi,k)}{2 \pi  i} \int^{k_0}_{E_1} \frac{\frac{iq_{1,1}}{q_{1,0}} \sqrt{s-E_1}}{\mathcal{X}_+(\xi,k)(s-k)} \,\dx s
= -\frac{1}{2 \pi}\frac{q_{1,1}}{q_{1,0}}\sqrt{k-E_1} \log(k-E_1) +  \bigO(\sqrt{k-E_1}).
\end{align}
Since $c_2 + \frac{q_{1,1}}{q_{1,0}} = 0$ by \eqref{coeffqjl}, the sum of the right-hand sides of \eqref{cauchy1c} and \eqref{cauchy2b} is $\bigO(\sqrt{k-E_1})$. Property \eqref{lem_D_iii} now follows from \eqref{expansion1}--\eqref{cauchy2b}.
	
	To prove \eqref{lem_D_k0}, we first note that
	\begin{equation} \label{X_k_0}
		\mathcal{X}(\xi,k) = \sqrt{k_0-E_1} \sqrt{k-k_0} + \frac{(k-k_0)^{3/2}}{2\sqrt{k_0-E_1}} + \bigO(|k-k_0|^{5/2}) \quad \text{as} \quad k\to k_0.
	\end{equation} 
	Next we expand the integrals appearing in \eqref{D}. For this purpose let $\kappa \in (E_1,k_0)$ and set
	\begin{align*}
		A(\xi,k) &\coloneqq \int^{E_1}_{-\infty} \frac{\log(1-|r(s)|^2)}{\mathcal{X}(\xi,s)(s-k)} \, \dx s, \\
		B_1 (\kappa,\xi,k) &\coloneqq \int^{\kappa}_{E_1} \frac{\log(r(s))}{\mathcal{X}_+(\xi,s)(s-k)} \, \dx s, \qquad
		B_2 (\kappa,\xi,k) \coloneqq \int^{k_0}_{\kappa} \frac{\log(r(s))}{\mathcal{X}_+(\xi,s)(s-k)} \, \dx s.
	\end{align*}
	From  \eqref{X_k_0}, the properties of $r$, and the fact that $k_0$ can be separated from the integration contours $(-\infty,E_1]$ and $[E_1,\kappa]$ by some small neighborhood, it follows that the Cauchy integrals $A$ and $B_1$ are analytic locally around $k_0$; hence there exists a small $\epsilon$-disc $D_\epsilon(k_0)$ centered at $k_0$ such that the power series
	\begin{align}\label{A1B1_near_k0}
		A(\xi,k) &= \sum^\infty_{n=0}\frac{A^{(n)}(\xi,k_0)}{n!}(k-k_0)^n,  \qquad
		B_1(\kappa,\xi,k) = \sum^\infty_{n=0}\frac{B_1^{(n)}(\kappa,\xi,k_0)}{n!}(k-k_0)^n, 
	\end{align}
	converge absolutely and uniformly for $(\xi,k)\in I_{\mathcal M}\times D_\epsilon(k_0)$.
	Let 
	\begin{equation*} 
		g_0(\kappa,\xi,\cdot) \colon [\kappa, k_0] \to \C, \quad k \mapsto g_0(\kappa,\xi,k) \coloneqq -\frac{ i \log(r(k))}{\sqrt{k-E_1}}.
	\end{equation*} 
	An application of Lemma \ref{lem_sing_int_near_boundary} yields that
	\begin{equation} \label{B2_near_k0}
		B_2(\kappa,\xi,k) = \frac{c_{-1}}{\sqrt{k-k_0}} + c_0 + c_1 \sqrt{k-k_0} + \bigO (|k-k_0|)
	\end{equation} 
	as $k\to k_0$ uniformly for $\xi\in I_{\mathcal M}$ with
	\begin{align*}
		c_{-1} &= - \pi \, g_0(k_0), \\
		c_0 &= \frac{2 \, g_0(k_0)}{\sqrt{k_0-\kappa}} + \int^{k_0}_\kappa \frac{g_0(s)- g_0(k_0)}{s-k_0}  \frac{\dx s}{\sqrt{k_0-s}}, \\
		c_1 &= -\pi \, g_0'(k_0),
	\end{align*}
	Combining \eqref{X_k_0}, \eqref{A1B1_near_k0}, \eqref{B2_near_k0}, and letting $\kappa \searrow E_1$ (which is justified by Lebesgue's dominated convergence theorem), we conclude that
	\eqref{D_at_k0} holds uniformly for all $\xi\in I_{\mathcal M}$ with $C_{k_0}$ given by \eqref{C_k0}; \eqref{invD_at_k0} is obtained by inverting \eqref{D_at_k0}.
\end{proof}

\subsubsection{Definition of $m^{(2)}$}
For $(x,t,k)\in\mathcal M\times \C\setminus\R$ we define $m^{(2)}$ by 
\begin{equation*} 
	m^{(2)} (x,t,k) \coloneqq \mathcal{D}_\infty(\xi)^{\sigma_3} \, m^{(1)}(x,t,k) \, \mathcal{D}(\xi,k)^{-\sigma_3}.
\end{equation*} 
As a consequence of Lemma \ref{lem_D}, $m^{(2)}$ satisfies the following RH problem for all $(x,t)\in \mathcal M$:
\begin{itemize}
\item $m^{(2)}(x,t,\cdot)$ is analytic in $\C\setminus \R$;

\item $m^{(2)}(x,t,\cdot)$ has continuous boundary values on $\R \setminus \{E_1\}$ satisfying the jump relation 
\begin{align*}
m^{(2)}_+(x,t,k) = m^{(2)}_-(x,t,k) v^{(2)}(x,t,k) \quad \text{for} \; k \in \R \setminus \{E_1\},
\end{align*}
where $v^{(2)} = \mathcal{D}^{\sigma_3}_- v^{(1)} \mathcal{D}^{-\sigma_3}_+$; 

\item $m^{(2)}(x,t,k)=I+ \bigO(k^{-1})$ as $k\to \infty$;

\item $m^{(2)}(x,t,k)= \bigO((k-E_1)^{-1/4})$ as $k\to E_1$.

\end{itemize}

Write $\R = \Gamma^{(2)}_1 \cup \Gamma^{(2)}_2 \cup \Gamma^{(2)}_3$, where 
\begin{equation*}
	\Gamma^{(2)}_1 = (k_0,+\infty), \qquad \Gamma^{(2)}_2 = (-\infty,E_1), \qquad \Gamma^{(2)}_3 = [E_1,k_0],
\end{equation*}
and let $v_j^{(2)}$ denote the restriction of $v^{(2)}$ to $\Gamma^{(2)}_j$.
Setting 
\begin{equation*} 
	r_2(k) \coloneqq 
	\frac{\overbar{r(k)}}{1- |r(k)|^2}, \qquad k\in \R \setminus [E_1,E_2],
\end{equation*}
and employing \eqref{D_jump}--\eqref{D_symm}, we infer that
\begin{flalign*}
	\begin{aligned}
		&v_1^{(2)} =
		\begin{pmatrix}
			1-|r|^2 & \mathcal{D}^2 \bar r \e^{-2 i t g} \\
			-\mathcal{D}^{-2} r \e^{2 i t g} & 1
		\end{pmatrix}
	=
	\begin{pmatrix}
		1 & \mathcal{D}^2 \bar r \e^{-2 i t g} \\
		0 & 1
	\end{pmatrix}
	\begin{pmatrix}
		1 & 0 \\
		-\mathcal{D}^{-2}  r \e^{2 i t g} & 1
	\end{pmatrix},  
		 \\
		&v_2^{(2)} =
		\begin{pmatrix}
			1 & \mathcal{D}_+ \mathcal{D}_- \bar r \e^{-2 i t g} \\
			- \mathcal{D}^{-1}_+ \mathcal{D}^{-1}_- r \e^{2 i t g} & 1 -|r|^2
		\end{pmatrix}
		=
	\begin{pmatrix}
		1 & 0 \\
		- \mathcal{D}^{-2}_- \bar r_2 \e^{2 i t g} & 1
	\end{pmatrix}
	\begin{pmatrix}
		1 & \mathcal{D}^2_+ r_2 \e^{-2 i t g} \\
		0 & 1
	\end{pmatrix},  
		 \\
		&v_3^{(2)} =
		\begin{pmatrix}
			0 &  1  \\
			- 1  & \mathcal{D}_+ \mathcal{D}^{-1}_- \e^{-2 i t g_+}
		\end{pmatrix}.
	\end{aligned}
\end{flalign*}


\subsection{Third transformation} \label{sec_4th_trans}
The third transformation requires analytic approximations of the spectral functions $r$ and $r_2$, which we summarize in the following two lemmas.
Let the open regions $U^{(3)}_j$, $j=1,2,3,4$, be as in Figure~\ref{fig:Gamma3}.

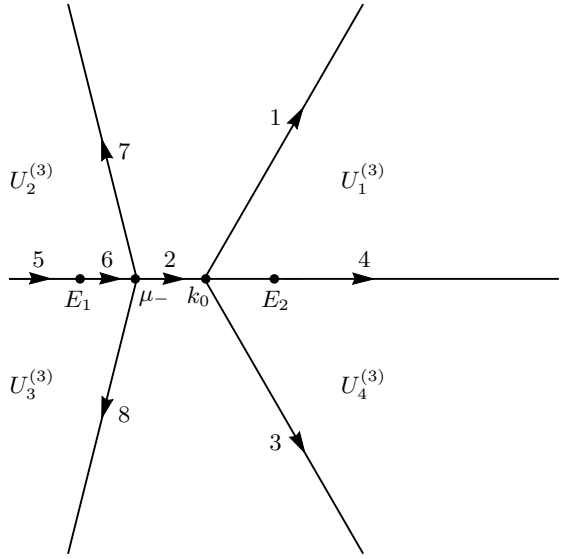
\begin{figure}[h] \centering
	\begin{overpic}[width=.46\textwidth]{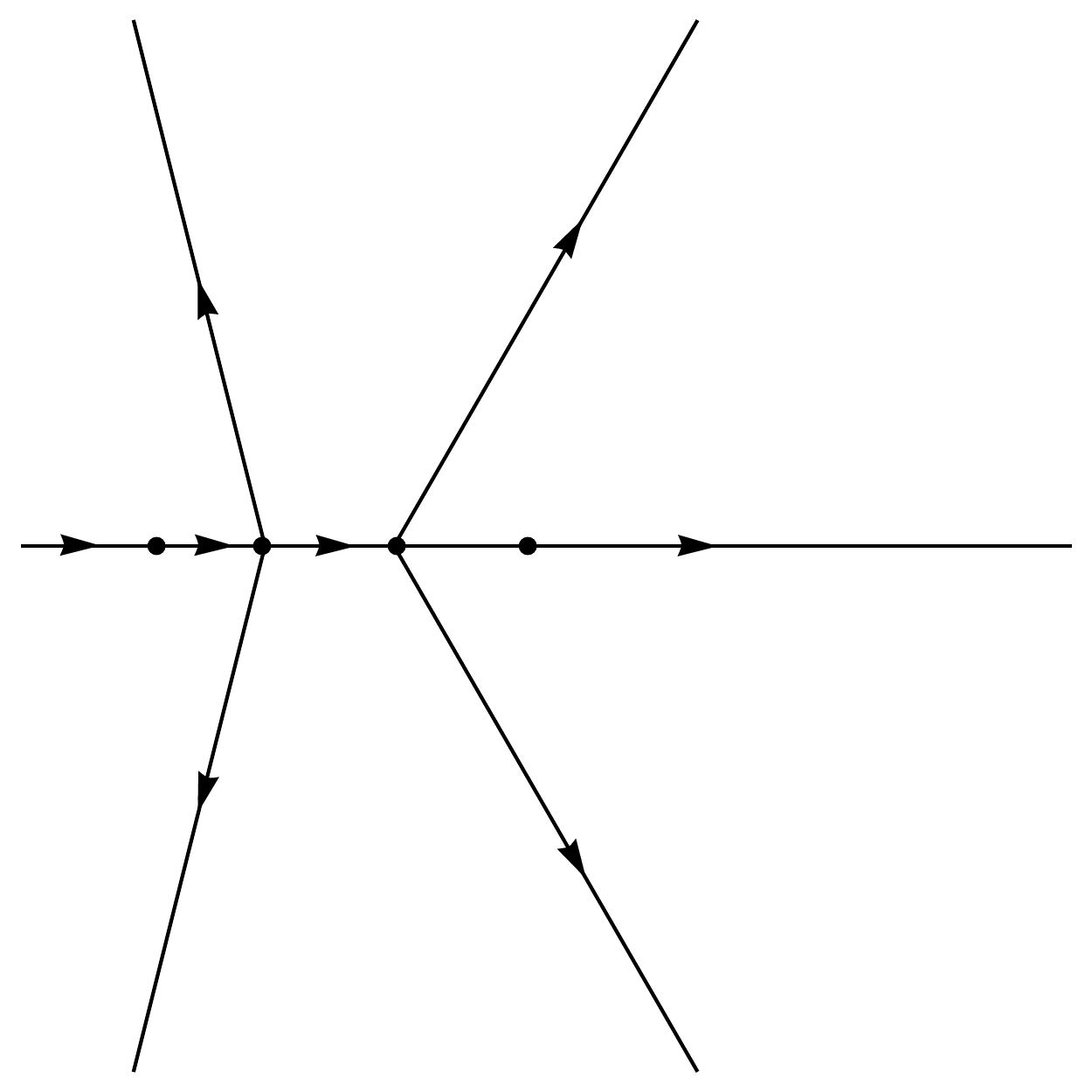}
		\put(11.5,45){\footnotesize{$E_1$}} 
		\put(33,45.7){\footnotesize{$k_0$}}   
		\put(24.6,46.2){\footnotesize{$\mu_-$}}    
		\put(46,45){\footnotesize{$E_2$}} 
		\put(60,66){\footnotesize{$U^{(3)}_1$}} 
		\put(60,30){\footnotesize{$U^{(3)}_4$}} 
		\put(2,66){\footnotesize{$U^{(3)}_2$}} 
		\put(2,30){\footnotesize{$U^{(3)}_3$}}
		\put(6,52){\footnotesize{$5$}}
		\put(18,52){\footnotesize{$6$}}
		\put(29,52){\footnotesize{$2$}}
		\put(63,52){\footnotesize{$4$}}
		\put(47.5,77){\footnotesize{$1$}}
		\put(47.5,20){\footnotesize{$3$}}
		\put(21,71){\footnotesize{$7$}}
		\put(21,25){\footnotesize{$8$}}
	\end{overpic}	
	\caption{The contour $\Gamma^{(3)}$ and the regions $U^{(3)}_j$, $j = 1, \dots, 4$.} 
	\label{fig:Gamma3}
\end{figure}

\subsubsection{Analytic approximations}
The next two lemmas provide analytic approximations of $r$ and $r_2$; we omit the proof for $r$, but give all details in the more difficult case of $r_2$.

\begin{lemma}[Analytic approximation of $r$]\label{lem_an_appr_r1}
	There exist continuous functions 
	\begin{equation*}
		r_{a}\colon  I_{\mathcal M} \times  (0,\infty) \times \overbar{U^{(3)}_1} \to \C \quad \text{and} \quad r_{r}\colon  I_{\mathcal M} \times  (0,\infty) \times (k_0,\infty)  \to \C,
	\end{equation*}
	which satisfy the following properties.
	\begin{enumerate}[\upshape (i)]
		\item \label{lem_an_appr_r1_i}
		$r(k) = r_{a}(\xi,t,k) +  r_{r}(\xi,t,k)$ for all $(\xi,t,k) \in  I_{\mathcal M} \times (0,\infty) \times (k_0,\infty)$.
		
		\item \label{lem_an_appr_r1_ii}
		For all $(\xi,t)\in  I_{\mathcal M} \times (0,\infty)$, the function $r_{a}(\xi,t,\cdot)\colon U^{(3)}_1\to \C$ is analytic. Moreover, there exists a constant $C>0$ such that for all $(\xi,t,k) \in  I_{\mathcal M} \times (0,\infty) \times \overbar{U^{(3)}_1}$, 
		\begin{equation*}
			\big|r_{a}(\xi,t,k) - r(k_0) - r'(k_0) (k-k_0) \big| \leq  C \, |k-k_0|^2 \, \e^{\frac{t}{4} | \im g(\xi,k)|},
		\end{equation*}
		and 
		\begin{equation*}
			\big|r_{a}(\xi,t,k) \big| \leq C \,  \frac{\e^{\frac{t}{4} | \im g(\xi,k)|}}{1+|k|^2}. 
		\end{equation*}
		\item \label{lem_an_appr_r1_iii}
		For all $(\xi,t)\in   I_{\mathcal M} \times  (0,\infty)$, the function $r_{r}(\xi,t,\cdot)$ lies in $L^p(k_0,\infty)$, $p=1,2,\infty$, and 
		\begin{equation*}
			\|r_{r}(\xi,t,\cdot)  \|_{L^p(k_0,\infty)} \in \bigO(t^{-2}) \quad \text{as} \quad t\to\infty
		\end{equation*}
		uniformly for $\xi\in I_{\mathcal M}$.	
	\end{enumerate}
\end{lemma}

\begin{lemma}[Analytic approximation of $r_2$]\label{lem_an_appr_r2}
	There exist continuous functions 
	\begin{equation*}
		r_{2,a}\colon  I_{\mathcal M} \times  (0,\infty) \times \overbar{U^{(3)}_2} \setminus\{E_1\} \to \C 
		\quad \text{and} \quad 
		r_{2,r}\colon  I_{\mathcal M} \times  (0,\infty) \times (-\infty,E_1)  \to \C,
	\end{equation*}
	which satisfy the following properties.
	\begin{enumerate}[\upshape (i)]
		\item \label{lem_an_appr_r2_i}
			$r_2(k) = r_{2,a}(\xi,t,k) +  r_{2,r}(\xi,t,k)$ for all $(\xi,t,k) \in  I_{\mathcal M} \times (0,\infty) \times (-\infty,E_1)$.
			
		\item  \label{lem_an_appr_r2_ii}
		For all $(\xi,t)\in  I_{\mathcal M} \times (0,\infty)$, the function $r_{2,a}(\xi,t,\cdot)\colon U^{(3)}_2\to \C$ is analytic, and $r_{2,a}(\xi,t,k) = \bigO(|k-E_1|^{-1/2})$ uniformly for $(\xi,t)\in  I_{\mathcal M} \times (0,\infty)$ as $k\to E_1$.
		Moreover, for every $\epsilon>0$ there exists a constant $C_\epsilon>0$ such that 
		\begin{equation} \label{lem_an_appr_r2_ii_1} 
			\big|r_{2,a}(\xi,t,k) \big| \leq C_\epsilon \,  \frac{\e^{\frac{t}{4} | \im g_+(\xi,k)|}}{1+|k|^2}\quad \text{for } (\xi,t,k) \in  I_{\mathcal M} \times (0,\infty) \times \big( \overbar{U^{(3)}_2}\setminus D_\epsilon(E_1)\big).
		\end{equation}
		
		\item \label{lem_an_appr_r2_iii}
		For all $(\xi,t)\in   I_{\mathcal M} \times  (0,\infty)$ and $p=1,2,\infty$, the function $r_{2,r}(\xi,t,\cdot)$ lies in $L^p(-\infty,E_1)$ and 
		\begin{equation*}
			\|r_{2,r}(\xi,t,\cdot)  \|_{L^p(-\infty,E_1)} \in \bigO(t^{-2}) \quad \text{as} \quad t\to\infty
		\end{equation*}
		uniformly for $\xi\in I_{\mathcal M}$.	
		
		\item \label{lem_an_appr_r2_iv}
		For all $(\xi,t)\in   I_{\mathcal M} \times  (0,\infty)$ and $p=1,2,\infty$, the function $\big(r r_{2,a} + \overbar{r r_{2,a}} + 1\big)(\xi,t,\cdot) \e^{-2 i t g_+(\xi,\cdot)}$ lies in $L^p(E_1,\mu_-)$, and 
		\begin{equation}\label{rr2aestimate}
			\big\|\big(r r_{2,a} + \overbar{r r_{2,a}} + 1\big)(\xi,t,\cdot) \e^{-2 i t g_+(\xi,\cdot)} \big\|_{L^p(E_1,\mu_-)} 
			= \bigO(t^{-2}) \quad \text{as} \quad t\to\infty
		\end{equation}
		uniformly for $\xi\in I_{\mathcal M}$.
	\end{enumerate}
\end{lemma}
\begin{remark}
Of particular interest in the statement of Lemma \ref{lem_an_appr_r2} is equation \eqref{rr2aestimate} which provides an estimate of a certain $L^p$-norm on part of the branch cut originating from the nonzero boundary conditions. This estimate ensures that the last entry of the jump matrix $v^{(3)}$, defined below in Section \ref{sec_m^3}, has the desired decay on $(E_1, \mu_-)$ as $t\to\infty$. 
\end{remark}
\begin{proof}[Proof of Lemma \ref{lem_an_appr_r2}]
By the assumptions of Theorem \ref{thm_sec_M} we infer from Theorem \ref{thm_direct} that $r_2$ lies in $\mathcal C^8\big((-\infty,E_1)\big)$ and satisfies
\begin{align}\label{expansionr2}
	r_2(k)=
	\begin{cases} 
		\sum_{l=-1}^4(-1)^l{Q_{1,l}}(E_1-k)^{l/2}+ \bigO\big((E_1-k)^{\frac{5}{2}}\big),&k\nearrow E_1, \\
		\bigO(k^{-4}),&k\to-\infty.
	\end{cases}
\end{align}
In view of \eqref{coeffqjl},
the coefficients $Q_{1,l}$ are rational functions of $q_{1,l}$ with nonzero denominators of the form 
$(q_{1,0} \overbar{q_{1,1}} + \overbar{q_{1,0}} q_{1,1})^{l+2}\neq 0$, $-1\leq l\leq 5$.
In particular, we note that
\begin{align}\label{leadingorderr2}
	Q_{1,-1}= -\frac{\overbar{q_{1,0}}}{q_{1,0} \overbar{q_{1,1}} + \overbar{q_{1,0}} q_{1,1}}.
\end{align}
Next, we extend the function $r_2$ to the interval $(E_1,\mu_-)$ by setting
\begin{equation*}
	r_2(k) \coloneqq \sum_{l=-1}^5 i^l{Q_{1,l}}(k-E_1)^{l/2}, \qquad k\in (E_1,\mu_-),
\end{equation*}
and define $ h  \colon I_{\mathcal M}\times \overbar{U^{(3)}_2}\setminus\{E_1\} \to \C$ by
\begin{align}\label{hdef}
	h (\xi,k)\coloneqq\frac{p_1(\xi,k)}{X_+(\xi,k)}+p_2(\xi,k),
\end{align}
with
\begin{equation*}
	p_1(\xi,k)\coloneqq \sum_{j=3}^{11}\frac{a_j(\xi)}{(k+i)^j}, \qquad p_2(\xi,k)\coloneqq\sum_{j=4}^{12}\frac{b_j(\xi)}{(k+i)^j}.
\end{equation*}
We choose $a_j(\xi)$ and $b_j(\xi)$ such that $r_2(k)- h (\xi,k)=\bigO\big((k-\mu_-)^6\big)$ as $k\nearrow \mu_-$,  
\begin{align}\label{p1xikXplus}
	\frac{p_1(\xi,k)}{X_+(\xi,k)}=
	\begin{cases} 
		\sum_{l\,\mathrm{odd}=-1}^4 i^l{Q_{1,l}}(k-E_1)^{l/2}+\bigO((k-E_1)^{5/2}),&k \searrow E_1,
		\\ 
		\sum_{l\,\mathrm{odd}=-1}^4 (-1)^l{Q_{1,l}}(E_1-k)^{l/2}+\bigO((E_1-k)^{5/2}),&k \nearrow E_1,
		\\
		\bigO(k^{-4}),&k\to-\infty,
	\end{cases}
\end{align}
and
\begin{align}\label{p2xik}
	p_2(\xi,k)=
	\begin{cases}
		\sum_{l\,\mathrm{even}=-1}^4 i^l{Q_{1,l}}(k-E_1)^{l/2}+\bigO((k-E_1)^{5/2}),&k \searrow E_1,
		\\ 
		\sum_{l\,\mathrm{even}=-1}^4 (-1)^l{Q_{1,l}}(E_1-k)^{l/2}+\bigO((E_1-k)^{5/2}),&k \nearrow E_1,
		\\
		\bigO(k^{-4}),&k\to-\infty,
	\end{cases}
\end{align}
where the sums in \eqref{p1xikXplus} run over the odd values $l = -1,1,3$, while the sums in  \eqref{p2xik} run over the even values $l = 0,2,4$.  
The coefficients $a_j$ and $b_j$ are uniquely determined by the above requirements and it is clear that they are uniformly bounded for $\xi\in I_{\mathcal M}$.
By construction, $h$ obeys the following properties: 
\begin{enumerate}[(i)]
	\item \label{tilde_f_i}
	uniformly for $\xi \in  I_{\mathcal M}$, $ h (\xi,k)=\bigO(|k-E_1|^{-1/2})$ as $k\to E_1$;
	\item \label{tilde_f_ii}
	for every $\epsilon>0$ there exists a constant $C_\epsilon >0$ such that
	\begin{equation}  \label{pr_lem_an_appr_r2_0}
		| h (\xi,k)| \leq \frac{C_\epsilon}{1+|k|^2} \quad \text{for all} \; k\in \overbar{\C^+}\setminus D_\epsilon(E_1),\, \xi\in I_{\mathcal M};
	\end{equation}
	\item \label{tilde_f_iii}
	the function $f\colon  I_{\mathcal M}\times  (-\infty,\mu_-) \to \C$ defined by $f(\xi,k)\coloneqq r_2(k) -  h (\xi,k)$ satisfies
	\begin{equation*}
		f^{(n)}(\xi,k) =
		\begin{cases}
			\bigO(|k-\mu_-|^{6-n}), & k\to \mu_-,\\
			\bigO(k^{-4}), & k\to -\infty, \\
			\bigO(|k-E_1|^{5/2-n}), & k\to E_1,
		\end{cases}
	\end{equation*}
	for $n=0,1,2,3$, uniformly for $\xi\in I_{\mathcal M}$; 
	\item \label{tilde_f_iv}
	uniformly for $\xi\in I_{\mathcal M}$,
	\begin{equation}\label{tilde_f_at_E1}
		|r(k)  h (\xi,k)+\overbar{r(k)  h (\xi,k)}+1|\in \bigO(|k-E_1|),\qquad (E_1,\mu_-)\ni k\searrow E_1.
	\end{equation}
\end{enumerate}
The estimate in \eqref{tilde_f_at_E1} can be verified by employing the properties of the coefficients $q_{1,l}$ in \eqref{coeffqjl}.

Next we define the function $\phi$ by
\begin{align*}
	\phi \colon I_{\mathcal M}\times   \R\to\R, \qquad   \phi(k) \coloneqq
	\begin{cases}
		g(k) & k \leq E_1 \\
		\im g_+(k) & E_1 <k \leq \mu_- \\
		-\im g_+(k) + 2 \im g_+(\mu_-) & \mu_- < k \leq k_0 \\
		- g(k) + 2 \im g_+(\mu_-) &k_0<k;
	\end{cases}
\end{align*}
its graph is depicted in Figure~\ref{fig:phi}.
\begin{figure}[h] \centering	
	\begin{overpic}[width=.45\textwidth]{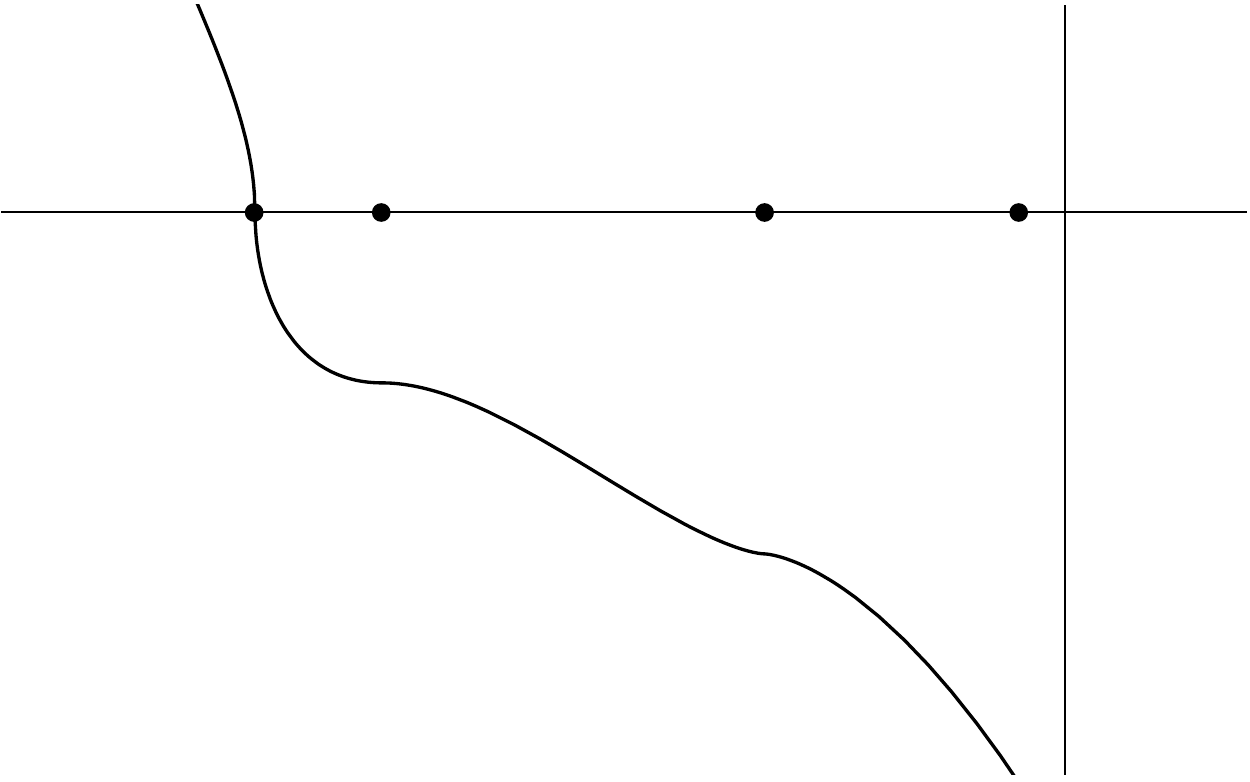}
		\put(13,47.5){\footnotesize{$E_1$}} 
		\put(30,47.5){\footnotesize{$\mu_-$}}  
		\put(62,47.5){\footnotesize{$k_0$}}     
		\put(77,47.5){\footnotesize{$E_2$}} 
		\put(96,47.5){\footnotesize{$0$}} 
	\end{overpic}
	\caption{The graph of $\phi$ for a particular choice of values $\alpha$, $\beta$, and $\xi$.} 
	\label{fig:phi}
\end{figure}
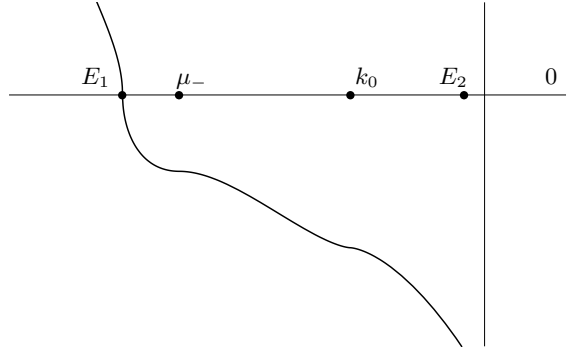
Clearly, $\phi(\xi,\cdot)$ is bijective and it is smooth on $\R\setminus\{E_1,\mu_-,k_0\}$. Furthermore,
\begin{align}
	\begin{aligned} \label{phi_asymp}
		\phi(k) &= 
		\begin{cases}
			\bigO ( |k-E_1|^{1/2}), & k \to E_1, \\
			\bigO (1), & k \to \mu_-, \\
			\bigO (k^{2}), & k\to\pm\infty,
		\end{cases}
		\qquad \,\,
		\phi'(k) = 
		\begin{cases}
			\bigO ( |k-E_1|^{-1/2}), & k \to E_1, \\
			\bigO (|k-\mu_-|), & k \to \mu_-, \\
			\bigO (k^{1}), & k\to\pm\infty,
		\end{cases}
		\\
		\phi''(k) &= 
		\begin{cases}
			\bigO ( |k-E_1|^{-3/2}), & k \to E_1, \\
			\bigO (1), & k \to \mu_-, \\
			\bigO (1), & k\to\pm\infty,
		\end{cases}
		\quad 
		\phi'''(k) = 
		\begin{cases}
			\bigO ( |k-E_1|^{-5/2}), & k \to E_1, \\
			\bigO (1), & k \to \mu_-, \\
			\bigO (k^{-4}), & k\to\pm\infty,
		\end{cases}
	\end{aligned}
\end{align}
uniformly for $\xi\in I_{\mathcal M}$.
Let
\begin{equation} \label{pr_lem_an_appr_r2_1}
	F\colon I_{\mathcal M} \times \R\to\R, \quad F(\xi,\phi)\coloneqq 
	\begin{cases}
		\frac{(k+i)^3}{k-E_1} f(k) &\text{for}\; \phi> \phi(\mu_-),  \\
		0 &\text{for}\; \phi \leq \phi(\mu_-).
	\end{cases}
\end{equation}
Then by the properties of $f$ (cf.~item \eqref{tilde_f_iii} above) and the asymptotics of $\phi$ in \eqref{phi_asymp} it holds that
\begin{equation*} 
	F^{(n)}(\phi) = \frac{\dx^n}{\dx  \phi^n} F(\phi) =
	\begin{cases}
		\bigO (|k-\mu_-|^{6- 2n}), & k\to \mu_-, \\
		\bigO (k^{-2}), & k\to-\infty, \\
		\bigO ( |k-E_1|^{\frac{3-n}{2}}), & k \to E_1,
	\end{cases} 
	\qquad n=0,1,2,3,
\end{equation*}
uniformly for $\xi\in I_{\mathcal M}$.
In particular, $F(\xi,\cdot)$ lies in the Sobolev space $H^3(\R)$ with a uniform bound for $\xi\in I_{\mathcal M}$.
Writing
\begin{equation} \label{pr_lem_an_appr_r2_2}
	F(\phi) = \int_\R \hat F(s) \e^{i \phi s} \, \dx s,
\end{equation}
where $\hat F$ denotes the Fourier transform of $F$, i.e.
\begin{equation*}
	\hat F(s) = \frac{1}{2\pi}  \int_\R  F(\phi) \e^{-i \phi s} \, \dx \phi,
\end{equation*}
we deduce that $\|  s^3 \hat F(s) \|_{L^2(\R)} \leq C < \infty$ for some $C>0$ uniformly for all $\xi\in I_{\mathcal M}$. 
By \eqref{pr_lem_an_appr_r2_1} and \eqref{pr_lem_an_appr_r2_2} it follows that 
\begin{equation*}
	f(\xi,k) = \frac{k-E_1}{(k+i)^3} \int_\R \hat F(\xi,s) \e^{i s  \phi(\xi,k)} \, \dx s, \quad \text{for} \quad k\in (-\infty,\mu_-).
\end{equation*}	

Let 
\begin{equation*}
	f_{a}\colon  I_{\mathcal M} \times  (0,\infty) \times \overbar{U^{(3)}_2} \to \C \quad \text{and} \quad f_{r}\colon  I_{\mathcal M} \times  (0,\infty) \times (-\infty,\mu_-)  \to \C
\end{equation*}
be defined by 
\begin{align*}
	f_{a}(\xi,t,k) &\coloneqq \frac{k-E_1}{(k+i)^3} \int^{t/4}_{-\infty} \hat F(\xi,s) \e^{i s g_+(\xi,k)} \, \dx s,
	\\
	f_{r}(\xi,t,k) &\coloneqq \frac{k-E_1}{(k+i)^3} \int^{\infty}_{t/4} \hat F(\xi,s) \e^{i s g_+(\xi,k)} \, \dx s.
\end{align*}
Then both $f_a$ and $f_r$ are continuous,  $f_{a}(\xi,t,\cdot)\colon U^{(3)}_2\to \C$ is analytic for all $(\xi,t)\in  I_{\mathcal M} \times (0,\infty)$, and  by the definition of $\phi$ it follows that
\begin{equation*}
	f(\xi,k) = f_{a}(\xi,t,k) + f_{r}(\xi,t,k)
\end{equation*}
for all $(\xi,t,k)\in  I_{\mathcal M} \times (0,\infty) \times  (-\infty,\mu_-)$. 
Furthermore, for all $(\xi,t,k) \in  I_{\mathcal M} \times (0,\infty) \times \overbar{U^{(3)}_2}$ it holds that
\begin{equation} \label{pr_lem_an_appr_r2_3}
		|f_a(\xi,t,k)| \leq \frac{|k-E_1|}{|k+i|^3} \, \|\hat F \|_{L^1(\R)} \, \sup_{s\geq -t/4} | \e^{i s  g_+(\xi,k)}| 
		\leq \frac{C}{1+|k|^2} \, \e^{\frac{t}{4} |\im g_+(\xi,k)|},
\end{equation}
and for all $(\xi,t,k) \in  I_{\mathcal M} \times (0,\infty) \times  (-\infty,\mu_-)$ the Cauchy-Schwartz inequality yields that
\begin{equation*} 
	|f_r(\xi,t,k)| \leq  \frac{|k-E_1|}{|k+i|^3} \, \|s^3 \hat F(\xi,s)\|_{L^2(\R)} \, \sqrt{ \int^{-t/4}_{-\infty} s^{-6} \, \dx s}
	\leq \frac{C}{1+|k|^2} \, t^{-2}.
\end{equation*}
Thus, for all $(\xi,t)\in  I_{\mathcal M} \times (0,\infty)$ the function $f_r(\xi,t,\cdot)\colon  (-\infty,\mu_-)\to\C$ lies in $L^p (-\infty,\mu_-)$, $p=1,2,\infty$, and the corresponding $L^p$-norms are of order $\bigO (t^{-2})$ uniformly with respect to $\xi \in I_{\mathcal M}$.

Defining 
\begin{equation*}
	r_{2,a}\colon  I_{\mathcal M} \times  (0,\infty) \times \overbar{U^{(3)}_2} \to \C \quad \text{and} \quad r_{2,r}\colon  I_{\mathcal M} \times  (0,\infty) \times  (-\infty,\mu_-)  \to \C,
\end{equation*}
by
\begin{equation*}
	r_{2,a}(\xi,t,k)\coloneqq  h (\xi,k) + f_a(\xi,t,k) \quad \text{and} \quad r_{2,r}(\xi,t,k) \coloneqq f_r(\xi,t,k)
\end{equation*}
yields the desired decomposition of $r_2$, which clearly satisfies the assertions \eqref{lem_an_appr_r2_i} and \eqref{lem_an_appr_r2_iii} of the lemma.
Furthermore, property \eqref{lem_an_appr_r2_ii} follows from \eqref{pr_lem_an_appr_r2_0} in conjunction with \eqref{pr_lem_an_appr_r2_3}:
for all $\epsilon>0$ there exists a constant $C_\epsilon>0$ such that
\begin{equation*}
	\big|r_{2,a}(\xi,t,k) \big| 
	\leq \big|  h (\xi,k) \big| + \big| f_a(\xi,t,k) \big|
	\leq \frac{C_\epsilon}{1+|k|^2} \big(1+ \e^{\frac{t}{4} |\im g(\xi,k)|}\big)
\end{equation*}
holds uniformly for all $(\xi,t,k) \in  I_{\mathcal M} \times (0,\infty) \times \overbar{U^{(3)}_2}\setminus D_\epsilon(E_1)$.
To show assertion \eqref{lem_an_appr_r2_iv}, we recall that 
\begin{align}\label{f_a_near_E1}
	|f_a(\xi,k)| 
	\leq C |k-E_1| \e^{\frac t4|\im g_+(\xi,k)|}, \qquad  k\in (E_1,\mu_-),
\end{align}
uniformly for $\xi\in I_{\mathcal M}$, cf.~\eqref{pr_lem_an_appr_r2_3}.
Applying the property \eqref{tilde_f_iv} of $ h $ and \eqref{f_a_near_E1} to the identity
\begin{align*}
	rr_{2,a}+\overbar{rr_{2,a}}+1=r  h +\overbar{r  h }+1+2\re(r f_a)
\end{align*}
and employing that $|r|=1$ on $[E_1,\mu_-]$ implies that
\begin{align*}
	\big|(r(k)r_{2,a}(x,t,k)+\overbar{r(k)r_{2,a}(x,t,k)}+1) \e^{-2i tg_+(\xi,k)}\big|
	\leq C |k-E_1|\e^{-\frac{7t}4|\im g_+(\xi,k)|} 
\end{align*}
for $k\in(E_1,\mu_-)$ and uniformly with respect to $\xi\in I_{\mathcal M}$.
Since $|\im g_+(\xi,k)|\in\bigO(|k-E_1|^{1/2})$ as $k \searrow E_1$ uniformly in $\xi\in I_{\mathcal M}$, assertion \eqref{lem_an_appr_r2_iv} follows.
\end{proof}

\begin{remark}\label{rem_r2a}
	The proof of Lemma \ref{lem_an_appr_r2} does not rely on an explicit expression of $r_2$ on $(E_1,E_2)$ in order to construct a function $r_{2,a}$ with the desired properties. This generalizes a related construction in \cite{Fromm19}, where $r_2$ was defined on $(E_1,E_2)$ in terms of $r$ by $r_2(k)\coloneqq -r(k)/(1+r(k^2))$.
	For consistency, additional assumptions were made in \cite{Fromm19}, namely that $r\neq\pm i$ on $(E_1,E_2)$ and $r(E_j)\in\{\pm i\}$, $j=1,2$. However, numerical experiments have shown that these assumptions are not satisfied in general. The above proof has the advantage that no additional assumptions are required. 
\end{remark}

\subsubsection{Definition of $m^{(3)}$}\label{sec_m^3}
For $(x,t,k)\in\mathcal M\times \C\setminus \Gamma^{(3)}$ we define $m^{(3)}$ by 
\begin{equation*} 
	m^{(3)} \coloneqq m^{(2)} \, \mathcal{D}^{\sigma_3} G^{(3)} \mathcal{D}^{-\sigma_3}
	= \mathcal{D}_\infty^{\sigma_3}  m^{(1)} G^{(3)} \mathcal{D}^{-\sigma_3},
\end{equation*} 
where
\begin{equation*}
	G^{(3)} \coloneqq
	\begin{cases}
		\begin{pmatrix}
			1 & 0 \\
			r_{a} \e^{2 i t g} & 1
		\end{pmatrix} &
		\text{in} \quad U^{(3)}_1, \\  
		\begin{pmatrix}
			1 &   - r_{2,a} \e^{-2 i t g} \\
			0 & 1
		\end{pmatrix} &
		\text{in} \quad U^{(3)}_2, \\  
		\begin{pmatrix}
			1 & 0 \\
			-    r^*_{2,a} \e^{2 i t g} & 1 
		\end{pmatrix} &
		\text{in} \quad U^{(3)}_3,  \\ 
		\begin{pmatrix}
			1 &   r^*_{a} \e^{-2 i t g} \\
			0 & 1
		\end{pmatrix} &
		\text{in} \quad U^{(3)}_4, 
		\\
		I, & \text{elsewhere},
	\end{cases}
\end{equation*}
and we write $f^*$ for the Schwartz conjugate of a function $f$, i.e., $f^*(k) = \overbar{f(\bar{k})}$.
The regions $U^{(3)}_j$, $j=1,2,3,4$, and the jump contour $\Gamma^{(3)}$ for $m^{(3)}$  are depicted in Figure~\ref{fig:Gamma3}.
As a consequence of Lemmas \ref{lem_D}, \ref{lem_an_appr_r1}, and \ref{lem_an_appr_r2}, 
the function $m^{(3)}$ satisfies the following RH problem for $(x,t)\in\mathcal M$:
\begin{itemize}
\item $m^{(3)}(x,t,\cdot)$ is analytic in $\C \setminus \Gamma^{(3)}$;

\item $m^{(3)}(x,t,\cdot)$ has continuous boundary values on $ \Gamma^{(3)} \setminus \{E_1, \mu_-, k_0\}$ satisfying the jump relation
\begin{align*}
m^{(3)}_+(x,t,k) = m^{(3)}_-(x,t,k) v^{(3)}(x,t,k) \quad \text{for} \; k \in \Gamma^{(3)}\setminus \{E_1,  \mu_-, k_0\},
\end{align*}
where $v^{(3)} =   \mathcal{D}^{\sigma_3}_- \big(G^{(3)}_- \big)^{-1} \mathcal{D}^{-\sigma_3}_- \, v^{(2)} \, \mathcal{D}^{\sigma_3}_+ G^{(3)}_+ \mathcal{D}^{-\sigma_3}_+$;

\item $m^{(3)}(x,t,k)=I+ \bigO(k^{-1})$ as $k\to \infty$;

\item $m^{(3)}(x,t,k)= \bigO((k-E_1)^{-1/4})$ as $k\to E_1$;

\item $m^{(3)}(x,t,k)=  \bigO(1)$ as $k\to k_0$;

\item $m^{(3)}(x,t,k)=  \bigO(1)$ as $k\to \mu_-$.
\end{itemize}
Using \eqref{calDatE1}, we can describe the behavior of $m^{(3)}$ at $E_1$ more precisely as follows:
\begin{align}\label{m3atE1}
m^{(3)}= \mathcal{D}_\infty^{\sigma_3}  m^{(1)} G^{(3)} \times \begin{cases} 
(k-E_1)^{-\sigma_3/4} \e^{-d_0\sigma_3}\big(I+\bigO(\sqrt{k-E_1})\big), & k\to E_1, ~ k \in \C_+,
	\\
(k-E_1)^{\sigma_3/4} \e^{\overbar{d_0}\sigma_3}\big(I+\bigO(\sqrt{k-E_1})\big), & k\to E_1, ~ k \in \C_-.
\end{cases}
\end{align}

Let us write $\Gamma^{(3)} = \cup_{j=1}^8 \Gamma^{(3)}_j$, where $\Gamma^{(3)}_j$ denotes the subcontour labeled by $j$ in Figure \ref{fig:Gamma3}.
Letting $v^{(3)}_j$ denote the restriction of $v^{(3)}$ to $\Gamma^{(3)}_j$, we find
\begin{flalign*}
	\begin{aligned}
		v^{(3)}_1 &= 
		\begin{pmatrix}
			1 & 0 \\
			- \mathcal{D}^{-2} r_{a} \e^{2 i t g} & 1
		\end{pmatrix}  
		&&v^{(3)}_2 = 
		\begin{pmatrix}
			0 &  1  \\
			- 1  & \mathcal{D}_+ \mathcal{D}^{-1}_- \e^{-2 i t g_+}
		\end{pmatrix}\\
		v^{(3)}_3 &= 
		\begin{pmatrix}
			1 & \mathcal{D}^2  r^*_{a} \e^{-2 i t g} \\
			0 & 1
		\end{pmatrix}   
		&&v^{(3)}_4 =   
		\begin{pmatrix}
			1 -|r_{r}|^2 & \mathcal{D}^2 r^*_{r} \e^{-2 i t g} \\
			-\mathcal{D}^{-2} r_{r} \e^{2 i t g} & 1
		\end{pmatrix} \\
		v^{(3)}_{5} &=
		\begin{pmatrix}
			1 & \mathcal{D}^2_+ r_{2,r} \e^{-2 i t g} \\
			-\mathcal{D}^{-2}_- r^*_{2,r} \e^{2 i t g} & 1 - |r_{2,r}|^2 (1-|r|^2)^2
		\end{pmatrix} 
		&&v^{(3)}_{6} =
		\begin{pmatrix}
			0 & 1 \\
			-1 & \mathcal{D}_+ \mathcal{D}^{-1}_- (r r_{2,a} + r^* r^*_{2,a} + 1) \e^{-2 i t g_+}
		\end{pmatrix} \\
		v^{(3)}_{7} &= 
		\begin{pmatrix}
			1 &   - \mathcal{D}^2 r_{2,a} \e^{-2 i t g} \\
			0 & 1
		\end{pmatrix} 
		&&v^{(3)}_{8} =
		\begin{pmatrix}
			1 & 0 \\
			\mathcal{D}^{-2} r^*_{2,a} \e^{2 i t g} & 1 
		\end{pmatrix}
	\end{aligned}
\end{flalign*} 
where the simplification of the $(2,2)$-entry in $v^{(3)}_{6} $ uses that 
$g_+=-g_-$, $\mathcal{D}_+ \mathcal{D}_- =r$, $\mathcal{D}^*_+ \mathcal{D}_- =1$, and $|r|\equiv 1$ on $\Gamma^{(3)}_{6}$.

To leading order, $m^{(3)}$ is approximated by a global parametrix (denoted by $m^{[E_1,k_0]}$) with a constant off-diagonal jump across the branch cut $[E_1,k_0]$. The sub-leading contribution stems from the local long-time behavior near the critical point $k_0$. It turns out that the local parametrix (denoted by $m^{k_0}$) can be constructed in terms of the solution of the well-known Airy RH problem.

\subsection{Global parametrix} \label{sec_global_parametrix}
On the branch cut $(E_1,k_0)$, the jump matrix $v^{(3)}$ approaches 
\begin{equation*} 
	v^{[E_1,k_0]} \coloneqq
	\begin{pmatrix}
		0 & 1 \\
		-1 & 0
	\end{pmatrix},
	\quad k\in (E_1,k_0),
\end{equation*} 
as $t\to\infty$, whereas $v^{(3)} \to I$ for each fixed $k\in \C \setminus [E_1,k_0]$. 
Thus for each $(x,t)\in \mathcal M$ we consider the following RH problem:
\begin{itemize}
\item $m^{[E_1,k_0]}(x,t,\cdot)$ is analytic in $\C\setminus [E_1,k_0]$;

\item $m^{[E_1,k_0]}(x,t,\cdot)$ has continuous boundary values on $(E_1, k_0)$ satisfying the jump relation
\begin{align}\label{RHP_minfty}
m^{[E_1,k_0]}_+(x,t,k) = m^{[E_1,k_0]}_-(x,t,k) v^{[E_1,k_0]}(x,t,k) \quad \text{for} \; k \in (E_1, k_0);
\end{align}

\item $m^{[E_1,k_0]}(x,t,k)=I+ \bigO(k^{-1})$ as $k\to \infty$;

\item $m^{[E_1,k_0]}(x,t,k)= \bigO((k-E_1)^{-1/4})$ as $k\to E_1$;

\item $m^{[E_1,k_0]}(x,t,k)= \bigO((k-k_0)^{-1/4})$ as $k\to k_0$.
\end{itemize}

Let $\Delta_0 \colon  I_{\mathcal M} \times \C \setminus [E_1,k_0]\to\C$ be given by
\begin{equation} \label{def_Delta}
	\Delta_0(\xi,k) = \bigg( \frac{k-k_0}{k-E_1} \bigg)^{\frac{1}{4}},
\end{equation}
where the branch is such that $\Delta_0(k) = 1 + \bigO(k^{-1})$ as $k\to\infty$.
The following result is standard.

\begin{lemma} \label{lem_minfty}
	The unique solution of the RH problem \eqref{RHP_minfty} is given by
	\begin{equation*}
		m^{[E_1,k_0]} \coloneqq
		\frac{1}{2}
		\begin{pmatrix}
			\Delta_0 + \Delta_0^{-1} & - i (\Delta_0 - \Delta_0^{-1}) \\
			 i(\Delta_0 - \Delta_0^{-1}) & \Delta_0 + \Delta_0^{-1}
		\end{pmatrix}
	\end{equation*}
Moreover, the following expansions hold uniformly for $\xi \in  I_{\mathcal M}$:
	\begin{enumerate}[\upshape (i)]
		\item \label{lem_minfty_ii} As $k\to\infty$,
		$m^{[E_1,k_0]} = I + \frac{ i (k_0-E_1)}{4k}
			\begin{pmatrix}
				0 & 1\\
				-1  & 0
			\end{pmatrix}
			+ \bigO (k^{-2})$.

		\item \label{lem_minfty_iii}
		As $k \to k_0$,
		\begin{align}\nonumber
				m^{[E_1,k_0]}(\xi,k) &= \frac{(k_0-E_1)^{1/4}}{2(k-k_0)^{1/4}}\Bigg\{ \begin{pmatrix} 1 &  i \\ - i & 1 \end{pmatrix} +
				\frac{(k-k_0)^{1/2}}{(k_0-E_1)^{1/2}} \begin{pmatrix} 1 & - i \\  i & 1 \end{pmatrix} 
				+ \frac{k-k_0}{4(k_0-E_1)}
				\begin{pmatrix}
					1 &  i \\
					- i & 1 \\
				\end{pmatrix}
				\\ \label{minftyatk0}
				&  \quad
				+\frac{(k-k_0)^{3/2}}{4(k_0-E_1)^{3/2}}
				\begin{pmatrix}
					-1 &  i \\
					- i & -1 \\
				\end{pmatrix}
				+ \bigO(|k-k_0|^2)
				\Bigg\}.
		\end{align} 
	\end{enumerate}
\end{lemma}

\subsection{Local parametrix near the critical point} \label{sec_local_parametrix}
We will define the local parametrix with the help of the explicit solution to the classical Airy model RH problem.

\subsubsection{Airy model problem} \label{sec_Airy_mp}
Consider the contour $\Gamma^{\mathrm{Ai}}$ consisting of the four components 
\begin{align*}
	&\Gamma^{\mathrm{Ai}}_1 = \{z\in \C\colon \arg z = 2\pi /3\},  &&\Gamma^{\mathrm{Ai}}_2 = \{z\in \C\colon \arg z = \pi\}, \\
	&\Gamma^{\mathrm{Ai}}_3 = \{z\in \C\colon \arg z = 4\pi /3\},  &&\Gamma^{\mathrm{Ai}}_4 = \{z\in \C\colon \arg z =0\},
\end{align*}
oriented as in Figure~\ref{fig:Airy}, which separate the sectors
\begin{align*}
	&S_1 = \{z\in \C\colon \arg z  \in (0,2\pi /3) \},  &&S_2 = \{z\in \C\colon \arg z  \in (2\pi /3, \pi) \}, \\
	&S_3 = \{z\in \C\colon \arg z \in (\pi, 4\pi /3,) \}, &&S_4 = \{z\in \C\colon \arg z  \in (4\pi /3, 2\pi) \}.
\end{align*}

\begin{lemma} \label{lem_Airy}
Let $\mathrm{Ai}\colon \C\to\C$ denote the complex Airy function and let $\zeta = \e^{\frac{2\pi i}{3}}$.
	The function $m^{\mathrm{Ai}} \colon \C\setminus \Gamma^{\mathrm{Ai}}\to\C^{2\times2}$ defined by 
	\begin{align*} 
		m^{\mathrm{Ai}}(z)& \coloneqq
		\begin{cases}
			\mathcal A(z) \, \e^{\frac{2}{3}z^{3/2}\sigma_3}, & z\in S_1, \\
			\mathcal A(z) 
			\begin{pmatrix}
				1&0\\-1&1
			\end{pmatrix}
			\e^{\frac{2}{3}z^{3/2}\sigma_3}, \hspace{-.1cm}  & z\in S_2, \\
			\mathcal A(z)  
			\begin{pmatrix}
				1&0\\1&1
			\end{pmatrix} 
			\e^{\frac{2}{3}z^{3/2}\sigma_3}, & z\in S_3, 
			\\
			\mathcal A(z) \, \e^{\frac{2}{3}z^{3/2}\sigma_3}, & z\in S_4,
		\end{cases}
		\quad \mathcal A(z) \coloneqq
		\begin{cases}
			\begin{pmatrix}
				\mathrm{Ai}(z) & \mathrm{Ai}(\zeta^2 z) \\
				\mathrm{Ai}'(z) & \zeta^2 \mathrm{Ai}'(\zeta^2 z)
			\end{pmatrix} 
			\e^{-\frac{\pi i}{6}\sigma_3}, \hspace{-.1cm} & z\in\C_+, \\
			\begin{pmatrix}
				\mathrm{Ai}(z) & -\zeta^2 \mathrm{Ai}(\zeta z) \\
				\mathrm{Ai}'(z) &  -\mathrm{Ai}'(\zeta z)
			\end{pmatrix} 
			\e^{-\frac{\pi i}{6}\sigma_3}, \hspace{-.1cm} & z\in\C_-,
		\end{cases}
	\end{align*}  
	is analytic in $\C\setminus \Gamma^{\mathrm{Ai}}$ and $m^{\mathrm{Ai}}_+(z) = m^{\mathrm{Ai}}_-(z) v^{\mathrm{Ai}}(z)$ for all $z\in \Gamma^{\mathrm{Ai}}$, where $v^{\mathrm{Ai}}$ is given by
	\begin{equation*} 
		v^{\mathrm{Ai}}(z) = 
		\begin{cases}
			\begin{pmatrix}
				1 & 0 \\
				-\e^{\frac{4}{3} z^{3/2}}  & 1
			\end{pmatrix}
			& \text{on} \quad  \Gamma^{\mathrm{Ai}}_1 \cup \Gamma^{\mathrm{Ai}}_3,  \\
			\begin{pmatrix}
				0 & -1 \\
				1  & 0
			\end{pmatrix}
			& \text{on} \quad  \Gamma^{\mathrm{Ai}}_2, \\
			\begin{pmatrix}
				1 &  -\e^{-\frac{4}{3} z^{3/2}}  \\
				0 & 1 
			\end{pmatrix}
			& \text{on} \quad  \Gamma^{\mathrm{Ai}}_4.
		\end{cases}
	\end{equation*} 
	The asymptotic behavior of $m^{\mathrm{Ai}}$ as $z \to \infty$ is given by
	\begin{align}\label{large z asymptotics Airy}
		m^{\mathrm{Ai}}(z) \sim z^{-\frac{\sigma_3}{4}} N \bigg(I + \sum_{j=1}^\infty \frac{m^{\mathrm{Ai}}_j}{z^{3j/2}}\bigg), \qquad z \to \infty,
	\end{align}
	where 
	\begin{equation*} 
		N = \frac{1}{\sqrt{2}}
		\begin{pmatrix} 
			1 &  i \\  i & 1
		\end{pmatrix}, 
		\quad
		m^{\mathrm{Ai}}_j = 
		\frac{\e^{\frac{\pi i}{4}}}{\sqrt{2}}
		N^{-1} 
		\begin{pmatrix} 
			1 & 0 \\ 
			0 & - i 
		\end{pmatrix}
		\bigg(\frac{3}{2}\bigg)^j
		\begin{pmatrix}
			(-1)^j u_j & u_j \\
			-(-1)^j \nu_j & \nu_j
		\end{pmatrix} \e^{-\frac{\pi i}{4}\sigma_3},
	\end{equation*} 
	and the real coefficients 
	$u_j,\nu_j$ are defined by 
	\begin{equation*} 
		u_j=\frac{(2j+1)(2j+3)\cdots(6j-1)}{(216)^jj!}\,,\quad\nu_j=\frac{6j+1}{1-6j}u_j,\quad j \in\N_{\geq 1}. 
	\end{equation*} 	
	That is,
	\begin{equation} \label{def_m^Ai_j}
		m^{\mathrm{Ai}}_j = -\frac{6^{-2 j} (j+\frac{1}{2})_{2 j}}{(6 j-1) j!}
		\begin{pmatrix}
			(-1)^{j} & -6  i j \\
			(-1)^{j}  6  i  j & 1\end{pmatrix}, \qquad j\in \N_{\geq 1},
	\end{equation}
	where the Pochhammer symbol $(a)_j$ is defined by 
	\begin{equation*} 
		(a)_j = \frac{\Gamma(a+j)}{\Gamma(a)} = a(a+1)(a+2) \cdots (a+j-1).
	\end{equation*} 
\end{lemma}

\subsubsection{Local transformation} \label{sec_local_transf}
Our aim is to show that the Airy solution $m^{\mathrm{Ai}}$ yields a good approximation of $m^{(3)}$ locally around $k_0$ when considered in suitable coordinates.  

Recalling the definition \eqref{def_g} of $g$,  
we define the fractional power $g(\xi, k)^{2/3}$ for $k \in \C \setminus (-\infty, k_0]$ with the branch fixed by the requirement that 
$(g(k))^{2/3} > 0$ for $k > k_0$. 
Let 
\begin{equation*} 
	f(\xi,k) \coloneqq  -\bigg(\frac{3}{2} \, g(\xi,k) \bigg)^{\frac{2}{3}}, \quad   k \in \C \setminus (-\infty, k_0].
\end{equation*} 
Then, uniformly for $\xi\in I_{\mathcal M}$,
\begin{align} \label{fk0neark0}
	f(\xi,k) =  - 3^{\frac{2}{3}} (k_0-E_1)^{1/3}(k-k_0)
	\bigg(1 + \frac{k-k_0}{3(k_0-E_1)}
	+ \bigO\big((k-k_0)^{2}\big)\bigg)
\end{align}
as $k \to k_0$, $k\in \C \setminus (-\infty, k_0]$, where  $3^{\frac{2}{3}} (k_0-E_1)^{1/3} > 0$ is bounded from below by a positive constant uniformly for $\xi\in I_{\mathcal M}$.
We introduce the new variable 
\begin{equation} \label{def_z}
	z(\xi,t,k) \coloneqq t^{2/3} f(\xi,k).
\end{equation}
Then
\begin{equation*} 
	\frac{4}{3} z^{3/2}
	= \begin{cases}
		2 i t g(k),  & \im k > 0,	\\
		-2  i t g(k), & \im k < 0,	\\
		2 i t g_+(k) = -2 i t g_-(k),  & k < k_0,	
	\end{cases}
\end{equation*} 
where the cut for $(\cdot)^{3/2}$ runs along $\R_-$ as in the Airy model problem. 

\subsubsection{Definition and properties of $m^{k_0}$} \label{sec_m^k0}
In the following we define the local parametrix $m^{k_0}$.
Let $D_\varepsilon(k_0)$ denote the open disc of radius $\varepsilon$ around $k_0$ and let $\varepsilon>0$ be so small that 
\begin{equation} \label{Gamma_eps}
	\bigcup^4_{j=1} \Big( \overbar{\Gamma^{(3)}_j} \cap D_\varepsilon(k_0) \Big) = \Gamma^{(3)}\cap D_\varepsilon(k_0)
\end{equation} 
is satisfied for all $\xi\in  I_{\mathcal M}$. The existence of such a uniform $\varepsilon$ is guaranteed by the definition of $ I_{\mathcal M}$.  We denote the contour in \eqref{Gamma_eps} by $\Gamma^{k_0,\varepsilon}$, see Figure~\ref{fig:Gamma_eps}.

Let $\mathfrak r\colon  I_{\mathcal M} \times \C \to \C$ denote the $(N-1)$-th degree Taylor polynomial of $r$ at $k_0$:
\begin{equation} \label{frak_r}
	\mathfrak r(\xi,k)\coloneqq \sum^{N-1}_{n=0} \frac{r^{(n)}(k_0)}{n!}(k-k_0)^n.
\end{equation} 
According to the assumptions on the initial datum $u_0$ it holds that $r(k)=\mathfrak r(\xi,k) + \bigO(|k-k_0|^N)$ as $k\to k_0$ for $1\leq N\leq7$. 
In order to obtain a good enough local approximation for $m^{(3)}$ near $k_0$ we need $N\geq 2$, thus we set $N\coloneqq 2$. 
Furthermore, we define 
$\mathfrak{D} \colon  I_{\mathcal M} \times \overbar{D_\varepsilon(k_0)}\setminus [k_0-\varepsilon,k_0+\varepsilon] \to \C$ by
\begin{align*}
	\mathfrak{D}(\xi,k) \coloneqq  
	\e^{\frac{\mathcal{X}(\xi,k)}{2 \pi  i} \Big[
		\int^{k_0}_{E_1} \frac{\frac{1}{2}\log(\mathfrak r(\xi,s)/\overbar{\mathfrak r(\xi,s)})}{\mathcal{X}_+(\xi,s)(s-k)} \, \dx s 
		+   \int^{E_2}_{k_0} \frac{\log|\mathfrak r(\xi,s)|}{s-k} \, \dx s 
		- \sum^{N-1}_{n=0} c_n(\xi) (k-k_0)^n \Big]}
\end{align*}
where the real coefficients $\{c_n(\xi)\}^{N-1}_{n=0}$ are uniquely determined by the requirement that
\begin{equation} \label{frakD_asympt}
	\mathcal{D}(\xi,k) = \mathfrak{D}(\xi,k) + \bigO(|k-k_0|^N) \quad \text{as} \quad k\to k_0
\end{equation} 
uniformly for $\xi\in I_{\mathcal M}$, cf.~Lemma \ref{lem_E}.

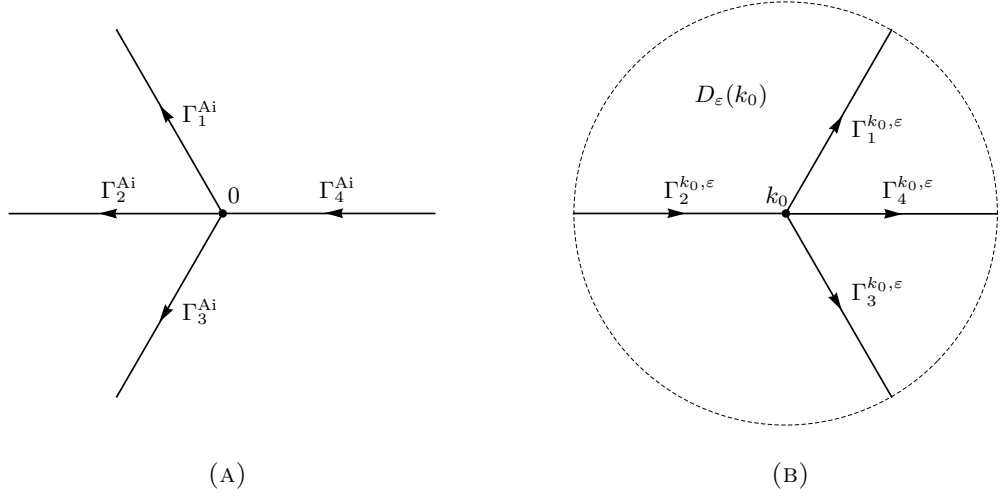
\begin{figure}[h] \centering
	\begin{subfigure}{.43\textwidth}  \centering
		\begin{overpic}[width=.95\textwidth]{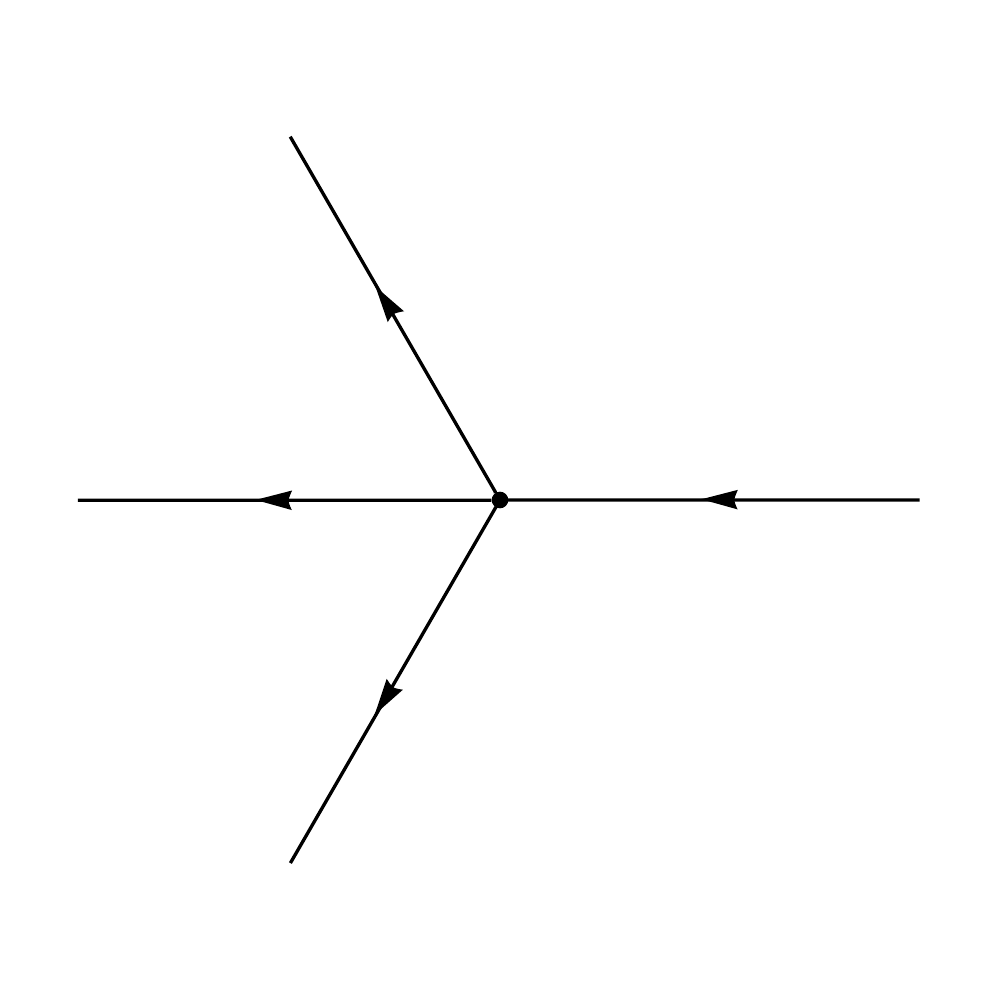}
			\put(51,52){\footnotesize{$0$}} 
			\put(26,53){\footnotesize{$\Gamma^{\mathrm{Ai}}_2$}} 
			\put(69,53){\footnotesize{$\Gamma^{\mathrm{Ai}}_4$}} 
			\put(42,68){\footnotesize{$\Gamma^{\mathrm{Ai}}_1$}}		
			\put(42,29){\footnotesize{$\Gamma^{\mathrm{Ai}}_3$}}	
		\end{overpic}
		\vspace{-.3cm}
		\caption{}
		\label{fig:Airy}
	\end{subfigure}%
	\quad
	\begin{subfigure}{.43\textwidth}  \centering
		\begin{overpic}[width=.95\textwidth]{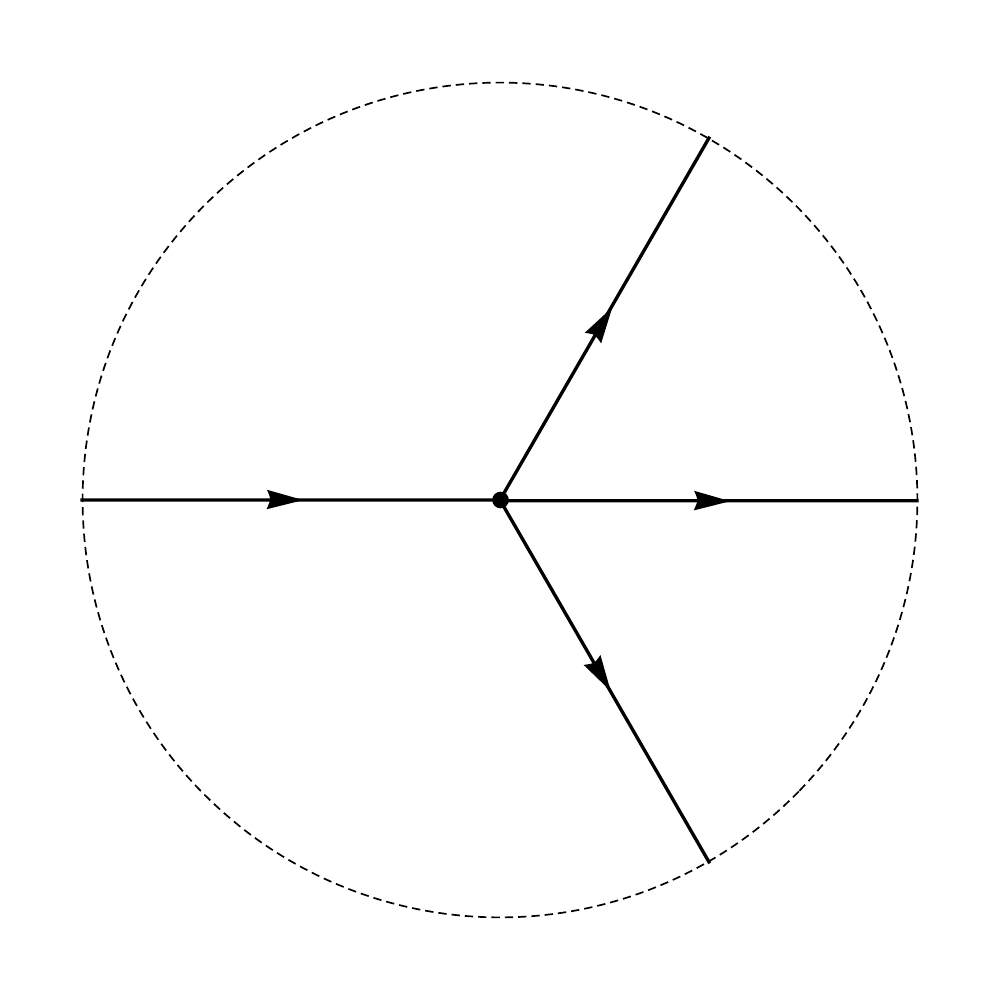}
			\put(46,52){\footnotesize{$k_0$}} 
			\put(26,53){\footnotesize{$\Gamma^{k_0,\varepsilon}_2$}} 
			\put(69,53){\footnotesize{$\Gamma^{k_0,\varepsilon}_4$}} 
			\put(63,65){\footnotesize{$\Gamma^{k_0,\varepsilon}_1$}}		
			\put(63,33){\footnotesize{$\Gamma^{k_0,\varepsilon}_3$}}	
			\put(32,72){\footnotesize{$D_\eps(k_0)$}}
		\end{overpic}
		\vspace{-.3cm}
		\caption{}
		\label{fig:Gamma_eps}
	\end{subfigure}%
	\caption{The jump contour $\Gamma^{\mathrm{Ai}}$ of the classical Airy RH problem  (left) and the jump contour $\Gamma^{k_0,\varepsilon}\subseteq D_\eps(k_0)$ of the local parametrix $m^{k_0}$ (right).} 
	\label{fig:Gamma_eps_tilde}
\end{figure}

We define the local parametrix $m^{k_0}\colon  I_{\mathcal M} \times (0,\infty) \times \overbar{{D_\varepsilon(k_0)}} \setminus \overbar{\Gamma^{k_0,\eps}} \to \C^{2\times2}$ by
\begin{equation} \label{def_m^{k_0}}
	m^{k_0}(\xi,t,k) \coloneqq E(\xi,t,k) \, m^{\mathrm{Ai}}\big(z(\xi,t,k)\big) \,   G(\xi,k)
\end{equation}
where
\begin{align*}
	G(\xi,k) &\coloneqq
	\begin{cases}
		G_1(\xi,k) \coloneqq	\e^{\frac{ i \pi \sigma_3}{2}} \mathfrak r(\xi,k)^{\frac{\sigma_3}{2}} \mathfrak{D}(\xi,k)^{-\sigma_3} \sigma_3, & k\in \C_+\cap \overbar{D_\varepsilon(k_0)}, \\
		G_2(\xi,k) \coloneqq	 i \mathfrak r^*(\xi,k)^{\frac{\sigma_3}{2}} \mathfrak{D}(\xi,k)^{\sigma_3}\sigma_1 \sigma_3, &  k\in \C_-\cap \overbar{D_\varepsilon(k_0)}, 
	\end{cases}
\end{align*}
and
\begin{align*}
	E(\xi,t,k) &\coloneqq  m^{[E_1,k_0]}(\xi,k) \, G^{-1}(\xi,k) \, N^{-1} \, z(\xi,t,k)^{\frac{\sigma_3}{4}}.
\end{align*}
The matching factor $E$ is chosen such that $m^{k_0} \big(m^{[E_1,k_0]}\big)^{-1} \to I$ uniformly on $ I_{\mathcal M} \times \partial D_\varepsilon(k_0)$ as $t\to\infty$, cf.~Lemma \ref{lem_m^k_0}\eqref{lem_m^k_0_ii}.

\begin{lemma} \label{lem_E}
	The following assertions hold true:	
	\begin{enumerate}[\upshape (i)]
		\item \label{lem_E_i}
		there exist unique $\{c_n(\xi)\colon  I_{\mathcal M} \to \R\}^{N-1}_{n=0}$ such that \eqref{frakD_asympt} holds with a uniform error term for all $\xi$;
		\item \label{lem_E_ii}
		for each $\xi\in I_{\mathcal M}$ and $t>0$ the function $E(\xi,t,\cdot)\colon \overbar{D_\varepsilon(k_0)} \to \C$ is analytic in $D_\varepsilon(k_0)$.
		Moreover, $E(\xi,t,k)$ is invertible for all $(\xi,t,k)\in  I_{\mathcal M} \times (0,\infty) \times \overbar{D_\varepsilon(k_0)}$.
	\end{enumerate}
\end{lemma}
\begin{proof}
\eqref{lem_E_i}
	We note that additionally to $\mathfrak r(\xi,k)=r(k) + \bigO(|k-k_0|^N)$ as $k\to k_0$ uniformly in $\xi\in I_{\mathcal M}$, we have that $|\mathfrak r(\xi,k)|=|r(k)| + \bigO(|k-k_0|^N)$  as $\R\ni k\to k_0$ uniformly in $\xi\in I_{\mathcal M}$ since $|r|\equiv 1$ on $[E_1,E_2]$. Consequently,
	\begin{equation*}
		\log |\mathfrak r(\xi,k)| = \bigO(|k-k_0|^N), \quad \R\ni k\to k_0 
	\end{equation*}
	uniformly in $\xi\in I_{\mathcal M}$.
	In view of Lemma \ref{lem_sing_int_near_boundary} and by noting that 
	\begin{equation*}
		\frac{1}{2}\log(\mathfrak r(\xi,s)/\overbar{\mathfrak r(\xi,s)})
		=  i \arg \mathfrak r(\xi,k)
		= \log  \mathfrak r(\xi,k),
	\end{equation*}
	this means that there exist unique coefficients $\{c_n(\xi)\colon  I_{\mathcal M} \to \R\}^{N-1}_{n=0}$ such that
	\begin{align*}
		\int^{k_0}_{E_1} \frac{\frac{1}{2}\log(\mathfrak r(\xi,s)/\overbar{\mathfrak r(\xi,s)})}{\mathcal{X}_+(\xi,s)(s-k)} \, \dx s 
		+   \int^{E_2}_{k_0} \frac{\log|\mathfrak r(\xi,s)|}{s-k} \, \dx s 
		- \sum^{N-1}_{n=0} c_n(\xi) (k-k_0)^n \\
		=  \int^{E_1}_{-\infty} \frac{\log(1-|r(s)|^2)}{\mathcal{X}(\xi,s)(s-k)} \, \dx s + 
		\int^{k_0}_{E_1} \frac{\log(r(s))}{\mathcal{X}_+(\xi,s)(s-k)} \, \dx s +  \bigO(|k-k_0|^N),
	\end{align*}
	which implies \eqref{frakD_asympt}.
	
	\eqref{lem_E_ii}
	Let $(\xi,t)\in  I_{\mathcal M} \times (0,\infty)$ be arbitrary.	
	It is clear from its definition that $E(\xi,t,\cdot)$ is analytic in $D_\varepsilon(k_0)\setminus (k_0-\varepsilon,k_0+\varepsilon)$, and that the one-sided limits $E_\pm(\xi,t,\cdot)$ are well-defined along $(k_0-\varepsilon,k_0+\varepsilon)\setminus \{0\}$.	
	Furthermore, $E(\xi,t,\cdot)$ has no jump across $(k_0, k_0+\varepsilon)$ because $\frac{\mathfrak{D}_-}{\mathfrak{D}_+}|\mathfrak r|=1$ for all $(\xi,k)\in I_{\mathcal M}\times (k_0,k_0+\varepsilon)$, so that
	\begin{align*}
		E_+(k) &=  m_+^{[E_1,k_0]}(k) \, G^{-1}_{1,+}(k) \, N^{-1} (z(k))_-^{\frac{\sigma_3}{4}} \\
		&=  m^{[E_1,k_0]}_{-}(k) \, G^{-1}_{1,+}(k) \, N^{-1} \,  i^{-\sigma_3} \, (z(k))_+^{\frac{\sigma_3}{4}} \\
		&=   m_-^{[E_1,k_0]}(k) \, G^{-1}_{2,-}(k) \, N^{-1} \, (z(k))_+^{\frac{\sigma_3}{4}} = E_-(k)
	\end{align*}
	for all $(\xi,t,k)\in I_{\mathcal M}\times (0,\infty) \times (k_0,k_0+\varepsilon)$.
	Similarly, $E(\xi,t,\cdot)$ has no jump across $(k_0-\varepsilon,k_0)$ because $\mathfrak{D}_+ \mathfrak{D}_- \frac{\sqrt{\mathfrak r^*}}{\sqrt{\mathfrak r}}=1$ for all $(\xi,k)\in I_{\mathcal M}\times (k_0 - \varepsilon,k_0)$, so that
	\begin{align*}
		E_+(k) &=  m_+^{[E_1,k_0]}(k) \, G^{-1}_{1,+}(k) \, N^{-1} \, z(k)^{\frac{\sigma_3}{4}} \\
		&=  m^{[E_1,k_0]}_{-}(k)\begin{pmatrix} 0 & 1 \\ -1 & 0 \end{pmatrix}
		G^{-1}_{1,+}(k) \, N^{-1} z(k)^{\frac{\sigma_3}{4}} \\
		&=   m^{[E_1,k_0]}_{-}(k) \, G^{-1}_{2,-}(k) \, N^{-1} \, z(k)^{\frac{\sigma_3}{4}}  = E_-(k)
	\end{align*}
	for all $(\xi,t,k)\in I_{\mathcal M}\times (0,\infty) \times (k_0 - \varepsilon,k_0)$.
	
	Next, we determine the asymptotic behavior of $E$ at $k_0$. 
	From assertion  \eqref{lem_E_i}, Lemma \ref{lem_D}\eqref{lem_D_k0}  and the fact that
	\begin{equation*}
		\begin{cases}
			\sqrt{r(k)} = \sqrt{r(k_0)} + \frac{r'(k_0)}{2\sqrt{r(k_0)}} (k-k_0) + \bigO(|k-k_0|^{3/2})\\
			\frac{1}{\sqrt{r(k)}} =\frac{1}{\sqrt{r(k_0)}} - \frac{r'(k_0)}{2 r(k_0)^{3/2}} (k-k_0) + \bigO(|k-k_0|^{3/2})
		\end{cases}
		\quad \text{as} \quad k\to k_0
	\end{equation*}
	uniformly for $\xi\in I_{\mathcal M}$,
	it follows that
	\begin{equation*}
		\begin{cases}
			\mathfrak r(\xi,k)^\frac{1}{2} \mathfrak{D}^{-1}(\xi,k) = 1 + C_{k_0} \sqrt{k - k_0}  + \frac{C_{k_0}^2}{2} (k - k_0) + \bigO(|k-k_0|^{3/2}), \\
			\mathfrak r(\xi,k)^{-\frac{1}{2}} \mathfrak{D}(\xi,k)    = 1 - C_{k_0} \sqrt{k - k_0}  + \frac{C_{k_0}^2}{2} (k - k_0) + \bigO(|k-k_0|^{3/2}). 
		\end{cases}
		\quad \text{as} \quad k\to k_0,
	\end{equation*}
	uniformly for $\xi\in I_{\mathcal M}$, where $C_{k_0}$ is given by \eqref{C_k0}.
	Using the expansions of $m^{[E_1,k_0]}$ and $f$ at $k_0$ given in \eqref{minftyatk0} and \eqref{fk0neark0}, respectively,
	we conclude that, uniformly for $\xi\in I_{\mathcal M}$,
	\begin{equation} \label{Ek0_order0}
		E(\xi,t,k) = 
		\begin{pmatrix}
			\frac{- i-1}{2} (k_0-E_1)^{\frac{1}{3}} (3t)^{\frac{1}{6}} & 
			\frac{\frac{ i+1}{2} (-1+C_{k_0} \sqrt{k_0-E_1})}{(3t)^{1/6} (k_0-E_1)^{1/3}} \\
			\frac{ i-1}{2} (k_0-E_1)^{\frac{1}{3}} (3t)^{\frac{1}{6}}  &
			\frac{\frac{1- i}{2} (1+C_{k_0} \sqrt{k_0-E_1})}{(3t)^{1/6} (k_0-E_1)^{1/3}}
		\end{pmatrix}
		+ \bigO(|k-k_0|) \quad \text{as} \quad k \to k_0.
	\end{equation}
	Hence $E(\xi,t,\cdot)$ does not develop a singularity at $k_0$ and we conclude that $E(\xi,t,\cdot)$ is analytic in $D_\varepsilon(k_0)$. 
	
	Finally, $E(\xi,t,\cdot)$ is by its definition invertible on $\overbar{D_\varepsilon(k_0)}\setminus \{k_0\}$, and  since $E(\xi,t,k_0)$ is invertible by \eqref{Ek0_order0} the second statement of the lemma follows. 
\end{proof}

\begin{lemma} \label{lem_m^k_0}
	For every $(\xi,t) \in  I_{\mathcal M} \times (0,\infty)$, the function $m^{k_0}(\xi,t,\cdot)\colon \overbar{D_\varepsilon(k_0)}\setminus \overbar{\Gamma^{k_0,\eps}}\to\C^{2\times2}$ defined by \eqref{def_m^{k_0}} is analytic in $D_\varepsilon(k_0)\setminus\Gamma^{k_0,\varepsilon}$, $m^{k_0}(\xi,t,k)$ is invertible for all $(\xi,t,k)\in  I_{\mathcal M} \times (0,\infty) \times \overbar{D_\varepsilon(k_0)}\setminus \overbar{\Gamma^{k_0,\eps}}$, and the following hold true.
	\begin{enumerate}[\upshape (i)]
		\item  \label{lem_m^k_0_i}
		There exists a uniform bound $K>0$ such that
		\begin{equation} \label{lem_m^k_0_i_1}
				|(m^{k_0}(\xi,t,k))^{\pm 1}|  \leq K		\quad \text{for all} \quad (\xi,t,k)\in  I_{\mathcal M} \times (0,\infty) \times \overbar{D_\varepsilon(k_0)}\setminus \overbar{\Gamma^{k_0,\eps}}.
		\end{equation}
		\item \label{lem_m^k_0_ii}
		As $t\to\infty$, uniformly for $\xi\in I_{\mathcal M}$,
		\begin{equation} \label{lem_m^k_0_ii_1}
			\|m^{k_0}(\xi,t,k) \, \big(m^{[E_1,k_0]}(\xi,k)\big)^{-1} -I\|_{L^\infty(\partial D_\varepsilon(k_0))} = \bigO(t^{-1}) 
		\end{equation}
		and
		\begin{align}
			\begin{aligned} \label{lem_m^k_0_ii_2}
				&-\frac{1}{2\pi  i} \int_{\partial D_{\varepsilon}(k_0)}   \Big(m^{k_0}(\xi,t,k) \big(m^{[E_1,k_0]}(\xi,k)\big)^{-1} - I\Big)   \, \dx k \\
				&\quad =   \frac{t^{-1}}{288}
				\begin{pmatrix}
					12 i C_{k_0}^2 - \frac{7 i}{k_0-E_1}  &   
					12 C_{k_0}^2  - \frac{24}{\sqrt{k_0-E_1}} C_{k_0} +\frac{7}{k_0-E_1}   \\
					12 C_{k_0}^2  + \frac{24}{\sqrt{k_0-E_1}} C_{k_0} +\frac{7}{k_0-E_1}  & - 12 i C_{k_0}^2 + \frac{7 i}{k_0-E_1}
				\end{pmatrix} 
				+ \bigO(t^{-2}),
			\end{aligned}
		\end{align}
		where the circle $\partial D_\varepsilon(k_0)$ is oriented clockwise and $C_{k_0}$ is given by \eqref{C_k0}.
		\item \label{lem_m^k_0_iii}
		Across $\Gamma^{k_0,\varepsilon}$, $m^{k_0}$ obeys the jump condition $m^{k_0}_+ = m^{k_0}_- v^{k_0}$, where the jump matrix $v^{k_0}$ satisfies
		\begin{align}
			&\| v^{(3)} - v^{k_0}  \|_{L^1(\Gamma^{k_0,\varepsilon})} = \bigO(t^{-2}), \label{lem_m^k_0_iii_1}\\
			&\| v^{(3)} - v^{k_0}  \|_{L^2(\Gamma^{k_0,\varepsilon})} = \bigO(t^{-5/3}),  \label{lem_m^k_0_iii_2} \\
			&\| v^{(3)} - v^{k_0}  \|_{L^\infty(\Gamma^{k_0,\varepsilon})} = \bigO(t^{-4/3}),   \label{lem_m^k_0_iii_3}
		\end{align}
		uniformly for $\xi\in I_{\mathcal M}$.
	\end{enumerate}
\end{lemma}
\begin{proof}
	With the help of Lemmas \ref{lem_Airy} and \ref{lem_E} it follows immediately that $m^{k_0}(\xi,t,\cdot)$ is analytic in $D_\varepsilon(k_0)\setminus \Gamma^{k_0,\varepsilon}$ for all $(\xi,t)\in I_{\mathcal M} \times (0,\infty)$, $m^{k_0}(\xi,t,k)$ is invertible for all $(\xi,t,k)\in  I_{\mathcal M} \times (0,\infty) \times \overbar{D_\varepsilon(k_0)}\setminus \overbar{\Gamma^{k_0,\eps}}$, and that \eqref{lem_m^k_0_i} is satisfied. 
	
	\eqref{lem_m^k_0_ii}
	By the definition of $z$, see \eqref{def_z}, and since our construction is uniform with respect to $\xi\in I_{\mathcal M}$ (particularly the radius $\varepsilon$ of the disc $D_{\varepsilon}(k_0)$ is uniformly chosen for all $\xi\in I_{\mathcal M}$) we find a uniform constant $c>0$ such that $|z(\xi,t,k)| \geq c t^{2/3}$ for all $(\xi,t,k)\in I_{\mathcal M}\times (0,\infty) \times \partial D_{\varepsilon}(k_0)$.
	Consequently, \eqref{large z asymptotics Airy} implies that uniformly for $(\xi,k)\in  I_{\mathcal M} \times \partial D_{\varepsilon}(k_0)$,
	\begin{align}
		\begin{aligned} \label{pr_lem_m^k_0_ii_1}
			m^{k_0}(\xi,t,k) \,  &\big(m^{[E_1,k_0]}(\xi,k)\big)^{-1} \\
			&=m^{[E_1,k_0]}(\xi,k) \,
			G^{-1}(\xi,k) \,
			\bigg(I + \sum_{j=1}^\infty \frac{m^{\mathrm{Ai}}_j}{(t^{2/3} f(\xi,k))^{3j/2}}\bigg)  \, 
			G(\xi,k) \, \big(m^{[E_1,k_0]}(\xi,k)\big)^{-1} 
		\end{aligned}
	\end{align}
	as $t\to\infty$, where the complex coefficients $m^{\mathrm{Ai}}_j$ are given by \eqref{def_m^Ai_j}.
	We define the family of functions
	\begin{align*}
		A_j \colon  I_{\mathcal M} \times  \overbar{D_{\varepsilon}(k_0)} \to\C,
		\quad
		A_j \coloneqq 	f^{-3j/2} \,	m^{[E_1,k_0]} \,	G^{-1} \, m^{\mathrm{Ai}}_j  \,	G \, \big(m^{[E_1,k_0]}\big)^{-1}, \quad j\in \N_{\geq 1},
	\end{align*}
	where each $A_j(\xi,\cdot)\colon \overbar{D_{\varepsilon}(k_0)} \to\C$ is a rational function with precisely one pole at $k_0$; it holds that
	\begin{equation*}
		\sup\{|A_j(\xi,k)|\colon (\xi,k,j)\in  I_{\mathcal M} \times \partial D_{\varepsilon}(k_0) \times \N_{\geq 1} \} < \infty.
	\end{equation*}
	Thus $\sum_{j=1}^\infty t^{-j} A_j(\xi,k)$ converges absolutely for all large enough $t$, say $t>2$, and uniformly for all $(\xi,k)\in  I_{\mathcal M} \times \partial D_{\varepsilon}(k_0)$.
	From \eqref{pr_lem_m^k_0_ii_1} we infer that  
	\begin{align*}
		m^{k_0}(\xi,t,k) \, \big(m^{[E_1,k_0]}(\xi,k)\big)^{-1} -I =   \sum_{j=1}^\infty \frac{A_j(\xi,k)}{t^j} \quad \text{as} \quad t\to\infty,
	\end{align*}
	uniformly for all $(\xi,k)\in  I_{\mathcal M} \times \partial D_{\varepsilon}(k_0)$, which proves \eqref{lem_m^k_0_ii_1}, and furthermore we infer that 
	\begin{align} \label{res_A_1_lim}
		-\frac{1}{2\pi  i}& \int_{\partial D_{\varepsilon}(k_0)}   \Big(m^{k_0}(\xi,t,k) \big(m^{[E_1,k_0]}(\xi,t,k)\big)^{-1} - I\Big)   \, \dx k 
		= t^{-1} \res_{k_0} A_1(\xi,k) + \bigO(t^{-2}) 
	\end{align}
	as $t\to\infty$, uniformly for $\xi\in I_{\mathcal M}$.

	It remains to determine the residue $\res_{k_0} A_1(\xi,k)$.
	Let 
	\begin{equation*}
		\tilde E(\xi,k) \coloneqq  E(1,\xi,k) = m^{[E_1,k_0]}(\xi,k) \,\tilde G^{-1}(\xi,k) \, N^{-1} \, f(\xi,k)^{\frac{\sigma_3}{4}}.
	\end{equation*}
	We already know that $\tilde E$ is analytic near $k_0$. It is also invertible, thus we can write 
	\begin{equation} \label{A_1}
		A_1 = \tilde E f^{-\frac{\sigma_3}{4}} N \frac{m_1^{\mathrm{Ai}}}{f^{3/2}} N^{-1}  f^{\frac{\sigma_3}{4}} \tilde E^{-1}.
	\end{equation}  
	In the following, we provide the expansions of the involved roots for $k\to k_0$ approaching from the upper half-plane; the corresponding calculations for the case that $k\to k_0$ approaches from the lower half-plane are similar and lead to the same result.  
	It follows from \eqref{fk0neark0} that 
	\begin{equation*}
		\begin{cases}
			f(\xi,k)^{1/4} =  \e^{-\frac{ i\pi}{4}} 3^{\frac{1}{6}} (k_0-E_1)^{\frac{1}{12}} (k-k_0)^{1/4}
			\Big(1 + \frac{k-k_0}{12(k_0-E_1)} + \bigO(|k-k_0|^2)\Big) \\
			f(\xi,k)^{3/2} =  \e^{\frac{ i\pi}{2}} 3 (k_0-E_1)^{\frac{1}{2}} (k-k_0)^{3/2}
			\Big(1 + \frac{k-k_0}{2(k_0-E_1)} + \bigO(|k-k_0|^2)\Big)
		\end{cases}
		 \text{as $\C_+\ni k \to k_0$,}
	\end{equation*}
	and thus the middle factor $f^{-\frac{\sigma_3}{4}} N \frac{m_1^{\mathrm{Ai}}}{f^{3/2}} N^{-1}  f^{\frac{\sigma_3}{4}}$ in \eqref{A_1} satisfies
	\begin{equation} \label{Xf}
		\begin{pmatrix}
			0 & \frac{5  i}{144 \cdot 3^{1/3} (k_0-E_1)^{2/3} (k-k_0)^2} - \frac{5 i}{216 \cdot 3^{1/3} (k_0-E_1)^{5/3} (k-k_0)} \\
			-\frac{7  i}{48  \cdot 3^{2/3} (k_0-E_1)^{1/3} (k-k_0)} & 0
		\end{pmatrix}
		+ \bigO(1)
	\end{equation} 
	as $\C_+\ni k \to k_0$.
	To calculate $\res_{k_0} A_1(k)$ it remains to determine $\tilde E(k_0)$ and the $k$-derivative $\tilde E'(k_0)$. 
	The matrix $\tilde E(k_0)$ equals the leading order term in \eqref{Ek0_order0} with $t=1$. We refrain from providing the bulky expression for $\tilde E'(k_0)$, but note that it is not necessary to fully determine all the entries of $\tilde E'(k_0)$ in order to calculate $\res_{k_0}A_1(\xi,k)$; in particular the coefficient of $(k-k_0)^{3/2}$ in the series expansion of $\mathcal{D}$ is not relevant for the sake of determining $\res_{k_0}A_1(\xi,k)$. By plugging \eqref{Xf} and the matrices $\tilde E(k_0)$ and $\tilde E'(k_0)$ into \eqref{A_1} we obtain that
	\begin{equation} \label{res_A_1}
		\underset{k=k_0}{\res}\, A_1(\xi,k) = \frac{1}{288}
		\begin{pmatrix}
			12 i C_{k_0}^2 - \frac{7 i}{k_0-E_1}  &   
			12 C_{k_0}^2  - \frac{24}{\sqrt{k_0-E_1}} C_{k_0} +\frac{7}{k_0-E_1}   \\
			12 C_{k_0}^2  + \frac{24}{\sqrt{k_0-E_1}} C_{k_0} +\frac{7}{k_0-E_1}  & - 12 i C_{k_0}^2 + \frac{7 i}{k_0-E_1}
		\end{pmatrix},
	\end{equation} 
	which verifies \eqref{lem_m^k_0_ii_2} and finishes the proof of assertion \eqref{lem_m^k_0_ii}.

	\eqref{lem_m^k_0_iii}	On $\Gamma^{k_0,\varepsilon}$, the jump matrix $v^{(3)}-v^{k_0}$ satisfies
	\begin{flalign*}
		\begin{aligned}
			&v^{(3)}_1-v^{k_0}_1 = 
			\begin{pmatrix}
				0 & 0 \\
				-\big[\mathcal{D}^{-2}r_{a} - \mathfrak{D}^{-2} \mathfrak r\big]  \e^{2 i t g} & 0
			\end{pmatrix}, 
			&&v^{(3)}_2 - v^{k_0}_2 =
			\begin{pmatrix}
				0&0\\0& \Big[ \frac{\mathcal{D}_+}{\mathcal{D}_-} - \frac{\mathfrak{D}_+}{\mathfrak{D}_-}\frac{1}{|\mathfrak r|} \Big] \e^{-2 i tg_+}
			\end{pmatrix},
			\\
			&v^{(3)}_3 - v^{k_0}_3 =
			\begin{pmatrix}
				0 & \big[\mathcal{D}^2 r^*_{a}  - \mathfrak{D}^2 \mathfrak r^* \big]\e^{-2 i t g} \\
				0 & 0
			\end{pmatrix},  
			&&v^{(3)}_4 - v^{k_0}_4 =
			\begin{pmatrix}
				-|r_{r}|^2 & \mathcal{D}^2 r^*_{r} \e^{-2 i t g} \\
				-\mathcal{D}^{-2} r_{r} \e^{2 i t g} & 0
			\end{pmatrix}.
		\end{aligned}
	\end{flalign*}
	We recall that $\mathcal{D}$ is uniformly bounded on $ I_{\mathcal M} \times D_\varepsilon(k_0)$ by Lemma \ref{lem_D}. Furthermore, $r_{a}(\xi,t,\cdot)$, $r_{r}(\xi,t,\cdot)$ and $\mathfrak r(\xi,\cdot)$  are bounded on their respective domains intersected with $D_\varepsilon(k_0)$ for all $(\xi,t)\in  I_{\mathcal M} \times (0,\infty)$, hence $v^{(3)} - v^{k_0}$ belongs to $L^p(\Gamma^{k_0,\varepsilon})$, $p=1,2,\infty$, for all $(\xi,t)\in  I_{\mathcal M} \times (0,\infty)$.
	Since $|\e^{\pm 2  i t g}|=1$ for all $(\xi,t,k)\in  I_{\mathcal M} \times (0,\infty) \times \Gamma^{k_0,\varepsilon}_4$, it follows directly from Lemma \ref{lem_an_appr_r1}\eqref{lem_an_appr_r1_iii} that 
	\begin{equation*} 
		\|v^{(3)} - v^{k_0}\|_{L^p(\Gamma^{k_0,\varepsilon}_4)} = \bigO(t^{-2}) \quad \text{as} \quad t\to\infty, \quad p=1,2,\infty,
	\end{equation*}
	uniformly for $\xi\in I_{\mathcal M}$.
	
	Next we show that $\|v^{(3)} - v^{k_0}\|_{L^p(\Gamma^{k_0,\varepsilon}_1)}$ satisfies the asserted asymptotics; 
	symmetry implies same behavior for $\|v^{(3)} - v^{k_0}\|_{L^p(\Gamma^{k_0,\varepsilon}_3)}$.
	Uniformly for $\xi\in I_{\mathcal M}$, it holds that
	\begin{equation*}
		\begin{cases}
			\mathfrak r(\xi,k) =   r(k_0) - r'(k_0) (k-k_0) + \bigO(|k-k_0|^2) \\
			\mathcal{D}(\xi,k) - \mathfrak{D}(\xi,k) = \bigO(|k-k_0|^2)
		\end{cases}
		\quad \text{as} \quad \Gamma^{k_0,\varepsilon}_1 \ni k \to k_0,
	\end{equation*}
	cf.~\eqref{frak_r} and \eqref{frakD_asympt}. 
	Therefore we deduce from Lemma \ref{lem_D}\eqref{lem_D_k0} and Lemma \ref{lem_an_appr_r1}\eqref{lem_an_appr_r1_ii}  that
	\begin{align*}
		\big|\mathcal{D}^{-2} r_{a}(\xi,t,k)  - \mathfrak{D}^{-2} \mathfrak r(\xi,k)\big| 
		\leq  C |k-k_0|^2 \, \big( \e^{\frac{t}{4}| \im g(\xi,k)|} + 1\big)
	\end{align*}
	uniformly for all $(\xi,t,k)\in I_{\mathcal M} \times (0,\infty) \times \Gamma^{k_0,\eps}_1$.
	Consequently,   
	\begin{align*}
		|v^{(3)}_1(\xi,t,k) -v^{k_0}_1(\xi,t,k)| &\leq C |k-k_0|^2 \, \e^{-t| \im g(\xi,k)|} \leq  C |k-k_0|^2 \, \e^{- K t |k-k_0|^{3/2}},
	\end{align*}
	uniformly for all $(\xi,t,k)\in I_{\mathcal M} \times (0,\infty) \times \Gamma^{k_0,\eps}_1$.
	The constant $K>0$ can be chosen uniformly for all $\xi\in I_{\mathcal M}$, e.g.,
	\begin{equation*} 
		K\coloneqq \inf_{\xi\in I_{\mathcal M}}\sqrt{k_0-E_1}>0,
	\end{equation*} 
	by recalling that 
	$g(\xi,k) = 2\sqrt{k_0-E_1}(k-k_0)^{3/2} + \bigO\big(|k-k_0|^{5/2}\big)$ as $k\to k_0$.
	We infer that uniformly for all $\xi\in  I_{\mathcal M}$,
	\begin{align*}
		\big\|v^{(3)}_1 -v^{k_0}_1\big\|_{L^1(\Gamma^{k_0,\eps}_1)} &\leq C \int^\eps_0 u^2 \e^{-K t u^{3/2}} \, \dx u = \bigO(t^{-2}), 
		\\
		\big\|v^{(3)}_1 -v^{k_0}_1\big\|^2_{L^2(\Gamma^{k_0,\eps}_1)} &\leq C \int^\eps_0 \Big( u^2 \e^{-K t u^{3/2}} \Big)^2 \, \dx u = \bigO(t^{-10/3}).
	\end{align*}
	Furthermore, substituting $u=v t^{-2/3}$ yields that
	\begin{align*}
		\sup_{0\leq u \leq \eps}u^2 \e^{-K t u^{3/2}} \leq  \sup_{u\geq 0}u^2 \e^{-K t u^{3/2}}  = t^{-\frac{4}{3}} \sup_{v\geq 0}v^2 \e^{-K v^{3/2}}  = \bigO(t^{-\frac{4}{3}}).
	\end{align*}
	
	The estimation of $\|v^{(3)} - v^{k_0}\|_{L^p(\Gamma^{k_0,\varepsilon}_2)}$ works similarly by noting that uniformly for $\xi\in I_{\mathcal M}$,
	\begin{equation*}
		\frac{\mathcal{D}_+(\xi,k)}{\mathcal{D}_-(\xi,k)} - \frac{\mathfrak{D}_+(\xi,k)}{\mathfrak{D}_-(\xi,k)}\frac{1}{|\mathfrak r(\xi,k)|} = \bigO(|k-k_0|^2)
		\quad \text{as} \quad \Gamma^{k_0,\varepsilon}_2 \ni k \to k_0,
	\end{equation*}
	and $|\e^{-2 i tg_+(\xi,k)}| \leq C  \e^{- K t |k-k_0|^{3/2}}$ on $ I_{\mathcal M} \times (0,\infty) \times \Gamma^{k_0,\varepsilon}_2$ for uniform constants $C,K>0$.
	In summary, we have verified \eqref{lem_m^k_0_iii_1}--\eqref{lem_m^k_0_iii_3}, which finishes the proof of statement \eqref{lem_m^k_0_iii}.
\end{proof}

\subsection{Long-time asymptotics} \label{sec_solution_asymp}
Define $m^{\mathrm{app}}\colon  I_{\mathcal M} \times (0,\infty) \times\C\setminus\big(\Gamma^{k_0,\eps}\cup(E_1,k_0)\big) \to \C^{2\times2}$ by 
\begin{equation*} 
	m^{\mathrm{app}}(\xi,t,k) \coloneqq
	\begin{cases}
		m^{k_0}(\xi,t,k), & k\in D_{\varepsilon}(k_0)\setminus \Gamma^{k_0,\eps},\\
		m^{[E_1,k_0]}(\xi,t,k), & \text{elsewhere in}~\C\setminus (E_1,k_0).
	\end{cases}
\end{equation*} 
Let $\hat \Gamma \coloneqq \Gamma^{(3)} \cup \partial D_{\varepsilon}(k_0)$, where the circle $\partial D_{\varepsilon}(k_0)$ is oriented clockwise, see Figure~\ref{fig:Gamma_hat}. 
Let $\hat{\Gamma}_\star$ denote the contour $\hat{\Gamma}$ with the six points of self-intersection and the point $E_1$ removed. 
The function $\hat{m}$ defined by
\begin{equation} \label{m_hat}
	\hat m \colon  I_{\mathcal M} \times (0,\infty) \times \C \setminus \hat{\Gamma} \to \C^{2\times2}, \quad
	\hat m(\xi,t,k) \coloneqq m^{(3)}  (m^{\mathrm{app}})^{-1},
\end{equation} 
satisfies the following RH problem for $(\xi,t)\in I_{\mathcal M} \times (0,\infty)$:
\begin{itemize}
\item $\hat{m}(\xi,t,\cdot)$ is analytic in $\C \setminus \hat{\Gamma}$;

\item $\hat{m}(\xi,t,\cdot)$ has continuous boundary values on $\hat{\Gamma}_\star$ satisfying the jump relation
\begin{align}\label{RHP_m_hat}
\hat{m}_+(\xi,t,k) = \hat{m}_-(\xi,t,k) \hat{v}(\xi,t,k) \quad \text{for} \; k \in \hat{\Gamma}_\star,
\end{align}
where
\begin{equation*} 
	\hat v  \coloneqq
	\begin{cases}
		m^{[E_1,k_0]}_- v^{(3)} \big(m^{[E_1,k_0]}_+\big)^{-1}, & k\in \hat \Gamma \setminus \overbar{D_{\varepsilon}(k_0)}, \\
		m^{k_0} \big(m^{[E_1,k_0]}\big)^{-1}, & k \in \partial D_{\varepsilon}(k_0), \\
		m^{k_0}_- v^{(3)} \big( m^{k_0}_+ \big)^{-1}, & k\in \hat\Gamma\cap D_{\varepsilon}(k_0);
	\end{cases}
\end{equation*} 

\item $\hat{m}(\xi,t,k)=I+ \bigO(k^{-1})$ as $k\to \infty$;

\item $\hat{m}(\xi,t,k)= \bigO(1)$ as $k\to E_1$;

\item $\hat{m}(\xi,t,k) = \bigO(1)$ as $k$ approaches any of the six self-intersection points of $\hat{\Gamma}$.
\end{itemize}
All the above properties of $\hat{m}$ follow easily from the properties of $m^{(3)}$, $m^{[E_1, k_0]}$, and $m^{k_0}$, except for the boundedness of $\hat{m}$ as $k \to E_1$ which requires a calculation. Let us verify that $\hat{m}$ stays bounded as $k \in \C_+$ approaches $E_1$; the proof for $k \in \C_-$ is similar. From the behavior \eqref{m3atE1} of $m^{(3)}$ at $E_1$ and the definition of $\hat{m}$, we have
\begin{equation*}
	\hat{m}(k) = \mathcal{D}_\infty^{\sigma_3}  m^{(1)}(k) G^{(3)}(k)
	(k-E_1)^{-\sigma_3/4} \e^{-d_0\sigma_3}\big(I+\bigO(\sqrt{k-E_1})\big)(m^{[E_1,k_0]})^{-1}(k)
\end{equation*}
as $\C_+ \ni k\to E_1$. The function $m^{(1)}(k)$ has continuous boundary values at $E_1$. Moreover, as $k \in \C_+$ tends to $E_1$, we have
\begin{equation*}
	G^{(3)}(k) = \begin{pmatrix}
		1 &   - r_{2,a} \\
		0 & 1
	\end{pmatrix} \big(I + \bigO(\sqrt{k-E_1})\big),
\end{equation*}
where the singular behavior of $r_{2,a}=h+f_a$ stems from the function $h$ defined in \eqref{hdef}. In view of \eqref{expansionr2} and \eqref{leadingorderr2}, it follows that
\begin{align}\label{asymptoticsr2a}
r_{2,a}(k)= h(k)+\bigO(1)= \frac{\overbar{q_{1,0}}}{q_{1,0} \overbar{q_{1,1}} + \overbar{q_{1,0}} q_{1,1}}\frac{1}{(-i)\sqrt{k - E_1}}+\bigO(1), \qquad k\to E_1, ~ k \in \C_+,
\end{align}
where the principal branch is used for the root.
Using the fact that $|q_{1,0}|=1$ and the expansion
\begin{equation*}
(m^{[E_1,k_0]})^{-1}(k) = \frac{e^{\frac{\pi i}{4}} (k_0 - E_1)^{1/4}}{2(k-E_1)^{1/4}} \begin{pmatrix} 1 & i \\ -i & 1 \end{pmatrix}\big(I + \bigO(\sqrt{k-E_1})\big), \qquad k\to E_1, ~ k \in \C_+,	
\end{equation*}
a calculation gives
\begin{equation*}
	G^{(3)}(k) (k-E_1)^{-\sigma_3/4} \e^{-d_0\sigma_3}(m^{[E_1,k_0]})^{-1}(k)
	= \bigO(1), \qquad k\to E_1, ~ k \in \C_+.
\end{equation*}
It follows that $\hat{m} = \bigO(1)$ as $\C_+ \ni k \to E_1$.

\begin{figure}[h!] \centering
	\begin{overpic}[width=.43\textwidth]{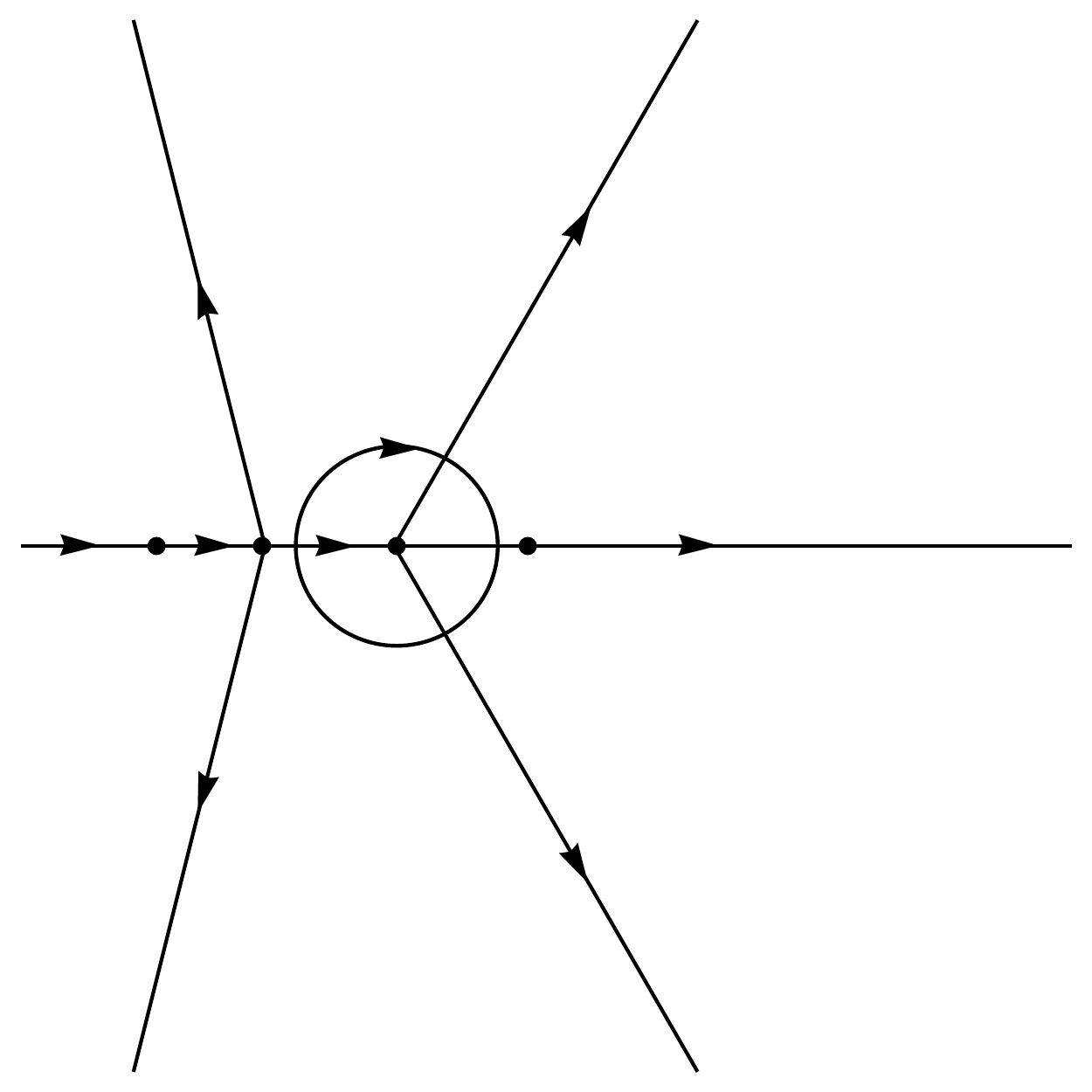}
		\put(11,45){\footnotesize{$E_1$}} 
		\put(32.5,45.5){\footnotesize{$k_0$}}   
		\put(17,53){\footnotesize{$\mu_-$}}    
		\put(48,45){\footnotesize{$E_2$}} 
	\end{overpic}
	\caption{The contour $\hat \Gamma = \Gamma^{(3)}\cup \partial D_{\varepsilon}(k_0)$.} 
	\label{fig:Gamma_hat}
\end{figure}
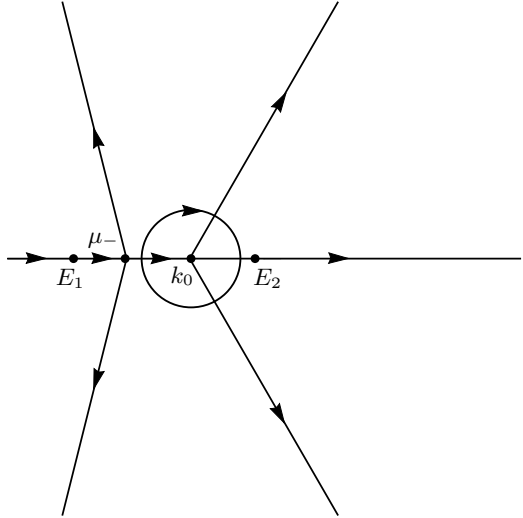

\subsubsection{Asymptotics of $\hat m$}
Our results from the previous sections suggest that 
\begin{equation*} 
	\hat w\colon  I_{\mathcal M} \times (0,\infty) \times \hat\Gamma \to \C^{2\times2} ,\quad 	\hat w \coloneqq \hat v - I
\end{equation*} 
is small for sufficiently large $t>0$; the following lemma makes this precise.

\begin{lemma}  \label{lem_w_hat}
	For $p\in\{1,2,\infty\}$ and $(\xi,t)\in I_{\mathcal M} \times (0,\infty)$, the function $\hat w(\xi,t,\cdot)\colon \hat \Gamma\to\C^{2\times2}$ lies in $L^p(\hat\Gamma)$, and	
	uniformly for $\xi\in I_{\mathcal M}$ it holds that
	\begin{subequations}
	\begin{align}
		\phantom{XX}
		&\| \hat w  \|_{L^p(\hat\Gamma\setminus (\overbar{D_\varepsilon(k_0)}\cup\R) )} = \bigO(\e^{- c t}),   \qquad c>0, \label{lem_w_hat_1} \\
		&\| \hat w  \|_{L^p(\R\setminus D_\varepsilon(k_0))} = \bigO(t^{-2}),    \label{lem_w_hat_2} \\ 
		&\| \hat w  \|_{L^p(\partial D_\varepsilon(k_0) )} = \bigO(t^{-1}),     \label{lem_w_hat_3} \\
		&\| \hat w  \|_{L^1(\Gamma^{k_0,\varepsilon})} = \bigO(t^{-2}),  \label{lem_w_hat_4} \\
		&\| \hat w  \|_{L^2(\Gamma^{k_0,\varepsilon})} = \bigO(t^{-5/3}),  \label{lem_w_hat_5} \\
		&\| \hat w  \|_{L^\infty(\Gamma^{k_0,\varepsilon})} = \bigO(t^{-4/3}). \label{lem_w_hat_6}
	\end{align}
	\end{subequations}
	In particular, uniformly for $\xi\in I_{\mathcal M}$ it holds that
	\begin{equation} \label{lem_w_hat_7}
		\| \hat w  \|_{L^p(\hat\Gamma)} = \bigO(t^{-1}).
	\end{equation}
\end{lemma}
\begin{proof}
	We have  
	\begin{equation*}
		\hat w = \hat v -I = 
		\begin{cases}
			m^{[E_1,k_0]} ( v^{(3)} - I)\big(m^{[E_1,k_0]}\big)^{-1} & k\in \hat \Gamma \setminus (\overbar{D_{\varepsilon}(k_0)} \cup [E_1,k_0]), \\
			m^{[E_1,k_0]}_- (v^{(3)} - v^{[E_1,k_0]}) \big(m^{[E_1,k_0]}_+\big)^{-1} & k\in  [E_1,k_0] \setminus \overbar{D_{\varepsilon}(k_0)},   \\
			m^{k_0} \big(m^{[E_1,k_0]}\big)^{-1} - I& k \in \partial D_{\varepsilon}(k_0), \\
			m^{k_0}_- (v^{(3)} - v^{k_0}) \big( m^{k_0}_+ \big)^{-1} & k\in \hat\Gamma\cap D_{\varepsilon}(k_0).
		\end{cases}
	\end{equation*}
	With the help of Lemmas \ref{lem_D}, \ref{lem_an_appr_r1}, \ref{lem_an_appr_r2}, \ref{lem_minfty}, \ref{lem_m^k_0}, and by recalling that $1-|r|^2 = \bigO(|k-E_1|^{1/2})$ as $\R\ni k\nearrow E_1$, we deduce that  
	$\hat w (\xi,t,\cdot)$ lies in $L^p(\hat \Gamma)$, $p=1,2,\infty$, with uniform bounds for all $(\xi,t)\in  I_{\mathcal M} \times (0,\infty)$. 
	In particular $\hat w(\xi,t,k)= \bigO(1)$ as $\R\ni k\to E_1$ uniformly for $(\xi,t)\in  I_{\mathcal M} \times (0,\infty)$; for the less obvious limit $\R\ni k\searrow E_1$ we 
	use that $\Delta_{0+}=  i \Delta_{0-}$ on $(E_1,k_0)$ and compute
	\begin{align*}
		\hat w
		&=\frac{(rr_{2,a}+r^*r_{2,a}^*+1) \e^{-2 i tg_+}}4
		\begin{pmatrix}
			- i \mathcal{D}_+\mathcal{D}_-^{-1}(\Delta_{0-}^2-\Delta_{0-}^{-2})&-\mathcal{D}_+\mathcal{D}_-^{-1}(\Delta_{0-}^2+\Delta_{0-}^{-2}-2) \\
			\mathcal{D}_+\mathcal{D}_-^{-1}(\Delta_{0-}^2+\Delta_{0-}^{-2}+2)& i \mathcal{D}_+\mathcal{D}_-^{-1}(\Delta_{0-}^2-\Delta_{0-}^{-2})
		\end{pmatrix}.
	\end{align*}
	Hence the behavior of $\mathcal{D}$, cf.~Lemma \ref{lem_D}\eqref{lem_D_iii}, and $\Delta_0$ implies boundedness as $k\searrow E_1$ uniformly with respect to $(\xi,t)\in  I_{\mathcal M} \times (0,\infty)$. Therefore, Lemma \ref{lem_an_appr_r2}\eqref{lem_an_appr_r2_iii}\&\eqref{lem_an_appr_r2_iv} imply that both $\| \hat w  \|_{L^p(-\infty,E_1)}$ and $\| \hat w  \|_{L^p(E_1,\mu_-)}$ are $\bigO(t^{-2})$ uniformly for $\xi\in I_{\mathcal M}$, $p=1,2,\infty$. 
	From Lemma \ref{lem_an_appr_r1}\eqref{lem_an_appr_r1_iii} we infer that $\| \hat w  \|_{L^p(k_0,\infty)} = \bigO(t^{-2})$ uniformly for $\xi\in I_{\mathcal M}$, $p=1,2,\infty$, while $\| \hat w  \|_{L^p(\mu_-,k_0-\varepsilon)}$ has exponential decay as $t\to\infty$ uniformly with respect to $\xi\in I_{\mathcal M}$.
	Furthermore, due to the (strict) signature of $g$ outside $D_\varepsilon(k_0)\cup\R$ we infer from  Lemma \ref{lem_an_appr_r1}\eqref{lem_an_appr_r1_ii} and Lemma \ref{lem_an_appr_r2}\eqref{lem_an_appr_r2_ii} that $\| \hat w  \|_{L^p(\hat\Gamma\setminus (D_\varepsilon(k_0)\cup\R) )}$ has exponential decay as $t\to\infty$, uniformly with respect to $\xi\in I_{\mathcal M}$. 
	Moreover, \eqref{lem_m^k_0_ii_1} yields \eqref{lem_w_hat_3}, and \eqref{lem_m^k_0_iii_1}--\eqref{lem_m^k_0_iii_3} in conjunction with \eqref{lem_m^k_0_i_1} yield \eqref{lem_w_hat_4}--\eqref{lem_w_hat_6}.
\end{proof}

Let $\mathcal C\colon L^2(\hat\Gamma) \to L^2(\C\setminus\hat\Gamma)$ denote the Cauchy transform given by
\begin{equation*} 
	(\mathcal C f)(k) = \frac{1}{2\pi  i}\int_{\hat\Gamma} \frac{f(z)}{z-k}\, \dx z.
\end{equation*} 
Furthermore, we consider the two operators $\mathcal C_\pm \colon L^2(\hat\Gamma) \to L^2(\hat\Gamma)$ being induced from $\mathcal C$ by taking the one-sided nontangential limits towards $\hat\Gamma$. These mappings lie in the space of bounded linear operators on $L^2(\hat\Gamma)$, denoted by $\mathcal B (L^2(\hat\Gamma))$, and satisfy the operator identity $\mathcal C_+ - \mathcal C_- = I$.
For every fixed $(\xi,t)\in I_{\mathcal M} \times (0,\infty)$ and $f\in L^2(\hat\Gamma) \cup L^\infty(\hat\Gamma)$ let
\begin{equation*} 
	\mathcal C_{\hat w} f \coloneqq \mathcal C_-(\hat w  f) \in  L^2(\hat\Gamma).
\end{equation*} 
In view of Lemma \ref{lem_w_hat}, $\mathcal C_{\hat w}\in \mathcal B (L^2(\hat\Gamma))$ for all $(\xi,t)\in I_{\mathcal M} \times (0,\infty)$, and, by \eqref{lem_w_hat_7},
\begin{equation} \label{C_w_asymptotics}
	\|\mathcal C_{\hat w(\xi,t,\cdot)}\|_{\mathcal B (L^2(\hat\Gamma))} = \bigO(t^{-1})
\end{equation}
uniformly for all $\xi\in I_{\mathcal M}$.
In particular there exists a time $T>0$ such that 
\begin{equation} \label{est_C_hat_w}
	\|\mathcal C_{\hat w(\xi,t,\cdot)}\|_{\mathcal B (L^2(\hat\Gamma))} \leq \frac{1}{2} \quad 
	\text{for all} \quad t\geq T 
\end{equation}
uniformly with respect to $\xi\in I_{\mathcal M}$, which yields that $I - C_{\hat w}$ is invertible:
\begin{equation*} 
	(I - \mathcal C_{\hat w(\xi,t,\cdot)})^{-1} \in \mathcal B (L^2(\hat\Gamma)) \quad 
	\text{for all} \quad t\geq T, \; \xi \in  I_{\mathcal M}.
\end{equation*} 
Consequently, for all $t\geq T$ and $\xi\in I_{\mathcal M}$, 
\begin{equation} \label{mu_sing_int_eq}
	\hat\mu (\xi,t,\cdot) \coloneqq I + (I - \mathcal C_{\hat w(\xi,t,\cdot)})^{-1} \, \mathcal C_{\hat w(\xi,t,\cdot)} I \; \in \; I + L^2(\hat\Gamma).
\end{equation} 
In particular, $\hat \mu - I \in L^2(\hat\Gamma)$  for all $(\xi,t)\in  I_{\mathcal M}\times [T,\infty)$, and a Neumann series estimate together with \eqref{lem_w_hat_7} and \eqref{C_w_asymptotics} implies that, uniformly for $\xi\in I_{\mathcal M}$,
\begin{equation} \label{mu-id}
	\| \hat \mu(\xi,t,\cdot) - I  \|_{L^2(\hat\Gamma)} 
	\leq \frac{C \|  \hat w(\xi,t,\cdot) \|_{L^2(\hat\Gamma)}}{1- \|  \mathcal C_{\hat w(\xi,t,\cdot)} \|_{\mathcal B (L^2(\hat\Gamma))}}
	= \bigO(t^{-1}) \quad \text{as} \quad t \to\infty.
\end{equation}

\begin{proposition} \label{prop_mhat}
	Let $T>0$ be such that \eqref{est_C_hat_w} is satisfied for all $(\xi,t)\in  I_{\mathcal M} \times [T,\infty)$. Then 
	\begin{equation} \label{m_hat_identity}
		\hat m(\xi,t,k) \coloneqq 
		I + \frac{1}{2\pi  i}\int_{\hat\Gamma} \hat\mu (\xi,t,k) \hat w (\xi,t,k) \, \frac{\dx z}{z-k}, \qquad (\xi,t)\in   I_{\mathcal M} \times [T,\infty).
	\end{equation}	
	Furthermore, as $t\to\infty$, 
	\begin{equation*}
	\lim^\angle_{k\to\infty} k(\hat{m}(\xi,t,k) - I) = \frac{t^{-1}}{288}
		\begin{pmatrix}
			12 i C_{k_0}^2 - \frac{7 i}{k_0-E_1}  &   
			12 C_{k_0}^2  - \frac{24 C_{k_0}}{\sqrt{k_0-E_1}}  +\frac{7}{k_0-E_1}   \\
			12 C_{k_0}^2  + \frac{24 C_{k_0}}{\sqrt{k_0-E_1}}  +\frac{7}{k_0-E_1}  & - 12 i C_{k_0}^2 + \frac{7 i}{k_0-E_1}
		\end{pmatrix} 
		+ \bigO(t^{-2})
	\end{equation*}
	uniformly with respect to $\xi\in I_{\mathcal M}$, where $C_{k_0}$ is given by \eqref{C_k0}. 
\end{proposition} 
\begin{proof}
Standard arguments imply that the unique solution of the RH problem \eqref{RHP_m_hat} is given by \eqref{m_hat_identity} and that
	\begin{equation} \label{m_hat_integral}
		\lim^\angle_{k\to\infty} k(\hat m(\xi,t,k) - I) =  -\frac{1}{2\pi  i}\int_{\hat\Gamma}  \hat\mu (\xi,t,k) \hat w (\xi,t,k) \, \dx k. 
	\end{equation}
	It remains to determine the asymptotics of the integral in \eqref{m_hat_integral}.
	Let us first consider the contribution from the subcontour $\Gamma' \coloneqq \hat \Gamma\setminus \partial D_\epsilon(k_0)$.
	From the Cauchy-Schwartz inequality in conjunction with \eqref{lem_w_hat_1}, \eqref{lem_w_hat_2}, \eqref{lem_w_hat_4}, \eqref{lem_w_hat_7}, and \eqref{mu-id} it follows that 
	\begin{align*}
		\int_{\Gamma'}  \hat\mu (\xi,t,k) \hat w (\xi,t,k) \, \dx k &= 
		\int_{\Gamma'}  \hat w (\xi,t,k) \, \dx k  + \int_{\Gamma'}  (\hat\mu (\xi,t,k)-I) \hat w (\xi,t,k) \, \dx k \\
		&= \bigO(\|  \hat w \|_{L^1(\Gamma')}) + \bigO(\|  \hat \mu - I \|_{L^2(\Gamma')}  \|  \hat w \|_{L^2(\Gamma')}) \\
		&= \bigO(t^{-2}) \quad \text{as} \quad t\to\infty
	\end{align*}
	uniformly for $\xi\in  I_{\mathcal M}$. 
	Finally, the contribution from $\partial D_{\varepsilon}(k_0)$ with clockwise orientation is 
	\begin{align*}
		-&\frac{1}{2\pi  i}\int_{\partial D_{\varepsilon}(k_0)}  \hat w (\xi,t,k) \, \dx k   -\frac{1}{2\pi  i}\int_{\partial D_{\varepsilon}(k_0)}  (\hat\mu (\xi,t,k)-I) \hat w (\xi,t,k) \, \dx k \\
		&= -\frac{1}{2\pi  i}\int_{\partial D_{\varepsilon}(k_0)}  \Big(m^{k_0}(\xi,t,k) \big(m^{[E_1,k_0]}(\xi,t,k)\big)^{-1} - I\Big)   \, \dx k \\
		&\quad + \bigO\big( \| \hat\mu (\xi,t,\cdot)-I \|_{L^2(\partial D_{\varepsilon}(k_0))} \, \| \hat w (\xi,t,\cdot) \|_{L^2(\partial D_{\varepsilon}(k_0))}    \big),
	\end{align*}
	where the error term is $\bigO(t^{-2})$ as $t\to\infty$ uniformly with respect to $\xi\in I_{\mathcal M}$ by \eqref{lem_w_hat_7} and \eqref{mu-id}. 
	Applying  \eqref{lem_m^k_0_ii_2} finishes the proof.
\end{proof}


\subsubsection{Proof of Theorem \ref{thm_sec_M}} 
Recall that $G^{(3)}\to I$ exponentially fast as $k\to\infty$ uniformly with respect to $(\xi,t)\in  I_{\mathcal M} \times [T,\infty)$. Therefore the asymptotics of $g$, $\mathcal{D}$, $m^{[E_1,k_0]}$, and $\hat m$ as $k\to\infty$ established in \eqref{g_ktoinfty}, Lemma \ref{lem_D}\eqref{lem_D_ii}, Lemma \ref{lem_minfty}\eqref{lem_minfty_ii}, and Proposition \ref{prop_mhat}, in conjunction with the notation in the proof of Lemma \ref{lem_m^k_0}\eqref{lem_m^k_0_ii} (see \eqref{res_A_1_lim} and \eqref{res_A_1}), respectively, yield that
	\begin{align*}
		u(x,t) 
		&= 2 i  \, \e^{2 i t g_\infty (\xi)} \, \mathcal{D}^{-2}_\infty(\xi) \, 	\lim^\angle_{k\to\infty} k \big(m^{[E_1,k_0]} + \hat m\big)_{1 2} \\
		&= - \e^{2 i t g_\infty (\xi)} \mathcal{D}^{-2}_\infty(\xi)  \bigg[\frac{4(\alpha + \beta)-\xi}{6} -  \frac{2 i  \res_{k = k_0} \big(A_1(\xi,k)\big)_{1 2}}{t} \bigg]
		+  \bigO (t^{-2})
	\end{align*}
	as $t\to\infty$, the error term being uniform with respect to $\xi\in I_{\mathcal M}$. This shows \eqref{lim_km_12} and completes the proof of Theorem \ref{thm_sec_M}.

\begin{remark} \label{rem_asymp_to_all_orders}
The above proof can easily be generalized to yield the long-time behavior of $u(x,t)$ in the sector $\mathcal{M}$ to \emph{any} finite order $N$ provided that the reflection coefficient $r$ has sufficient regularity and decay away from the branch points $E_1$ and $E_2$, which, by Theorem \ref{thm_direct}, is the case whenever the initial data has sufficient regularity and decay. In this case, the analytic approximation can be refined accordingly, and Lemma \ref{lem_sing_int_near_boundary} permits us to recursively determine all coefficients in the expansion of $\mathcal{D}$ at $k_0$ up to the desired order, which in turn yields the coefficients of the terms of order $t^{-2}, t^{-3},\dots,t^{-N}$ in the expansion \eqref{lim_km_12} of $u|_{\mathcal{M}}(x,t)$. 
\end{remark}

\section{Proof of Theorem \ref{thm_sec_L}}\label{leftsec}
\noindent
In the previous section, we presented the derivation of the asymptotics in the middle sector $\mathcal{M}$ in great detail. Our treatments of the left and right sectors will be much briefer. In fact, the proof of the large $t$ asymptotics \eqref{u_asymp_L} of $u$ in the left sector $\mathcal L$ is almost identical to the proof of \cite[Theorem 1]{Fromm19}, where a related initial-boundary value problem was considered and the same asymptotic formula was obtained. With the help of Remark \ref{rem_r2a}, the proof in \cite{Fromm19} can easily be adapted to our situation.  As for the proof of the asymptotic formula \eqref{u_asymp_Lx} valid as $x \to -\infty$, the proof proceeds along the same lines and the minor modifications that are necessary were explained in \cite{DZ1993} for the mKdV equation. 

It remains to prove the identity in \eqref{lim_D^{-2}_infty}. By the definition of $D_\infty(\xi)$ in \eqref{leftsectorquantities}, we have
\begin{align}\label{Dinftyxi}
D_\infty(\xi)  = \e^{-\frac{1}{2\pi  i}\big\{\big(\int_{-\infty}^{E_1} + \int_{E_2}^{k_0}\big) \frac{\log(1 - |r(s)|^2)}{X(s)}  \, \dx s + \int_{E_1}^{E_2} \frac{ i\arg(r(s))}{X_+(s)}  \, \dx s \big\}}.
\end{align}
Due to the decay of $r$, cf.~Theorem \ref{thm_direct}, it holds that  
\begin{equation*}
	\lim_{\xi\to -\infty} D_\infty(\xi) = \e^{-\frac{1}{2\pi  i}J}
\end{equation*}  
where the real constant $J$  is given by
\begin{equation}\label{Jdef}
	J \coloneqq \bigg(\int_{-\infty}^{E_1} + \int_{E_2}^{+\infty}\bigg) \frac{\log(1 - |r(k)|^2)}{X(k)} \,\dx k 
	+ \int_{E_1}^{E_2} \frac{ i \arg(r(k))}{X_+(k)} \,\dx k.
\end{equation}
The following lemma implies that $\lim_{\xi\to-\infty} D^{-2}_\infty(\xi) = -1$.

\begin{lemma} \label{lem_D_on_x-Axis}
	$J \in \pi^2(1 + 2\mathbb{Z})$.
\end{lemma}
\begin{proof}
Recalling the definition of $r$ in \eqref{def_r} and the fact that $\det s(k) = |a(k)|^2 - |b(k)|^2 = 1$ for $k \in \R\setminus[E_1,E_2]$, we infer that
	\begin{equation}
		\begin{cases} \label{lem_D_on_x-Axis_1}
			\log(1 - |r(k)|^2) 
			= \log \frac{1}{|a(k)|^2} 	
			= -2\re(\log a(k)),  &\quad k\in\R\setminus[E_1,E_2],\\
			 i \arg(r(k)) 
			\in \overbar{\log a(k)} - \log a(k) + \pi i + 2\pi i\Z, &\quad k \in (E_1, E_2).
		\end{cases}
	\end{equation}
	By employing \eqref{lem_D_on_x-Axis_1} and the fact that $X(k) \in \R$ for $k \in \R \setminus (E_1, E_2)$ and 
	$X_+(k) \in  i\R$ for $k \in (E_1, E_2)$, a direct computation confirms that
	\begin{align*}
		J \in -2 \re  \int_\R \frac{\log a(k)}{X_+(k)} \, \dx k + \pi i(1 + 2\Z) \int_{E_1}^{E_2} \frac{\dx k}{X_+(k)}. 
	\end{align*}
The function $a(k)$ is analytic in $\C_+$ and satisfies $a(k) = 1 + \bigO(k^{-1})$ as $\overbar{\C_+}\ni k \to \infty$ by Proposition \ref{cor_prop_mu}. Thus, deforming the contour $\R$ to infinity in the upper half-plane, we deduce that
\begin{equation*}
	\int_\R \frac{\log a(k)}{X_+(k)} \,\dx k = 0.
\end{equation*}
A direct computation shows that
\begin{equation*}
 \int_{E_1}^{E_2} \frac{\dx k}{X_+(k)} = \int_{E_1}^{E_2} \frac{\dx k}{ i \sqrt{(k-E_1)(E_2-k)}} = -\pi  i.
\end{equation*}
Consequently, $J \in \pi  i (1 + 2 \Z) (-\pi i) = \pi^2(1 + 2\Z)$.
\end{proof}

\section{Proof of Theorem \ref{thm_sec_R}}\label{rightsec}
\noindent
Just like in the left and middle sectors, the asymptotics for $u$ in the right sector $\mathcal{R}$ follows from a steepest descent analysis of the RH problem \eqref{RHm} for $m$. However, the proof in the right sector is much easier than in the left and middle sectors. Indeed, if $\xi \in I_{\mathcal{R}}$ then there is a single critical point at $k_0 = -\xi/4$ which lies to the left of $E_1$. This means that the nonzero boundary conditions at $x = -\infty$ play no role in the analysis; in particular, there is no need for a $g$-function or a global parametrix. Just like in the analysis of solutions of the NLS with decay at $x = \pm \infty$ \cite{DIZ1993}, the leading term in the asymptotics is of $\bigO(t^{-1/2})$ and stems from the local parametrix at $k_0$, which is constructed in terms of parabolic cylinder functions \cite{I1982}. We refer to \cite[Sec.\@ 6]{LQ21} for further details of the proof, where an asymptotically equivalent scenario is treated.  
\proofend

\section{Proof of Theorem \ref{thm_ibvp}}\label{ibvpsec}
\noindent
Suppose $(\alpha, \omega, c) \in (0,\infty) \times \R \times \C$ belongs to the family \eqref{familyc} and let $\beta \coloneqq c/(2i\alpha) > 0$. Suppose $u_0$ satisfies the assumptions of Theorem \ref{thm_direct} corresponding to $\alpha$ and $\beta$ with $N_1=8$ and $N_2=4$, and let $u$ be the solution of the Cauchy problem \eqref{dNLS_IVP} with initial data $u_0$ constructed in Theorem \ref{thm_inverse}. 

It follows from \eqref{familyc} that $c = i \alpha \sqrt{-2\alpha^2-\omega}$ and hence $\omega = \frac{c^2}{\alpha^2} - 2\alpha^2$. Since $\beta = c/(2i\alpha)$, we infer that $\omega = - 4\beta^2 - 2 \alpha^2$ in consistency with the definition of $\omega$ in the rest of the paper. The condition $\omega < -3\alpha^2$ in \eqref{familyc} then implies that $4\beta^2  > \alpha^2$ and hence $4 \beta - 2\alpha > 0$, which shows that the $t$-axis is contained in the left sector $\mathcal{L}$ (for all sufficiently small $\delta >0$). Therefore the asymptotic formulas \eqref{u_asymp_L} and \eqref{ux_asymp_L} evaluated at $x = 0$ show that, as $t \to \infty$,
\begin{align}\label{uuxontaxis}
  u(0,t) = -D^{-2}_\infty(0)\, \alpha \e^{i\omega t}+ \bigO(t^{-1/2}), \qquad u_x(0,t) = -D^{-2}_\infty(0)\, 2i\beta \alpha \e^{i\omega t}+ \bigO(t^{-1/2}).
\end{align}
It follows from \eqref{Dinftyxi} that $|D^2_\infty(\xi)| = 1$ for all $\xi$. In particular, $|D^2_\infty(0)| = 1$, so the phase invariance of NLS implies that $\tilde{u}(x,t)\coloneqq -D^2_\infty(0) u(x,t)$ also satisfies equation \eqref{dNLS}. Since $c = 2i\alpha \beta$, equation \eqref{uuxontaxis} implies that
\begin{align}\label{}
  \tilde{u}(0,t) = \alpha \e^{i\omega t}+ \bigO(t^{-1/2}), \qquad \tilde{u}_x(0,t) = c \e^{i\omega t}+ \bigO(t^{-1/2}).
\end{align}
On the other hand, we infer from Theorem \ref{thm_sec_R} that $u(x,t) \to 0$ as $x \to +\infty$ for each fixed $t > 0$. For $t = 0$, we have $u(x,0) \to 0$ as $x \to +\infty$ by the decay assumptions on $u_0$. We conclude that $\tilde{u}(x,t) \to 0$ as $x \to +\infty$ for each $t\geq 0$. This completes the proof.
\proofend

\appendix
	
\section{Endpoint behavior of a Cauchy integral}\label{appB}
\noindent
\begin{lemma} \label{lem_sing_int_near_boundary}
Let $N \geq 0$ be an integer. Let $[a,b]\in \R$ be a closed finite interval and suppose $g_0 \in \mathcal C^{N} ([a,b],\C)$ is such that the $N$\!th derivative of $g_0$ is H\"older continuous on $[a,b]$. Then the function
\begin{align}\label{f0def}
  f_0(z)\coloneqq \int^b_a \frac{g_0(s)}{\sqrt{b-s}}  \frac{\dx s}{s-z}, \qquad z \in \C \setminus (-\infty, b],
\end{align} 
satisfies
\begin{align}\label{f0expansionN} 
f_0(z) = - \pi \sum_{n=0}^{N} \frac{g_0^{(n)}(b)}{n!} (z-b)^{n-\frac{1}{2}} + \sum_{n=0}^{N-1} f_{0,n} (z-b)^n + o((z-b)^{N-\frac{1}{2}}), \qquad z\to b,
\end{align} 
uniformly for $\arg (z-b) \in (-\pi,\pi)$, where the coefficients $f_{0,n}\in \C$ are given by 
\begin{align} \label{f0nexpression}
  f_{0,n} = \sum^n_{j=0} \frac{g_0^{(j)}(b)}{j!} h_{n-j} + C_n, \qquad n = 0, 1, \dots, N-1,
\end{align} 
with $C_n \in \C$ and $h_n \in \R$ defined by
	\begin{equation}\label{Cnhndef} 
		C_n = \int^b_a \frac{g_{n+1}(s)}{\sqrt{b-s}} \, \dx s, \quad 
		h_n = \frac{2(-1)^n(b-a)^{-n-\frac{1}{2}}}{2n+1}, \qquad n = 0, 1, \dots, N-1.
	\end{equation} 
	The functions $g_n\in \mathcal C^{N-n} ([a,b],\C)$ in \eqref{Cnhndef} are defined by
	\begin{equation} \label{lem_A_def_g_n}
		g_n(s) = \frac{g_{n-1}(s) - g_{n-1}(b)}{s-b} = \frac{g_0(s) - \sum^{n-1}_{j=0} \frac{g^{(j)}(b)}{j!}}{(s-b)^n}, \qquad n = 1, \dots, N.
	\end{equation} 
\end{lemma}
\begin{proof}
Define $h(z)$ by
\begin{equation*}
	h(z) = \int_{a}^{b}\frac{1}{\sqrt{b - s}} \frac{\dx s}{s-z}, \qquad z \in \C \setminus (-\infty,b].
\end{equation*}
For $s \in (a, b)$, we have
\begin{equation*}
	\frac{\partial}{\partial s} \frac{\pi-2\arctan\left(\frac{\sqrt{z-b }}{\sqrt{b -s}}\right)}{\sqrt{z-b }}
	= \frac{1}{\sqrt{b -s}(s-z)}.
\end{equation*}
Hence
\begin{equation*}
	h(z) = \frac{2\arctan\left(\frac{\sqrt{z-b }}{\sqrt{b - a}}\right) - \pi}{\sqrt{z-b }}, \qquad z \in \C \setminus (-\infty,b].
\end{equation*}
Using that
\begin{equation*}
\arctan z = \sum_{k = 0}^\infty \frac{(-1)^k z^{2k+1}}{2k+1}, \qquad |z| < 1,	
\end{equation*}
we conclude that
\begin{align}\label{hexpansion}
h(z) \sim -\frac{\pi}{\sqrt{z-b}} + \sum_{n=0}^\infty h_n (z-b)^n, \qquad z \to b,
\end{align}
where $\{h_n\}_0^\infty$ are the coefficients defined in \eqref{Cnhndef}.

For $n = 0, 1, \dots, N$, let
\begin{align*}
  f_n(z) \coloneqq \int_{a}^{b}\frac{g_n(s)}{\sqrt{b -s}} \frac{\dx s}{s-z}
\end{align*}
with $\{g_n\}^N_1$ being defined in \eqref{lem_A_def_g_n}.
Since
\begin{equation*}
	\frac{g_n(s)}{s-z} = \frac{g_{n+1}(s)}{s-z}(z-b)
	+ \frac{g_n(b)}{s-z} + g_{n+1}(s), \qquad n = 0,1, \dots, N-1,
\end{equation*}
we obtain
\begin{align}\label{fnrecursive}
f_n(z) = f_{n+1}(z)(z-b) + g_n(b) h(z) + C_n, \qquad n = 0,1, \dots, N-1.
\end{align}
Recursive use of \eqref{fnrecursive} gives
\begin{align}\nonumber
f_0(z) & = f_1(z) (z-b) + g_0(b)h(z) + C_0
	\\\nonumber
& =  [f_2(z) (z-b) + g_1(b)h(z) + C_1] (z-b) + g_0(b)h(z) + C_0
	\\ \label{f0recursive}
 &= \cdots =  f_{N}(z) (z-b)^{N} + \sum_{n=0}^{N-1} (g_n(b)h(z) + C_n)(z-b)^n.
\end{align}
Since $g_{N}$ is H\"older continuous on $[a,b]$, \cite[Eq. (29.5)]{M1992} implies that
\begin{equation*}
	f_{N}(z) = -\frac{\pi}{\sqrt{z-b}}g_{N}(b) + o\big((z-b)^{-1/2}\big).
\end{equation*}
Hence, noting that
\begin{equation*}
g_n(b) = \frac{g_0^{(n)}(b)}{n!}, \qquad n = 0, 1, \dots, N,	
\end{equation*}
the expansion \eqref{f0expansionN} follows by expanding \eqref{f0recursive} as $z \to b$.
\end{proof}

\medskip
\noindent
\textbf{Acknowledgments.} 
Support is acknowledged from the \emph{European Research Council, Grant Agreement No.~682537}, the \emph{Ruth and Nils-Erik Stenb\"ack Foundation}, the \emph{Swedish Research Council, Grant No. 2015-05430}, and the \emph{Austrian Science Fund FWF, Erwin Schr\"odinger fellowship J~4339-N32}.



\begin{thebibliography}{10}	
\bibitem{A2011}
M.J.\@ Ablowitz, {\it Nonlinear dispersive waves. Asymptotic analysis and solitons.} Cambridge Texts in Applied Mathematics, Cambridge University Press, New York, 2011. 

\bibitem{BT2013}
M.\@ Bertola and A.\@ Tovbis, Universality for the focusing nonlinear Schr\"odinger equation at the gradient catastrophe point: rational breathers and poles of the tritronquée solution to Painlev\'e I, {\it Comm.\@ Pure Appl.\@ Math.}~{\bf 66}, 678--752 (2013). 

\bibitem{BB2019}
D.\@ Bilman and R.\@ Buckingham, Large-order asymptotics for multiple-pole solitons of the focusing nonlinear Schr\"odinger equation, {\it J.\@ Nonlinear Sci.}~{\bf 29}, 2185--2229 (2019). 

\bibitem{BM2019}
D.\@ Bilman and P.D.\@ Miller, A robust inverse scattering transform for the focusing nonlinear Schr\"odinger equation, {\it Comm.\@ Pure Appl.\@ Math.}~{\bf 72}, 1722--1805 (2019).

\bibitem{B2018}
G.\@ Biondini, Riemann problems and dispersive shocks in self-focusing media, 
{\it Phys.\@ Rev.\@ E}~{\bf 98}, 052220 (2018).

\bibitem{BLM2021}
G.\@ Biondini, J.\@ Lottes, and D.\@ Mantzavinos, Inverse scattering transform for the focusing nonlinear Schr\"odinger equation with counterpropagating flows, {\it Stud.\@ Appl.\@ Math.}~{\bf 146}, 371--439 (2021).

\bibitem{BM2017}
G.\@ Biondini and D.\@ Mantzavinos, Long-time asymptotics for the focusing nonlinear Schr\"odinger equation with nonzero boundary conditions at infinity and asymptotic stage of modulational instability, {\it Comm.\@ Pure Appl.\@ Math.}~{\bf 70}, 2300--2365, (2017).

\bibitem{BP1982}
M.\@ Boiti and F.\@ Pempinelli, The spectral transform for the NLS equation with left-right asymmetric boundary conditions, {\it Nuovo Cimento B}~{\bf 69}, 213--227, (1982).

\bibitem{BF2008}
J.L.\@ Bona and A.S.\@ Fokas, Initial-boundary-value problems for linear and integrable nonlinear dispersive partial differential equations, 
{\it Nonlinearity}~{\bf 21}, T195--T203 (2008). 

\bibitem{BSZ2018}
J.L.\@ Bona, S.M.\@ Sun, and B.Y.\@  Zhang, Nonhomogeneous boundary-value problems for one-dimensional nonlinear Schr\"odinger equations, {\it J.\@ Math.\@ Pures Appl.}~{\bf 109}, 1--66 (2018). 

\bibitem{BIK2009}
A.\@ Boutet de Monvel, A.\@ Its and V.\@ Kotlyarov, Long-Time Asymptotics for the Focusing NLS Equation with Time-Periodic Boundary Condition on the Half-Line, {\it Commun.\@ Math.\@ Phys.}~{\bf 290}, 479--522 (2009).

\bibitem{BKS2009}
A.\@ Boutet de Monvel, V.\@ Kotlyarov, and D.\@ Shepelsky, Decaying long-time asymptotics for the focusing NLS equation with periodic boundary condition, {\it Int.\@ Math.\@ Res.\@ Not.}~{\bf 2009}(3), 547--577 (2009).

\bibitem{BLS2021}
A.\@ Boutet de Monvel, J.\@ Lenells, and D.\@ Shepelsky, The focusing NLS equation with step-like oscillating background: scenarios of long-time asymptotics, {\it Comm.\@ Math.\@ Phys.}~{\bf 383}, 893--952, (2021).

\bibitem{BKS2011}
A.\@ Boutet de Monvel, V.P.\@ Kotlyarov, and D.\@ Shepelsky,
Focusing NLS equation: long-time dynamics of step-like initial data,
{\it Int.\@ Math.\@ Res.\@ Not.}~{\bf 2011}(7), 1613--1653 (2011).

\bibitem{BTVZ2007}
R.\@ Buckingham, A.\@ Tovbis, S.\@ Venakides, and X.\@ Zhou, The semiclassical focusing nonlinear Schr\"odinger equation, {\it Proc.\@ Sympos.\@ Appl.\@ Math.}~{\bf 65}, AMS Short Course Lecture Notes, Amer.\@ Math.\@ Soc., Providence, RI, 2007. 

\bibitem{BV2007}
R.\@ Buckingham and S.\@ Venakides, Long-time asymptotics of the nonlinear Schr{\"o}dinger equation shock problem, {\it Comm.\@ Pure Appl.\@ Math.}~{\bf 60}, 1349--1414 (2007).

\bibitem{DIZ1993}
P.\@ Deift, A.R.\@ Its, and X.\@ Zhou, 
Long-time asymptotics for integrable nonlinear wave equations. In: Fokas, A.S., Zakharov, V.E.\@ (eds) {\it Important Developments in Soliton Theory.}, 181--204, Springer Ser.\@ Nonlinear Dynam., Springer, Berlin, 1993. 

\bibitem{DZ1993}
P.\@ Deift and X.\@ Zhou,
A steepest descent method for oscillatory Riemann-Hilbert problems. Asymptotics for the MKdV equation, 
{\it Ann.\@ of Math.}~{\bf 137}(2), 295--368 (1993).

\bibitem{DZ2003}
P.\@ Deift and X.\@ Zhou, Long-time asymptotics for solutions of the NLS equation with initial data in a weighted Sobolev space, {\it Comm.\@ Pure Appl.\@ Math.}~{\bf 56}, 1029--1077 (2003).

\bibitem{DPMV2013}
F.\@ Demontis, B.\@ Prinari, C.\@ van der Mee, and F.\@ Vitale, The inverse scattering transform for the defocusing nonlinear Schr{\"o}dinger equations with nonzero boundary conditions, {\it Stud.\@ Appl.\@ Math.}~{\bf 131}, 1--40 (2013).

\bibitem{FT1986}
L.D.\@ Faddeev and L.A.\@ Takhtajan, {\it Hamiltonian methods in the theory of solitons}, Classics in Mathematics, Springer, Berlin, 2007.

\bibitem{F1997}
A.S.\@ Fokas, 
A unified transform method for solving linear and certain nonlinear PDEs, 
{\it Proc.~Roy.~Soc.~Lond.~A}~{\bf 453}, 1411--1443 (1997).

\bibitem{F2005}
A.S.\@ Fokas, A generalised Dirichlet to Neumann map for certain nonlinear evolution PDEs, {\it Comm.\@ Pure Appl.\@ Math.}~{\bf LVIII}, 639--670 (2005).

\bibitem{FHM2017}
A.S.\@ Fokas, A.A.\@ Himonas, and D.\@ Mantzavinos, The nonlinear Schr\"odinger equation on the half-line, {\it Trans.\@ Amer.\@ Math.\@ Soc.}~{\bf 369}, 681--709 (2017).

\bibitem{FIS2005}
A.S.\@ Fokas, A.R.\@ Its, and L.Y.\@ Sung, The nonlinear Schr\"odinger equation on the half-line, 
{\it Nonlinearity}~{\bf 18}, 1771--1822  (2005).

\bibitem{Fromm19}
S.\@ Fromm,
Construction of solutions of the defocusing nonlinear Schr\"odinger equation with asymptotically time-periodic boundary values, 
{\it Stud.\@ Appl.\@ Math.}~{\bf 143}(4), 404--448 (2019). 

\bibitem{GK2012}
J.\@ Garnier and K.\@ Kalimeris, Inverse scattering perturbation theory for the nonlinear Schr\"odinger equation with non-vanishing background, {\it J.\@ Phys.\@ A}~{\bf  45}, 035202, 13 pp.\@ (2012). 

\bibitem{HMY2019}
A.A.\@ Himonas, D.\@ Mantzavinos, and F.\@ Yan, The nonlinear Schr\"odinger equation on the half-line with Neumann boundary conditions, {\it Appl.\@ Num.\@ Math.}~{\bf  141}, 2--18 (2019).

\bibitem{H2005}
J.\@ Holmer, The initial-boundary-value problem for the 1D nonlinear Schr\"odinger equation on the half-line, {\it Diff.\@ Int.\@ Eq.}~{\bf 18}, 647--668 (2005). 

\bibitem{I1982}
A.R.\@ Its, Asymptotic behavior of the solutions to the nonlinear Schr\"odinger equation, and isomonodromic deformations of systems of linear differential equations, {\it Soviet Math.\@ Dokl.}~{\bf 24}, 452--456 (1982).

\bibitem{IU1986}
A.R.\@ Its and A.F.\@ Ustinov, Time asymptotics of the solution of the Cauchy problem for the nonlinear Schr\"odinger equation with boundary conditions of finite density type, {\it Dokl.\@ Akad.\@ Nauk SSSR}~{\bf 291}, 91--95 (1986).

\bibitem{IU1991}
A.R.\@ Its and A.F.\@ Ustinov, Formulation of the scattering theory for the NLS equation with boundary conditions of finite density type in a soliton-free sector, {\it J.\@ Soviet Math.}~{\bf 54}, 900--905 (1991).

\bibitem{J2015}
R.\@ Jenkins, Regularization of a sharp shock by the defocusing nonlinear Schr\"{o}dinger equation, {\it Nonlinearity}~{\bf 28}, 2131--2180 (2015).

\bibitem{K1996}
S.\@ Kamvissis, Long time behavior for the focusing nonlinear Schroedinger equation with real spectral singularities, {\it Comm.\@ Math.\@ Phys.}~{\bf 180}, 325--341 (1996). 

\bibitem{KMM2003}
S.\@ Kamvissis, K.T.R.\@ McLaughlin, and P.D.\@ Miller, {\it Semiclassical soliton ensembles for the focusing nonlinear Schr\"odinger equation}, Annals of Mathematics Studies, 154, Princeton University Press, Princeton, NJ, 2003.

\bibitem{KI1978}
T.\@ Kawata and H.\@ Inoue, Inverse scattering method for the nonlinear evolution equations under nonvanishing conditions, {\it J.\@ Phys.\@ Soc.\@ Japan}~{\bf 44}, 1722--1729 (1978).

\bibitem{Len15}
J.\@ Lenells, 
Admissible boundary values for the defocusing nonlinear Schr\"odinger equation with asymptotically $t$-periodic data, 
{\it J.~Diff.~Eq.}~{\bf 259}, 5617--5639 (2015).

\bibitem{Lenells16}
J.\@ Lenells,
Nonlinear Fourier transforms and the mKdV equation in the quarter plane, 
{\it Stud.\@ Appl.\@ Math.}~{\bf 136}(1), 3--63 (2016). 

\bibitem{LFunifiedII}
J.\@ Lenells and A.S.\@ Fokas, The unified method: II.\@ NLS on the half-line with t-periodic boundary conditions, {\it J.\@ Phys.\@ A}~{\bf 45}, 195202 (2012). 

\bibitem{LFtperiodicI} 
J.\@ Lenells and A.S.\@ Fokas, The nonlinear Schr\"odinger equation with t-periodic data: I.\@ Exact results, {\it Proc.\@ Roy.\@ Soc.\@ A}~{\bf 471}, 20140925 (2015).

\bibitem{LFtperiodicII} 
J.\@ Lenells and A.S.\@ Fokas, The nonlinear Schr\"odinger equation with t-periodic data: II.\@ Perturbative results, {\it Proc.\@ Roy.\@ Soc.\@ A}~{\bf 471}, 20140926 (2015). 

\bibitem{LQ20}
J.\@ Lenells and R.\@ Quirchmayr, 
On the spectral problem associated with the time-periodic nonlinear Schr\"odinger equation, 
{\it Math.\@ Ann.}~{\bf 377}, 1193--1264 (2020).

\bibitem{LQ21}
J.\@ Lenells and R.\@ Quirchmayr, 
Construction of solutions and asymptotics for the defocusing NLS with periodic boundary data, 
{\it J.\@ Differ.\@ Equ.}~{\bf 304},  348--374 (2021).

\bibitem{M1992}
N.I.\@ Muskhelishvili, {\it Singular integral equations. Boundary problems of function theory and their application to mathematical physics.} Dover Publications, New York, 1992.

\bibitem{TVZ2004}
A.\@ Tovbis, S.\@ Venakides, and X.\@ Zhou, On semiclassical (zero dispersion limit) solutions of the focusing nonlinear Schr\"odinger equation, {\it Comm.\@ Pure Appl.\@ Math.}~{\bf 57}, 877--985 (2004). 

\bibitem{TVZ2006}
A.\@ Tovbis, S.\@ Venakides, and X.\@ Zhou, On the long-time limit of semiclassical (zero dispersion limit) solutions of the focusing nonlinear Schr\"odinger equation: pure radiation case, {\it Comm.\@ Pure Appl.\@ Math.}~{\bf 59}, 1379--1432 (2006).

\bibitem{ZS1973}
V.E.\@ Zakharov and A.B.\@ Shabat, Interaction between solitons in a stable medium, {\it Sov.\@ Phys.\@ JETP}~{\bf 37}(5), 823 (1973).

\bibitem{Z1989}
X.\@ Zhou, The Riemann-Hilbert problem and inverse scattering, {\it SIAM J.\@ Math.\@ Anal.}~{\bf 20}, 966--986 (1989). 
\end{thebibliography}
\end{document}